\documentclass[11pt,reqno]{amsart}
\usepackage{amsmath}
\usepackage{amsthm}
\usepackage{amssymb}
\usepackage{amscd}
\usepackage{amsfonts}
\usepackage{amsbsy}
\usepackage{epsfig}
\usepackage{psfrag}
\usepackage{amssymb,amsmath,amsthm,mathrsfs}
\usepackage{graphicx}
\usepackage{float}
\usepackage{color, colortbl}
\usepackage{enumerate}
\usepackage{upref}
\usepackage[bookmarks=false]{hyperref}
\usepackage{epstopdf}
\usepackage[T2A,T1]{fontenc}
\DeclareSymbolFont{cyrillic}{T2A}{cmr}{m}{n}
\DeclareMathSymbol{\D}{\mathalpha}{cyrillic}{196}

\theoremstyle{plain}
\newtheorem{theorem}{Theorem}[section]
\newtheorem{lemma}{Lemma}[section]
\newtheorem{corollaryP}[theorem]{Corollary}
\newtheorem{propositionP}[theorem]{Proposition}
\newtheorem{proposition}[lemma]{Proposition}

\theoremstyle{definition}
\newtheorem{definition}[lemma]{Definition}

\newtheorem*{condition}{Condition}

\theoremstyle{remark}
\newtheorem{remark}[lemma]{Remark}

\usepackage{enumitem}
\makeatletter
\def\namedlabel#1#2{\begingroup
   #2%
 \def\@currentlabel{#2}%
   \phantomsection\label{#1}\endgroup
}

\def\S{Section}
\def\ie{{\em i.e.,} }

\newfont\bbf{msbm10 at 12pt}
\def\eps{\varepsilon}
\def\R{{\mathbb R}}

\def\N{{\mathbb N}}

\def\B{{\mathcal B}}
\def\I{{\mathcal I}}
\def\P{{\mathcal P}}

\def\sm{\setminus}

\def\cv{\ensuremath{\text {Cor}}}
\newcommand{\dif}{\mathrm{d}}

\DeclareMathOperator*{\essinf}{ess\;inf}

\def\R{\ensuremath{\mathbb R}}
\def\N{\ensuremath{\mathbb N}}

\def\I{\ensuremath{{\bf 1}}}
\def\e{\ensuremath{\text e}}

\def\S{\ensuremath{\mathcal S}}
\def\RR{\ensuremath{\mathcal R}}

\def\B{\ensuremath{\mathcal B}}
\def\l{\ensuremath{\text{Leb}}}

\def\M{\ensuremath{\mathcal M}}

\def\P{\ensuremath{\mathcal P}}
\def\p{\ensuremath{\mathbb P}}

\def\RR{\ensuremath{\mathcal R}}

\def\II{\ensuremath{\mathscr I}}
\def\aa{\ensuremath{\mathscr A}}

\def\X{\mathcal{X}}
\def\ie{{\em i.e.}, }

\def\E{\mathbb E}

\def\cyl{\text{Z}}

\numberwithin{equation}{section}

\def\dist{\ensuremath{\mbox{dist}}}

\def\dist{\mbox{dist}}

\def\le{\leqslant}
\def\ge{\geqslant}

\begin{document}

\title{Convergence of Marked Point Processes of Excesses for Dynamical Systems}

\author[A. C. M. Freitas]{Ana Cristina Moreira Freitas}
\address{Ana Cristina Moreira Freitas\\ Centro de Matem\'{a}tica \&
Faculdade de Economia da Universidade do Porto\\ Rua Dr. Roberto Frias \\
4200-464 Porto\\ Portugal} \email{amoreira@fep.up.pt}

\author[J. M. Freitas]{Jorge Milhazes Freitas}
\address{Jorge Milhazes Freitas\\ Centro de Matem\'{a}tica \& Faculdade de Ci\^encias da Universidade do Porto\\ Rua do
Campo Alegre 687\\ 4169-007 Porto\\ Portugal}
\email{jmfreita@fc.up.pt}
\urladdr{http://www.fc.up.pt/pessoas/jmfreita}

\author[M. Magalh\~aes]{M\'ario Magalh\~aes}
\address{M\'ario Magalh\~aes\\ Centro de Matem\'{a}tica da Universidade do Porto\\ Rua do
Campo Alegre 687\\ 4169-007 Porto\\ Portugal} \email{mdmagalhaes@fc.up.pt}

\thanks{MM was partially supported  by FCT grant SFRH/BPD/89474/2012, which is supported by the program POPH/FSE. ACMF, JMF are partially  supported by FCT (Portugal) projects PTDC/MAT/120346/2010, FAPESP/19805/2014, PTDC/ MAT-CAL/3884/2014, which are funded by national and European structural funds through the programs  FEDER and COMPETE. All authors were partially supported by CMUP (UID/MAT/00144/2013), which is funded by FCT (Portugal) with national (MEC) and European structural funds through the programs FEDER, under the partnership agreement PT2020. We thank Mike Todd for helpful comments and suggestions.    
}

\date{\today}


%37A50  	Relations with probability theory and stochastic processes
%60G70  	Extreme value theory; extremal processes
%37B20  	Notions of recurrence
%60G10  	Stationary processes
%37C25  	Fixed points, periodic points, fixed-point index theory
%--------
%37A25  	Ergodicity, mixing, rates of mixing
%37D25  	Nonuniformly hyperbolic systems (Lyapunov exponents, Pesin theory, etc.)
%37D35  	Thermodynamic formalism, variational principles, equilibrium states
%37C40  	Smooth ergodic theory, invariant measures
%60G55   Point processes
%60G57   Random measures

\begin{abstract}
We consider stochastic processes arising from dynamical systems simply by evaluating an observable function along the orbits of the system and study marked point processes associated to extremal observations of such time series corresponding to exceedances of high thresholds. Each exceedance is marked by a quantity intended to measure the severity of the exceedance. In particular, we consider marked point processes measuring the aggregate damage by adding all the excesses over the threshold that mark each exceedance (AOT) or simply by adding the largest excesses in a cluster of exceedances (POT). We provide conditions to prove the convergence of such marked point processes to a compound Poisson process, for whose multiplicity distribution we give an explicit formula. These conditions are shown to follow from a strong form of decay of correlations of the system. Moreover, we prove that the convergence of the marked point processes for a `nice' first return induced map can be carried to the original system. The systems considered include non-uniformly expanding maps (in one or higher dimensions), maps with intermittent fixed points or non-recurrent critical points. For a general class of examples, the compound Poisson limit process is computed explicitly and, in particular, in the POT case we obtain a generalised Pareto multiplicity distribution. \end{abstract}

\maketitle

\section{Introduction}
\label{sec:introduction}

In the past few years the study of the extremal behaviour of dynamical systems has drawn much attention (see for example: \cite{C01,FFT10,FFT12, HNT12, LFTV12, CC13, LFFF16}). The occurrence of extreme or rare events is often seen as the entrance of an orbit in some small (hence rare) target set in the phase space. These target sets are usually taken either as cylinders or shrinking balls around some determined point $\zeta$ in the phase space and we want to study the elapsed time before hitting such targets. This is obviously related to the recurrence properties of the system and can also be associated to the occurrence of extreme observations (or exceedances of high thresholds) for a given potential so that entrances in the target set mean that the respective observations of the potential achieve very high values. This relationship between hitting times and extreme values was formally established in \cite{FFT10, FFT11}.  In this paper, we will use an extreme value approach rather than a hitting times approach, but bear in mind that these are two sides of the same coin as can be fully appreciated in the aforementioned papers.

One of the motivations for studying such properties is that extreme events are associated with risk assessment and understanding their likelihood is of crucial importance. One way of keeping track of extreme events is through the study of point processes, which keep record of the number of exceedances (entrances in the target sets) observed in a certain normalised time frame. In \cite{FFT10,FFT13, AFV15} these processes were studied in a dynamical setting and called Rare Events Point Processes (REPP). The REPP could be described in a simplified way as follows (see precise definition in Section~\ref{sec:setting}). Let $X_0, X_1,\ldots $ be a stationary stochastic process arising from a dynamical system by observing a given potential along its orbits. Let $u$ be a high threshold, consider the time interval $[0,t)$ and a normalising scale factor $v_u$ that depends on $u$ and which will be made precise below. Let 
$$
N_u(t)=\sum_{j=0}^{\lfloor v_u t\rfloor} \I_{X_j>u},
$$
where $\I_A$ is the indicator function over the set $A$. Note that $N_u(t)$ gives the number of exceedances during the normalised time interval $[0,v_u t)$.

The convergence of REPP is affected significantly by the presence or absence of clustering of exceedances. As shown in \cite{FFT10}, in the absence of clustering, the REPP converge to a Poisson process of intensity $1$, meaning that, in particular, $N_u(t)$ converges in distribution to a Poisson random variable of average $t$. In \cite{FFT13}, under the presence of clustering, the REPP was shown to converge to a compound Poisson process of intensity $0<\theta<1$ and geometric multiplicity distribution with mean $\theta^{-1}$, that can be interpreted as the average cluster size. This parameter $\theta$ is called the Extremal Index (EI). In particular, this means that $N_u(t)$ converges in distribution to a P\'olya--Aeppli distribution. One can think of the compound Poisson process as having two independent components: the Poisson events on the time axis ruled by exponential inter arrival times and the multiplicity of each event (or weight associated to each event) that in the latter case is determined by a geometric distribution.  

The convergence of REPP can be used to obtain relevant information such as the expected time between the occurrence of catastrophic events, the intensity of clustering,  the distribution of the higher order statistics of a finite size sample, which ultimately are crucial for assessing risk. However, in many circumstances such as in actuarial science, or in structural safety, not only the frequency of rare undesirable events is relevant for the evaluation of risk associated to certain phenomenon. In fact, insurance companies and safety regulation agencies are also very much interested in, on the one hand, the severity of high impacts, and on the other hand, in the effects of aggregate damage. This motivates studying other point processes that are not limited to count the number of exceedances but rather quantify somehow the amount of damage by adding the excesses over a certain high threshold:
$$
A_u(t)=\sum_{j=0}^{\lfloor v_u t\rfloor} (X_j-u)_+,
$$
where $(x)_+=\max\{0,x\}$.

In the case there is no clustering $A_u(t)$ gives rise to Excesses Over Threshold (EOT) marked point process. When there is clustering then one has (at least) two natural possibilities to handle the excesses within a cluster: either we are interested in the aggregate damage and in that case we sum all the excesses within a cluster to obtain a Area Over Threshold (AOT) marked point process; or we are interested in the record impact of the highest exceedance and in that case we take the maximum excess within a cluster to obtain a Peak Over Threshold (POT) marked point point process (in this case we need to adjust the definition of $A_u(t)$ but postpone it to Section~\ref{sec:setting}). 

The interest in AOT arises for example in situations where the immediate large observations have an accumulated detrimental effect over a certain structure or a companies' financial situation that ultimately results in a system failure/collapse or bankruptcy. On the other hand the interest in the POT may appear when there  is some sort of recovery mechanism that softens the effect of small exceedances but one is mostly worried with the sensitivity to singular very high impacts.

Several difficulties arise in studying the convergence of such point processes. The most obvious is the fact that instead of expecting a discrete multiplicity distribution (Geometric distribution), as in \cite{FFT13}, here we expect continuous multiplicity distributions of Pareto type. This means that we have to build up on the work done in \cite{FFT13}, adapting the mixing conditions considered there in order to study joint Laplace transforms associated with these processes and ultimately prove their convergence for systems with good mixing rates.

In the classical setting of stationary stochastic processes the convergence of the REPP was proved to be a compound Poisson process in \cite{HHL88,LR88} under assumption $\Delta(u_n)$ (which is very similar to Leadbetter's $D(u_n)$ introduced in \cite{L73}) and assuming the existence of an EI. In the dynamical setting, in \cite{HV09}, the authors prove the convergence of $N_u(t)$ to a P\'olya--Aeppli distribution for cylinder target sets. In \cite{FFT13}, which builds upon \cite{FFT12}, some conditions were devised to prove the existence of an EI when the target sets are balls around repelling periodic points, the authors proved the convergence of the REPP to a compound Poisson process with geometric multiplicity distribution. The conditions proposed in \cite{FFT13} can be checked for systems with sufficient decay of correlations (contrary to $D(u_n)$ or $\Delta(u_n)$) and allow to prove the existence of an EI and compute its value from the expansion rate at the repelling periodic point.

In the classical literature in \cite{L91}, Leadbetter shows the convergence of the EOT, for independent and identically distributed (iid) random variables, and of the POT marked point process to a compound Poisson process with multiplicity distribution given by a generalised Pareto distribution, whose type is determined by the tail of the distribution of $X_0$. The convergence of the latter is obtained for stationary stochastic processes under condition $\Delta(u_n)$ that cannot be verified in a dynamical setting. The result is obtained under the assumption of existence of an EI. In \cite{L99}, the convergence of the AOT under $\Delta(u_n)$ is also addressed but assuming the existence of an unknown limit for the multiplicity distribution.

In the dynamical setting the appearance of clustering was linked to periodicity of the point $\zeta$ playing the role of base of the target sets in \cite{H93, HV09, FFT12,FFT13}. In fact, as proved in \cite{AFV15}, when the target sets are balls around $\zeta$ then we have a dichotomy regarding the convergence of the REPP for systems with a strong form of decay of correlations known as decay of correlations against $L^1$ (see Definition~\ref{def:dc} below): either $\zeta$ is periodic and in that case it converges to a compound Poisson process of intensity $\theta$ and geometric multiplicity distribution or $\zeta$ is not periodic and in that case we have no clustering and convergence to a standard Poisson process. In a very recent paper \cite{AFFR16}, the authors use multiple maxima $\zeta_1, \ldots, \zeta_k$ correlated by belonging to the same orbit to create a fake periodic effect that ultimately creates clustering, in this case, with possibly different multiplicity distributions.

In this paper we give conditions (a long range and short range conditions on the dependence structure of the stochastic processes)  to guarantee the convergence of the EOT, AOT, POT marked point processes, which can also be used to prove the convergence of the REPP. In fact, the result (Theorem~\ref{thm:convergence}) is quite general and can be used to prove the convergence of other marked point processes associated to exceedances by using other possible marks over each exceedance. The result applies both in the presence and absence of clustering. The conditions are devised  to be applied in the dynamical setting (contrary to $\Delta(u_n)$) and to simplify the proof of the existence of an EI. Moreover, from these new conditions we provide a new formula to compute the multiplicity distribution of the limiting compound Poisson process. Furthermore, the conditions can be used in a wide range of scenarios including target sets around multiple maxima as in \cite{AFFR16} or discontinuity points as in \cite{AFV15} or even other more geometrically intricate sets.

Then in Theorem~\ref{thm:check-D-D'} we show that such conditions can be easily verified if the system has for example decay of correlations against $L^1$ observables, which allows to apply the theory to uniformly expanding maps of the interval (such as Rychlik maps) or higher dimensional uniformly expanding systems like the ones studied by Saussol in \cite{S00}.

Motivated by an idea introduced in \cite{BSTV03} and extended in the recent paper \cite{HWZ14}, we prove Theorem~\ref{thm:mpp_ret_orig}, which states that if a system admits a `nice' first return time induced map for which we can prove the convergence of marked point processes associated to the exceedances (such as the EOT, AOT or POT) then the original system shares the same property. This allows the application of our results to maps with intermittent fixed points, like the Manneville-Pommeau maps or Liverani-Saussol-Vaienti maps, or maps with critical points such as Misiurewicz quadratic maps. 

In order to exemplify the application of the main results, prove the convergence of marked point processes and actually compute the limit distributions (using the formula we provide to compute its multiplicity distribution), we considered the case where the targets are balls around a single maximum at $\zeta$ with some natural regularity conditions to obtain a result (Theorem~\ref{thm:EOT-POT}) stating that for a fairly large scope of examples the EOT and POT marked point processes converge to a compound Poisson process with intensity $\theta$ (for which we provide a precise formula) and multiplicity distribution corresponding to a generalised Pareto. Then in Theorem~\ref{thm:AOT} we address the more difficult AOT case for which, under some more restrictive assumptions on the system, we also compute the multiplicity of the limiting compound Poisson distribution.

\section{The setting and statement of results}
\label{sec:setting}

Take a system $(\X,\mathcal B,\p,f)$, where $\X$ is a Riemannian manifold, $\mathcal B$ is the Borel $\sigma$-algebra, $f:\X\to\X$ is a measurable map and $\p$ an $f$-invariant probability measure. Suppose that the time series $X_0, X_1,\ldots$ arises from such a system simply by evaluating a given observable $\varphi:\X\to\R\cup\{\pm\infty\}$ along the orbits of the system, or in other words, the time evolution given by successive iterations by $f$:
\begin{equation}
\label{eq:def-stat-stoch-proc-DS}
X_n=\varphi\circ f^n,\quad \mbox{for each } n\in\N.
\end{equation}
Clearly, $X_0, X_1,\ldots$ defined in this way is not necessarily an independent sequence. However, $f$-invariance of $\p$ guarantees that this stochastic process is stationary.

The most simple point processes keep record of the exceedances of the high thresholds $u_n$ by counting the number of such exceedances on a rescaled time interval. The sequence of thresholds $(u_n)_{n\in\N}$ is chosen such that
\begin{equation}
\label{eq:un}
n\p(X_0>u_n)\to \tau,\;\mbox{ for some $\tau>0$, as $n\to\infty$,}
\end{equation}
so that the number of exceedances among the first $n$ observations is kept, approximately, at the constant rate $\tau>0$.  These counting processes were called Rare Events Point Processes (REPP) and were studied in \cite{FFT10,FFT13,FHN14, CFFH15}. Here, we consider even more sophisticated cases like when each exceedance is marked by the respective excess over the threshold $u_n$. In fact, the marked point processes will be defined by keeping record of the occurrence of clusters of exceedances and each such occurrence will be marked by the number of exceedances in the cluster (which allows us to recover the REPP), the sum of the excesses of all exceedances in a cluster, the maximum excess in the cluster or any other measure weighing the intensity of each cluster.

In order to provide a proper framework of the problem we introduce next the necessary formalism to state the results regarding the convergence of point processes and random measures. We recommend the book of Kallenberg \cite{K86} for further details on these topics.

\subsection{Random measures and weak convergence}

First we introduce the notions of \emph{random measures} and, in particular, \emph{point processes} and \emph{marked point processes} on the positive real line. Consider the interval $[0,\infty)$ and its Borel $\sigma$-algebra $\B_{[0,\infty)}$. A positive measure $\nu$ on $\B_{[0,\infty)}$ is said to be a Radon measure if $\nu(A)<\infty$ for every bounded set $A\in\B_{[0,\infty)}$. Let $\M:=\M([0,\infty))$ denote the space of all Radon measures defined on $([0,\infty),\B_{[0,\infty)})$. We equip this space with the vague topology, i.e., $\nu_n\to \nu$ in $\M([0,\infty))$ whenever $\nu_n(\psi)\to \nu(\psi)$ for any continuous function $\psi:[0,\infty)\to \R$ with compact support. Consider the subsets of $\M$ defined by $\M_p:=\{\nu\in\M: \nu(A)\in\N \mbox{ for all $A\in\B_{[0,\infty)}$}\}$ and $\M_a:=\{\nu\in\M: \nu \mbox{ is an atomic measure}\}$. A \emph{random measure} $M$ on $[0,\infty)$ is a random element of $\M$, \ie let $(\X,\B_\X, \p)$ be a probability space, then any measurable $M:\X\to \M$ is a random measure on $[0,\infty)$. A \emph{point process} $N$ and \emph{marked point process} $A$ are defined similarly as random elements on $\M_p$ and $\M_a$, respectively.

The elements $\nu$ of $\M_p$ can be interpreted as counting measure, \ie
$
\nu=\sum_{i=1}^\infty \delta_{x_i},
$
where  $x_1, x_2, \ldots$ is a collection of not necessarily distinct points in $[0,\infty)$ and $\delta_{x_i}$ is the Dirac measure at $x_i$, \ie for every $A\in\B_{[0,\infty)}$, we have that $\delta_{x_i}(A)=1$ if $x_i\in A$ and
$\delta_{x_i}(A)=0$, otherwise. The elements $\nu$ of $\M_a$ can be written as 
$
\nu=\sum_{i=1}^\infty d_i \delta_{x_i},
$
where  $x_1, x_2, \ldots\in [0,\infty)$ and $d_1, d_2,\ldots\in[0,\infty)$.

To give a concrete example of a marked point process, which in particular will appear as the limit of the marked point processes, we consider:
\begin{definition}
\label{def:compound-poisson-process}
Let $T_1, T_2,\ldots$ be  an i.i.d. sequence of r.v. with common exponential distribution of mean $1/\theta$. Let  $D_1, D_2, \ldots$ be another i.i.d. sequence of r.v., independent of the previous one, and with d.f. $\pi$. Given these sequences, for $J\in\B_{[0,\infty)}$, set
$$
A(J)=\int \I_J\;d\left(\sum_{i=1}^\infty D_i \delta_{T_1+\ldots+T_i}\right).
$$
%where $\delta_t$ denotes the Dirac measure at $t>0$. 
Let $\X$ denote the space of all possible realisations of $T_1, T_2,\ldots$ and   $D_1, D_2, \ldots$, equipped with a product $\sigma$-algebra and measure, then $A:\X\to \M_a([0,\infty))$ is a marked point process which we call a compound Poisson process of intensity $\theta$ and multiplicity d.f. $\pi$.
\end{definition}
\begin{remark}
\label{rem:poisson-process}
When $D_1, D_2, \ldots$ are integer valued positive random variables, $\pi$ is completely defined by the values $\pi_k=\p(D_1=k)$, for every $k\in\N_0$ and $A$ is actually a point process. Note that, if $\pi_1=1$ and $\theta=1$, then $A$ is the standard Poisson process and, for every $t>0$, the random variable $A([0,t))$ has a Poisson distribution of mean $t$.
\end{remark}

Now, we define what we mean by convergence of random measures (see \cite{K86} for more details).
\begin{definition}
\label{def:convergence-point-processes}
Let $(M_n)_{n\in\N}:\X\to  \M$ be a sequence of random measures defined on a probability space $(\X,\mathcal B_\X, \mu)$ and let $M:Y \to  \M$ be another random measure defined on a possibly distinct probability space $(Y,\mathcal B_Y, \nu)$. We say that $M_n$ converges in distribution to $M$  if $\mu\circ M_n^{-1}$ converges weakly to $\nu\circ M^{-1}$, \ie for every bounded continuous function $\varphi$ defined on $\M$, we have $\lim_{n\to\infty}\int \varphi d \mu\circ M_n^{-1}=\int \varphi d \nu\circ M^{-1}$.  We write $M_n \stackrel{\mu}{\Longrightarrow} M $.
\end{definition}

In order to check convergence of random measures it is useful to translate it into convergence in distribution of more tractable random variables or in terms of Laplace transforms. With that purpose, we let $\S$ denote the semi-ring of subsets of  $\R_0^+$ whose elements
are intervals of the type $[a,b)$, for $a,b\in\R_0^+$. Let $\RR$
denote the ring generated by $\S$.  Recall that for every $J\in\RR$
there are $\varsigma\in\N$ and $\varsigma$ disjoint intervals $I_1,\ldots,I_\varsigma\in\S$ such that
$J=\dot\cup_{i=1}^\varsigma I_j$. In order to fix notation, let
$a_j,b_j\in\R_0^+$ be such that $I_j=[a_j,b_j)\in\S$.

\begin{definition}
\label{def:Laplace-rv}
Let $Z$ be a non-negative, random variable with distribution $F$. For every $y\in\R_0^+$, the \emph{Laplace transform} $\phi(y)$ of the distribution $F$ is given by
\[
\phi(y):=\E\left(\e^{-yZ}\right)=\int \e^{-yZ} d\mu_F,
\]
where $\mu_F$ is the Lebesgue-Stieltjes probability measure associated to the distribution function $F$.   
\end{definition}

\begin{definition}
\label{def:Laplace-point-process}
For a random measure $M$ on $\R_0^+$ and $\varsigma$ disjoint intervals $I_1, I_2,\ldots, I_\varsigma\in\S$ and non-negative $y_1, y_2,\ldots,y_\varsigma$, we define the \emph{joint Laplace transform} $\psi(y_1, y_2,\ldots,y_\varsigma)$ by
\[
\psi_M(y_1, y_2,\ldots,y_\varsigma)=\E\left(\e^{-\sum_{\ell=1}^\varsigma y_\ell M(I_\ell)}\right).
\] 
\end{definition}

If $M$ is a compound Poisson point process with intensity $\lambda$ and multiplicity distribution $\pi$, then given $\varsigma$ disjoint intervals $I_1, I_2,\ldots, I_\varsigma\in\S$ and non-negative $y_1, y_2,\ldots,y_\varsigma$ we have:
\[
\psi_M(y_1, y_2,\ldots,y_\varsigma)=\e^{-\lambda\sum_{\ell=1}^\varsigma (1-\phi(y_\ell))|I_\ell|},
\]
where $\phi(y)$ is the Laplace transform of the multiplicity distribution $\pi$. 

\begin{remark}
\label{rem:convergence-point-processes}
By \cite[Theorem~4.2]{K86}, the sequence of random measures $(M_n)_{n\in\N}$ converges in distribution to the random measure $M$ iff the sequence of vector r.v.  $(M_n(J_1), \ldots, M_n(J_\varsigma))$ converges in distribution to $(M(J_1), \ldots, M(J_\varsigma))$, for every $\varsigma\in\N$ and all disjoint $J_1,\ldots, J_\varsigma\in\S$ such that $M(\partial J_\ell)=0$ a.s., for $\ell=1,\ldots,\varsigma$, which will be the case if the respective joint Laplace transforms $\psi_{M_n}(y_1, y_2,\ldots,y_\varsigma)$ converge to the joint Laplace transform $\psi_M(y_1, y_2,\ldots,y_\varsigma)$, for all $y_1,\ldots, y_\varsigma\in[0,\infty)$.
\end{remark}

\subsection{Marked Point Processes of Rare Events}

We start by defining some concepts and events that will be used in the definition of the marked point processes and of the dependence conditions needed to assure their convergence.

Let $A\in\B$. We define a function that we refer to as \emph{first hitting time function} to $A$, denoted by $r_A:\X\to\N\cup\{+\infty\}$ where
\begin{equation}
\label{eq:hitting-time}
r_A(x)=\min\left\{j\in\N\cup\{+\infty\}:\; f^j(x)\in A\right\}.
\end{equation}
The restriction of $r_A$ to $A$ is called the \emph{first return time function} to $A$. We define the \emph{first return time} to $A$, which we denote by $R(A)$, as the essential infimum of the return time function to $A$,
\begin{equation}
\label{eq:first-return}
R(A)=\essinf_{x\in A} r_A(x).
\end{equation}

We define, for each $j>1$, the $j$-th \emph{waiting} (or \emph{inter-hitting}) \emph{time} as
\begin{equation}
\label{eq:def-multiple-returns}
w_A^j(x):=r_A\left(f^{\sum_{i=1}^{j-1}w_A^i(x)}(x)\right),
\end{equation}
where $w_A^1(x):=r_A(x)$ and the \emph{$j$-th hitting time} as
\begin{equation}
\label{eq:jth-return}
r_A^j(x):=\sum_{i=1}^j w_A^i(x).
\end{equation}

For $u\in\R$, $p,i,\kappa,s\in\N_0$ and $\ell\in\N$, we set $U_p^{(0)}(u)=U(u)=\{X_0>u\}$ and define the following events:
\begin{align*}
U_p^{(\kappa)}(u)&:=U(u)\cap\bigcap_{i=1}^\kappa\left\{w_{U(u)}^i\leq p\right\} \qquad
U_p^{(\infty)}(u):=U(u)\cap\bigcap_{i=1}^\infty\left\{w_{U(u)}^i\leq p\right\}=\bigcap_{\kappa=0}^\infty U_p^{(\kappa)}(u)\\
Q_{p,i}^\kappa(u)&:=f^{-i}\left(U_p^{(\kappa)}(u)\setminus U_p^{(\kappa+1)}(u)\right)=f^{-i}\left(U_p^{(\kappa)}(u)\cap\{w_{U(u)}^{\kappa+1}>p\}\right)
\end{align*}
If $p=0$ then $U_0^{(\kappa)}(u)=\emptyset$ for all $\kappa\geq1$ and $Q_{0,0}^0(u)=U(u)=\{X_0>u\}$. One of the main ideas in \cite{FFT12} and further developed in \cite{FFT13} is that the events $Q_{p,0}^0(u)=\{X_0>u, X_{1}\leq u,\ldots,X_{p}\leq u\}$ (when $p>0$) play a key role in determining and identifying the clusters. In fact,  we have that every cluster ends with an entrance in $Q_{p,0}^0(u)$, meaning that the inter cluster exceedances must appear separated at most by $p$ units of time.
Hence, given an interval $I\in\S$, $x\in\X$ and $u\in\R$, we define $$N(I)(x,u)=\sum_{j\in I\cap\N_0} \I_{f^{-j}(Q_{p,0}^0(u))}(x).$$ Let $i_1(x,u)<i_2(x,u)<\ldots<i_{N(I)(x,u)}(x,u)$ denote the times at which the orbit of $x$ entered  $Q_{p,0}^0(u)$, \ie $f^{i_k(x,u)}(x)\in Q_{p,0}^0(u)$ for all $k=1, \ldots,N(I)(x,u)$.  We now define the cluster periods: for every $j=1,\ldots,N(I)(x,u)-1$ let $I_j(x,u)=(i_j(x,u), i_{j+1}(x,u)]$ and set $I_0(x,u)=[\min I, i_1(x,u)]$ and $I_{N(I)(x,u)}(x,u)=(i_{N(I)(x,u)}(x,u), \sup (I))$. In order to define the marks for each cluster we consider the following mark functions that depend on the level $u$ and on the random variables in a certain time frame $I^*\in\mathcal S$:
%
%We will treat simultaneously the AOT and the POT point processes by introducing the following notation: given a set of variables $\{X_i\}_{i\in I}$, let
\begin{equation}
\label{eq:mark-type}m_u(\{X_i\}_{i\in I^*\cap\N_0}):=\begin{cases}
\sum_{i\in I^*\cap \N_0}(X_i-u)_+ & \text{AOT case}\\
\max_{i\in I^*\cap \N_0}\{(X_i-u)_+\} & \text{POT case}\\
\sum_{i\in I^*\cap \N_0} \I_{X_i>u} & \text{REPP case},
\end{cases}
\end{equation}
where $(y)_+=\max\{y,0\}$ and 
when $I^*\neq\emptyset$. Also set $m_u(\emptyset):=0$.

We now define the cluster marks for each $j=0,1,\ldots, N(I)(x,u)$ by:
\begin{equation*}
\label{eq:cluster-mark}
D_j(x,u):=m_u(\{X_i\}_{i\in I_j(x,u)\cap\N_0}).
\end{equation*}
Finally, we define
\begin{equation}
\label{eq:A_u-definition}
\aa_u(I)(x):=\sum_{j=0}^{N(I)(x,u)} D_j(x,u).
\end{equation}

In order to define the marked point processes in such a way that they admit a non-degenerate limit, we introduce a link between the number of observations and the thresholds by considering the sequence of levels $(u_n)_{n\in\N}$ satisfying \eqref{eq:un} and by rescaling time by the factor $$v_n:=1/\p(X_0>u_n)$$ given by Kac's Theorem so that the expected number of exceedances of the level $u_n$ in each time frame considered is kept `constant' as $n\to\infty$. Hence, we introduce the following notation.  
 For
$I=[a,b)\in\S$ and $\alpha\in \R$, we denote $\alpha I:=[\alpha
a,\alpha b)$ and $I+\alpha:=[a+\alpha,b+\alpha)$. Similarly, for
$J\in\RR$, such that $J=J_1\dot\cup\ldots\dot\cup J_k$,  define $\alpha J:=\alpha J_1\dot\cup\cdots\dot\cup \alpha J_k$ and
$J+\alpha:=(J_1+\alpha)\dot\cup\cdots\dot\cup (J_k+\alpha)$.

\begin{definition}
\label{def:MREPP}
We define the \emph{marked rare event point process} (MREPP) by setting for every $J\in\RR$, with $J=J_1\dot\cup\ldots\dot\cup J_k$, where $J_i\in \S$ for all $i=1,\ldots, k$,
\begin{equation}
\label{eq:def-MREPP} A_n(J):=\sum_{i=1}^k\aa_{u_n}(v_nJ_i).
\end{equation}
\end{definition}

When $m_u$ given in \eqref{eq:mark-type} is as in the AOT case, then the MREPP $A_n$ computes the sum of all excesses over the threshold  $u_n$ and, in such case, we will refer to $A_n$ as being an \emph{area over threshold} or AOT MREPP. Observe that in this case we may write:
$$
A_n(J)=\sum_{j\in v_nJ\cap \N_0} (X_j-u_n)_+.
$$

When $m_u$ given in \eqref{eq:mark-type} is as in the POT case, then the MREPP $A_n$ computes the sum of the largest excess (peak)  over the threshold $u_n$ within each cluster and, in such case, we will refer to $A_n$ as being a \emph{peaks over threshold} or POT MREPP. 

When $m_u$ given in \eqref{eq:mark-type} is as in the REPP case, then the MREPP $A_n$ is actually a point process that counts the number of exceedances of $u_n$ and, in such case, we will refer to $A_n$ as being a \emph{rare events point process} or REPP, as it was referred in \cite{FFT13}. Observe that in this case we have:
$$
A_n(J)=\sum_{j\in v_nJ\cap \N_0} \I_{X_j>u_n}.
$$
If $p=0$, then $Q_{p,0}^0(u_n)=U(u_n)=\{X_0>u_n\}$ and  in this case the AOT MREPP and the POT MREPP coincide and both compute  the sum of all excesses over the threshold  $u_n$. In such situation we say that $A_n$ is an \emph{excesses over threshold} (EOT) MREPP.

Now, we introduce the dependence conditions $\D_p(u_n)^*$ and $\D'_p(u_n)^*$, with the same flavour as $\D_p(u_n)$ and $\D'_p(u_n)$ considered in \cite{FFT15} but designed to establish the convergence of MREPP (whether they are of the type AOT, POT or simpler REPP), which allow us to state the main result of this paper. The mixing type condition $\D_p(u_n)^*$ also follows easily from sufficiently fast decay of correlations, which makes it particularly useful to apply to stochastic processes arising from dynamical systems, in contrast with condition $\Delta(u_n)$ used by Leadbetter in \cite{L91} or any other similar such condition available in the literature.

For $u\in\R$, $x\geq 0$, $p,i,\kappa,s\in\N_0$ and $\ell\in\N$, we define the following events:
\begin{align*}
U_{p,i}^{\kappa}(u,x)&:=f^{-i}\left(Q_{p,0}^{\kappa}(u)\cap \left\{m_u\left(\{X_j\}_{0\leq j\leq r_{U(u)}^{\kappa}}\right)>x\right\}\right)\\
U_{p,i}(u,x)&:=f^{-i}\left(\bigcup_{\kappa=0}^{\infty}U_{p,0}^\kappa(u,x)\cup U_p^{(\infty)}(u)\right)\\
R_{p,i}(u,x)&:=f^{-i}\left(U_{p,0}(u,x)\cap\left\{r_{U_{p,0}(u,x)}>p\right\}\right)\\
\II_{p,s,\ell}(u,x)&:=\bigcap_{i=s}^{s+\ell-1}\left(U_{p,i}(u,x)\right)^c \qquad \RR_{p,s,\ell}(u,x):=\bigcap_{i=s}^{s+\ell-1} \left(R_{p,i}(u,x)\right)^c
\end{align*}

In particular, for $x=0$ we have
\begin{align*}
U_{p,i}^\kappa(u,0)&=Q_{p,i}^\kappa(u)\qquad
U_{p,i}(u,0)=f^{-i}\left(\bigcup_{\kappa=0}^{\infty}Q_{p,0}^\kappa(u)\cup U_p^{(\infty)}(u)\right)=\{X_i>u\}\\
R_{p,i}(u,0)&=\{X_i>u, X_{i+1}\leq u,\ldots,X_{i+p}\leq u\}=Q_{p,i}^0(u)\\
\end{align*}
and, for $p=0$ we have
\begin{align}
&U_0^{(\kappa)}(u)=\emptyset \mbox{ for $\kappa>0$}\nonumber \qquad Q_{0,i}^0(u)=\{X_i>u\} \mbox{ and } Q_{0,i}^{\kappa}(u)=\emptyset \mbox{ for $\kappa>0$}\nonumber\\
&U_{0,i}^0(u,x)=\{X_i>u,m_u(\{X_i\})>x\} \mbox{ and } U_{0,i}^{\kappa}(u,x)=\emptyset \mbox{ for $\kappa>0$}
\label{eq:Uk-p=0}\\
&R_{0,i}(u,x)=U_{0,i}(u,x)\nonumber%=U_{0,i}^0(u)
\end{align}

\begin{condition}[$\D_p(u_n)^*$]\label{cond:Dp*}We say that $\D_p(u_n)^*$ holds for the sequence $X_0,X_1,X_2,\ldots$ if for $t,n\in\N$, $x_1,\ldots,x_{\varsigma}\geq 0$ and any $J=\cup_{i=2}^\varsigma I_j\in \mathcal R$ with $\inf\{x:x\in J\}\ge t$, \[ \left|\p\left(R_{p,0}(u_n,x_1)\cap \left(\cap_{j=2}^\varsigma \aa_{u_n}(I_j)\leq x_j \right) \right)-\p\left(R_{p,0}(u_n,x_1)\right)\p\left(\cap_{j=2}^\varsigma \aa_{u_n}(I_j)\leq x_j \right)\right|\leq \gamma(n,t),\]
where for each $n$ we have that $\gamma(n,t)$ is nonincreasing in $t$ and $n\gamma(n,t_n)\to 0$ as $n\rightarrow\infty$, for some sequence $t_n=o(n)$, where $\aa_{u_n}$ is given by \eqref{eq:A_u-definition}.
\end{condition}

As mentioned before, this mixing condition is easy to check for stochastic processes arising from dynamical systems with sufficiently fast decay of correlations, as can be appreciated in Theorem~\ref{thm:check-D-D'} (see also Remark~\ref{rem:L1}). This is the main advantage of this condition when compared with Leadbetter's $\Delta(u_n)$  and others of the same kind.

For some fixed $p\in\N_0$, consider the sequence $(t_n)_{n\in\N}$ given by $\D_p(u_n)^*$  and  let $(k_n)_{n\in\N}$ be such that
\begin{equation}
\label{eq:kn-sequence}
k_n\to\infty\quad \mbox{and} \quad k_n t_n = o(n).
\end{equation}

\begin{condition}[$\D'_p(u_n)^*$]\label{cond:D'p} We say that $\D'_p(u_n)^*$ holds for the sequence $X_0,X_1,X_2,\ldots$ if there exists a sequence $(k_n)_{n\in\N}$ satisfying \eqref{eq:kn-sequence} such that
\label{eq:D'rho-un}
\[\lim_{n\rightarrow\infty}\,n\sum_{j=p+1}^{\lfloor n/k_n\rfloor-1}\p\left(Q_{p,0}^0(u_n)\cap \{X_j>u_n\}\right)=0.\]
\end{condition}
In this approach, it is rather important to observe the prominent role played by condition $\D'_p(u_n)^*$. In particular, note that if condition $\D'_p(u_n)^*$ holds for some particular $p=p_0\in\N_0$, then condition $\D'_p(u_n)^*$ holds for all $p\geq p_0$. This suggests that in trying to find the existence of EVL, one should try the values $p=p_0$ until we find the smallest one that makes $\D'_p(u_n)$ hold. Assume that there exists $p\in\N_0$ such that
\begin{equation}
\label{eq:q-def EVL}
p:=\min\left\{j\in\N_0: \lim_{n\to\infty}R(Q_{j,0}^0(u_n))=\infty\right\},
\end{equation}
where $R$ is as in \eqref{eq:first-return}. 
Such $p$ is the natural candidate to try to show the validity of $\D'_p(u_n)$ and then define
\begin{equation}
\label{def:thetan}
\theta_n:=\frac{\p\left(Q_{p,0}^0(u_n)\right)}{\p(U(u_n))}.
\end{equation}
If $\D'_{p_0}(u_n)^*$ holds and if the limit of $\theta_n$ in \eqref{def:thetan} exists for such $p=p_0$, it will also exist for all $p\geq p_0$ and takes always the same value. In this case, let $\theta=\lim_{n\to\infty}\theta_n$ and then $\theta$ is called the Extremal Index (EI). 
When $p=0$, observe that $\D'_p(u_n)^*$ is condition $D'(u_n)$ from Leadbetter, which prevents clustering of exceedances. In particular, in this case $\theta_n=1$, for all $n\in\N$ and we  get an EI equal to 1. 

When $p>0$, we have clustering of exceedances, \ie the exceedances have a tendency to appear aggregated in groups (called clusters), whose mean size is typically given by the inverse of the value of the EI $\theta$. 

We will also assume: 
\begin{condition}[Multiplicity limit]
There exists a normalising sequence $(a_n)_{n\in\N}$ and a probability distribution $\pi$ such that:
\begin{equation}
\label{eq:multiplicity}
\lim_{n\to\infty}\frac{\p(R_{p,0}(u_n,x/a_n))}{\p(U(u_n))}=\theta(1-\pi(x)),\forall x\geq 0.
\end{equation}
\end{condition}
We will see that \eqref{eq:multiplicity} provides a nice formula to compute the multiplicity distribution of the limiting compound Poisson process, which will be used in Sections~\ref{subsec:POT}  and \ref{subsec:AOT}.

Finally, we give a technical condition which imposes a sufficiently fast decay of the probability of having very long clusters. We will call it $U\!LC_p(u_n)$ that stands for `Unlikely Long Clusters'. Of course this condition is trivially satisfied when there is no clustering. Moreover, this condition can be easily checked (see Proposition~\ref{prop:ULC} below) when we are dealing with the case when $\zeta$ is a repelling periodic point. 

\begin{condition}[$U\!LC_p(u_n)$]
We say that condition $U\!LC_p(u_n)$ holds if for all $y>0$
\[\limsup_{n\to\infty}\;n\, \int_0^{\infty}y\e^{-yx}\delta_{p,\lfloor n/k_n\rfloor,u_n}(x/a_n)dx  <\infty\]

where $a_n$ is as in \eqref{eq:multiplicity}, $\delta_{0,s,u}(x):=0$ and, for $p>0$,
\begin{equation}\label{eq:delta-definition}\delta_{p,s,u}(x):=\sum_{\kappa=1}^{\lfloor s/p\rfloor}\kappa p\p(U^{\kappa}_{p,0}(u,x))+s\sum_{\kappa=\lfloor s/p\rfloor+1}^{\infty}\p(U^{\kappa}_{p,0}(u,x))+p\p(U_{p,0}(u,x))\end{equation}
\[=p\sum_{\kappa=0}^{\lfloor s/p\rfloor}(\kappa+1)\p(U^{\kappa}_{p,0}(u,x))+(s+p)\sum_{\kappa=\lfloor s/p\rfloor+1}^{\infty}\p(U^{\kappa}_{p,0}(u,x))\]
is an integrable function in $\R^+$ for $u$ sufficiently close to $u_F=\varphi(\zeta)$. %In particular, it is sufficient to check if
%\[\int_0^{\infty}\delta_{p,\lfloor n/k_n\rfloor,u_n}(x)dx = O(1/n)\]
\end{condition}

Note that, by definition, condition $U\!LC_0(u_n)$ always holds.

We emphasise that this is indeed a technical condition that hardly imposes any restriction to the applications to dynamical systems. In fact, although we do not address such examples here, it can also be checked in situations when $\zeta$ is a discontinuity point as in \cite{AFV15} or when we have multiple correlated or uncorrelated maximal points $\zeta_1, \ldots, \zeta_k$ as in \cite{AFFR16}.

We are now ready to state the main convergence result:

\begin{theorem}
\label{thm:convergence}
Let $X_0,X_1,\ldots$ be given by \eqref{eq:def-stat-stoch-proc-DS} and $(u_n)_{n\in\N}$ be a sequence satisfying \eqref{eq:un}. Assume that $\D_p(u_n)^*$, $\D'_p(u_n)^*$ and $U\!LC_p(u_n)^*$ hold, for some $p\in\N_0$. Assume that $\lim_{n\to\infty}\theta_n=\theta\in(0,1]$ and the existence of a normalising sequence $(a_n)_{n\in\N}$ and a probability distribution $\pi$ such that \eqref{eq:multiplicity} holds. Then, the MREPP $a_n A_n$, where $A_n$ is given by Definition~\ref{def:MREPP} for any of the 3 mark functions considered in \eqref{eq:mark-type},  converges in distribution to a compound Poisson process $A$ with intensity $\theta$ and multiplicity d.f. $\pi$.
\end{theorem}

\begin{remark}
\label{rem:general-marks}
In the proof of this theorem what is essential about the mark function $m_u$ considered in \eqref{eq:mark-type} to define the respective MREPP is that it satisfies the following assumptions:
\begin{enumerate}
\item $m_u(\{X_i\}_{i\in I^*\cap\N_0})\geq 0$ and $m_u(\emptyset)=0$
\item $m_u(\{X_i\}_{i\in I^*\cap\N_0})\leq m_u(\{X_i\}_{i\in J^*\cap\N_0})$ if $I^*\subset J^*$
\item $m_u(\{X_i\}_{i\in I^*\cap\N_0})=m_u(\{X_i\}_{i\in J^*\cap\N_0})$ if $X_i\leq u,\forall i\in (I^*\setminus J^*)\cap\N_0$
\end{enumerate}
Note that, in particular, we must have $m_u(\{X_i\}_{i\in I^*\cap\N_0})=0 \mbox{ if } X_i\leq u,\forall i\in I^*\cap\N_0$.

As long as the above assumptions hold then the conclusion of Theorem~\ref{thm:convergence} holds for the MREPP defined from such a mark function $m_u$ satisfying the three assumptions just enumerated.
\end{remark}

\begin{remark}
\label{rem:stationary-processes}
The main purpose of this paper is to develop a methodology to prove the convergence of marked rare events point processes for stochastic processes arising from chaotic dynamics. For that reason we assume a priori that the processes are generated as in \eqref{eq:def-stat-stoch-proc-DS}. However, Theorem~\ref{thm:convergence} holds for general stationary stochastic processes, which can be seen by realising that every stationary stochastic process can be modelled by  \eqref{eq:def-stat-stoch-proc-DS}. In fact, if $X_0, X_1, \ldots$ is a stationary stochastic process, then taking $\mathcal X$ as the space of each possible realisation of the stochastic process, $f$ as the shift map on such space and $\varphi$ as the projection on the first coordinate, we can write any stationary stochastic process in the form given by  \eqref{eq:def-stat-stoch-proc-DS}.
\end{remark}

In order to have an idea of the scope of applications to specific dynamical systems we consider the type of properties that a system must have in order to check the abstract conditions of Theorem~\ref{thm:convergence}.

First we start by understanding what exceeding a high threshold means in terms of the dynamics. To that end,
we suppose that the r.v. $\varphi:\X\to\R\cup\{\pm\infty\}$
achieves its maximum value at a finite number of points, namely, $\zeta_1,\ldots, \zeta_N\in \X$ (we allow
$\varphi(\zeta_i)=+\infty$).

We assume that $\varphi$ and $\p$ are sufficiently regular, so that %for $u$ sufficiently close to $u_F:=\varphi(\zeta)$, the event 
\begin{enumerate}
\item[{(R1)}]
for $u$ sufficiently close to $u_F:=\varphi(\zeta_i)$ $(i\in \{1,\ldots,N\})$,
\begin{equation*}
\label{def:U}
U(u):=\{x\in\mathcal{X}:\; \varphi(x)>u\}=\{X_0>u\}
\end{equation*}
corresponds to a disjoint union of balls  centred at the points $\zeta_i$, i.e., $U(u)=\bigcup_{i=1}^NB_{\varepsilon_i}(\zeta_i)$
with $\varepsilon_i=\varepsilon_i(u)$. Moreover, the quantity $\p(U(u))$, as a function of $u$, varies continuously on a neighborhood of $u_F$.
\end{enumerate}

The conditions $\D_p(u_n)^*$ and $\D'_p(u_n)^*$ are conditions on the long range and short range dependence structure of the processes, respectively, and they can be easily checked if the system has some strong form of decay of correlations such as decay of correlations against $L^1$ observables, which we define next.

\begin{definition}[Decay of correlations]
\label{def:dc}
Let \( \mathcal C_{1}, \mathcal C_{2} \) denote Banach spaces of real valued measurable functions defined on \( \X \).
We denote the \emph{correlation} of non-zero functions $\phi\in \mathcal C_{1}$ and  \( \psi\in \mathcal C_{2} \) w.r.t.\ a measure $\p$ as
\[
\cv_\p(\phi,\psi,n):=\frac{1}{\|\phi\|_{\mathcal C_{1}}\|\psi\|_{\mathcal C_{2}}}
\left|\int \phi\, (\psi\circ f^n)\, \dif\p-\int  \phi\, \dif\p\int
\psi\, \dif\p\right|.
\]

We say that we have \emph{decay
of correlations}, w.r.t.\ the measure $\p$, for observables in $\mathcal C_1$ \emph{against}
observables in $\mathcal C_2$ if, for every $\phi\in\mathcal C_1$ and every
$\psi\in\mathcal C_2$ we have
 $$\cv_\p(\phi,\psi,n)\to 0,\quad\text{ as $n\to\infty$.}$$
  \end{definition}

We say that we have \emph{decay of correlations against $L^1$
observables} whenever  this holds for $\mathcal C_2=L^1(\p)$  and
$\|\psi\|_{\mathcal C_{2}}=\|\psi\|_1=\int |\psi|\,\dif\p$.

Examples of systems with such property include:
\begin{itemize}

\item Uniformly expanding maps on the circle/interval (see \cite{BG97});

\item Markov maps (see \cite{BG97});

\item Piecewise expanding maps of the interval with countably many branches like Rychlik maps (see \cite{R83});

\item Higher dimensional piecewise expanding maps studied by Saussol in \cite{S00}.

\end{itemize}

\begin{remark}
\label{rem:C-space}
In the first three examples above the Banach space $\mathcal C_1$ for the decay of correlations can be taken as the space of functions of bounded variation. In the fourth example  the Banach space $\mathcal C_1$ is the space of functions with finite quasi-H\"older norm studied in \cite{S00}. We refer the readers to \cite{BG97,S00} or \cite{AFV15} for precise definitions but mention that if $I\subset\R$ is an interval then $\I_I$ is of bounded variation and its BV-norm is equal to 2, \ie $\|\I_I\|_{BV}=2$, and if $A$ denotes a ball or an annulus then $\I_A$ has a finite quasi-H\"older norm.
\end{remark}

\begin{theorem}
\label{thm:check-D-D'}
Let $f:\X\to\X$ be a system with summable decay of correlations against $L^1$ observables, \ie for all $\phi\in\mathcal C_1$ and $\psi\in L^1$, then $\cv(\phi,\psi,n)\leq \rho_n$, with $\sum_{n\geq 1}\rho_n<\infty$. Assume that there exists $p\in\N_0$ such that \eqref{eq:q-def EVL} holds and there exists $C>0$ such that for all $n\in\N$ and $x\in\R_0^+$ we have $\I_{R_{p,0}(u_n,x)}\in \mathcal C_1$ and $\|\I_{R_{p,0}(u_n,x)}\|_{\mathcal C_1}\leq C$. Then conditions $\D_p(u_n)^*$ and $\D'_p(u_n)^*$ hold.
%$\I_{\A_n}, \I_{\AAk}\in\mathcal C_1$ and $\|\I_{\A_n}\|_{\mathcal C_1}, \|\I_{\AAk}\|_{\mathcal C_1}\leq C$. }
\end{theorem}

\begin{remark}
\label{rem:Holder-norm}
Although we assumed for simplicity that $\I_{R_{p,0}(u_n,x)}\in \mathcal C_1$ in the last theorem to simplify the proof of $\D_p(u_n)^*$, which can easily be verified when $\mathcal C_1$ is the space of functions of bounded variation or quasi-H\"older, one can still check condition $\D_p(u_n)^*$ when $\mathcal C_1$ is the space of H\"older functions, for example, in which case we have $\I_{R_{p,0}(u_n,x)}\notin \mathcal C_1$. This can be proved with minor adjustments to \cite[Proposition~3.1]{FFT13}.
\end{remark}

As shown in \cite{FFT12} the appearance of clustering of exceedances in a dynamical setting is associated to periodic behaviour. This was seen in \cite{HV09, FFT12, FFT13} when the maximum value of $\varphi$ is attained at a single point $\zeta$ that happens to be a repelling periodic point \footnote{\label{repelling periodic point} We say that $\zeta$ is a periodic point of prime period $p$ if $f^p(\zeta)=\zeta$ and $f^j(\zeta)\neq\zeta$, for all $j=1,\ldots,p-1$. A periodic point is said to be repelling if $Df^p$ is defined at $\zeta$ and $\|(Df^p(\zeta))^{-1}\|<1$, where $\|\cdot\|$ is the norm on the tangent space to $\X$ at $\zeta$ given by the Riemannian structure.} but, as in the \cite{AFFR16}, it can also appear due to fake periodicity created by taking multiple maximal points which are related by belonging to the orbit of the same point $\xi$. To illustrate that condition $U\!LC_p(u_n)$ is very easily checked, we show that it holds whenever we have a single maximum $\zeta$, which is a repelling periodic point of prime period $p$. Assume that $\varphi$ and $\p$ are sufficiently regular at $\zeta$ so that:
\begin{enumerate}
\item[\namedlabel{item:repeller}{(R2)}]
the periodicity of $\zeta$ implies that for all large $u$, $\{X_0>u\}\cap f^{-p}(\{X_0>u\})\neq\emptyset$ and the fact that the prime period is $p$ implies that $\{X_0>u\}\cap f^{-j}(\{X_0>u\})=\emptyset$ for all $j=1,\ldots,p-1$. The fact that $\zeta$ is repelling means that we have backward contraction implying that
$U_p^{(\infty)}(u)=\{\zeta\}$
and implying that there exists $0<\theta<1$ so that $\bigcap_{j=0}^\kappa f^{-jp}(X_0>u)$ is a ball around $\zeta$ with
\[\p\left(\bigcap_{j=0}^\kappa f^{-jp}(\{X_0>u\})\right)\sim(1-\theta)^\kappa\p(X_0>u)\]
\end{enumerate}

\begin{proposition}
\label{prop:ULC}
Let $f:\X\to\X$ be a system and $\varphi:\X\to\R$ have global maximum at $\zeta$, which is a repelling periodic point of prime period $p$ for which $(R2)$ holds. Then condition $U\!LC_p(u_n)$ is satisfied. 
\end{proposition}

\begin{remark}
\label{rem:ULC}
We remark that for examples considered in \cite{AFFR16}, condition $U\!LC_p(u_n)$ can also be checked with the same amount of effort necessary to prove the last proposition. For its proof see the end of Section~\ref{subsec:single-maximum}.
\end{remark}

The assumption on decay of correlations against $L^1$ observables is quite strong. In fact, as shown in \cite{AFLV11}, summable decay of correlations against $L^1$ implies exponential decay of correlations  of H\"older observables against $L^\infty$ ones. From the examples listed above, one perceives that it holds essentially (and up to our best knowledge) in uniformly expanding realm.  

One way of expanding the scope of applications is to consider systems which admit nice first-return induced maps, for which we can prove the existence of limits for the MREPP and then pass that information for the original system. In \cite{BSTV03}, the authors showed that the original system and the first return induced system shared the same Hitting Times Statistics %(which implies the same EVLs) 
for ball targets shrinking to $\zeta$ (which plays the role of the single maximum of $\varphi$). Their statement held for a.e $\zeta$ and the standard exponential law applied. Then in \cite{FFT13}, the authors showed that the same limit for REPP applies for the original system and the first return induced system when $\zeta$ is a repelling periodic point. In \cite{HWZ14}, the result of \cite{BSTV03} was generalised to all points $\zeta$ and in \cite{FFTV16} the latter was generalised for the convergence of REPP. However, the statement of \cite[Theorem~3]{FFTV16} holds only for point processes and its proof relies on  \cite[Corollary~6]{Z07} that was only proved for point processes also. Hence, in order to be able to extend our results here for systems admitting a nice first-return induced map, we need to prove a generalisation of 
\cite[Theorem~3]{FFTV16} to atomic random measures, for which we cannot use  \cite[Corollary~6]{Z07}. Hence, we will prove Theorem~\ref{thm:mpp_ret_orig} below.

Let $f:\X\to \X$ be a system with an ergodic $f$-invariant probability measure $\p$, choose a subset $B\subset \X$ and consider $F_B:B\to B$ to be the first return map $f^{r_B}$ to $B$ (note that $F_B$ may be undefined at a zero Lebesgue measure set of points which do not return to $B$, but most of these points are not important, so we will abuse notation here). Let $\p_B(\cdot)=\frac{\p(\cdot\; \cap B)}{\p(B)}$ be the conditional measure on $B$. By Kac's Theorem $\p_B$ is $F_B$-invariant.

Setting $v_n^B=\frac{1}{\p_B(X_0>u_n)}$, for the induced process $X_i^B=\varphi\circ F_B^i$ we define for every $J\in\RR$, with $J=J_1\dot\cup\ldots\dot\cup J_k$ and $J_i\in\S$ for all $i=1,\ldots, k$,
$$A_n^B(J):=\sum_{i=1}^k\aa_{u_n}^B(v_n^BJ_i).$$
where, for every interval $I\in\S$,
$$\aa_u^B(I)(x):=\sum_{j=0}^{N(I)(x,u)} m_u(\{X_i^B\}_{i\in I_j(x,u)\cap\N_0}).$$

For an interval $I\in\mathcal S$ and $\eps<|I|$ we define: 
$$I^{\eps+}=(I+\eps)\cup(I-\eps)\in\mathcal S\qquad I^{\eps -}=(I+\eps)\cap (I-\eps) \in\mathcal S.$$
If $J\in\mathcal R$ we define $J^{\eps\pm}$ accordingly.

\begin{theorem}
For $\eps>0$, we assume that the limit marked point process $A(I^{\eps\pm})$ is continuous in $\eps$, for all small $\eps$. Also assume that for $n$ sufficiently large we have $U(u_n)\subset B\in\mathcal B$. Then
$$A_n^B\stackrel{\p_B}{\Longrightarrow} A \text{ as } n\to\infty \text{ implies } A_n \stackrel{\p}{\Longrightarrow} A \text{ as } n\to\infty.$$
\label{thm:mpp_ret_orig}
\end{theorem}
 
As consequence, if we have a system $f:\X\to\X$ that admits a first-return time induced systems $F_B:B\to B$ such that $F_B$ has decay of correlations against $L^1$ so that we can apply Theorem~\ref{thm:check-D-D'} to prove the convergence of an MREPP then we may use Theorem~\ref{thm:mpp_ret_orig} to prove convergence of the corresponding MREPP for the original system $f$.

Two examples of systems that admit such `nice' first-return induced maps are:
\begin{itemize}

\item \emph{Manneville-Pomeau} (MP) map equipped with an absolutely continuous invariant probability measure.  The form for such maps given in \cite{LSV99, BSTV03} is,   for $\alpha\in (0,1)$,
\begin{equation*}
f=f_\alpha(x)=\begin{cases} x(1+2^\alpha x^\alpha) & \text{ for } x\in [0, 1/2)\\
2x-1 & \text{ for } x\in [1/2, 1]\end{cases}
\end{equation*}
Members of this family of maps are often referred to as Liverani-Saussol-Vaienti maps since their actual equation was first introduced in \cite{LSV99}. Let $\P$ be the \emph{renewal partition}, that is the partition defined inductively by $\cyl\in \P$ if $\cyl=[1/2, 1)$ or $f(\cyl)\in \P$.  Now let $Y\in \P$ and let $F_Y$ be the first return map to $Y$ and $\mu_Y$ be the conditional measure on $Y$.  It is well-known that $(Y, F_Y, \mu_Y)$ is a Bernoulli map and hence, in particular, a Rychlik system (see \cite{R83} or \cite[Section~3.2.1]{AFV15} for the essential information about such systems).

\item  We consider a class of $C^3$ unimodal interval maps $f:I \to I$ with an invariant probability measure absolutely continuous with respect to Lebesgue measure.  Let $c$ be the critical point.  Such a map is called \emph{$S$-unimodal} if it has negative \emph{Schwarzian derivative}, i.e., $D^3f(x)/Df(x) - \frac32 (D^2f(x)/Df(x))^2<0$ for any $x\in I\sm\{c\}$.  We say that $c$ is \emph{non-flat} if there exists $\ell\in (1,\infty)$ such that $\lim_{x\to c}|f(x)-f(c)|/|x-c|^\ell$ exists and is positive. Here $\ell$ is called the \emph{order} of the critical point.

As in \cite{BSTV03}, if the critical point has an orbit which is not dense in $I$ (e.g. the Misiurewicz case), it is possible to construct a first return map which gives a Rychlik system.

\end{itemize}

In Section~\ref{sec:applications} we will address the issue of the convergence in \eqref{eq:multiplicity} which is related to the shape of the observable $\varphi$ and its behaviour near its maximum value, as well as to the regularity of $\p$. In particular, for certain examples of dynamical systems we will show the convergence of AOT, POT MREPP and compute its limit multiplicity distributions.

\section{Applications to dynamical systems}
\label{sec:applications}

\subsection{The conditions on the dependence structure of the processes}

We begin proving Theorem~\ref{thm:check-D-D'} which allows to automatically verify conditions $\D_p(u_n)^*$ and $\D'_p(u_n)^*$ from decay of correlations against $L^1$ observables. The proof follows the same lines as the verification of earlier conditions of the same type (like $\D_p(u_n)$ and $\D'_p(u_n)$) as in \cite{FFT15} or similar conditions in \cite{FFT12, FFT13, AFV15}, under the same assumption. However, for completeness and because it is short we do it here.

\begin{proof}[Proof of Theorem~\ref{thm:check-D-D'}]
Recall that by assumption $\cv_\p(\phi,\psi,n)\leq\rho_n$, with $\sum_{n\geq1}\rho_n<\infty$. As mentioned earlier, condition $\D_p(u_n)^*$, as its predecessors, is designed to follow easily from decay of correlations (and it does not need to be against $L^1$). 
Take $\phi=\I_{R_{p,0}(u_n,x_1)}$, $\psi=\I_{\left(\cap_{j=2}^\varsigma \aa_{u_n}(I_j-t)\leq x_j \right)}$. %, in \eqref{DC:L1}.
By assumption, there exists $C'>0$ such that $\left\|\I_{R_{p,0}(u_n,x_1)}\right\|_{\mathcal C_1}\leq C'$ for all $n\in\N$ and $x_1\in\R_0^+$. Hence, we have that condition~$\D_p(u_n)^*$ holds with $\gamma(n,t)=\gamma(t):=C'\rho_t$ and by choosing a sequence $(t_n)_{n\in\N}$ such that $t_n=o(n)$ and $\lim_{n\to\infty}n\rho_{t_n}=0$.

We now turn to condition $\D'_p(u_n)^*$. Notice that $Q_{p,0}^0(u_n)=R_{p,0}(u_n,0)$, so taking $\phi=\I_{Q_{p,0}^0(u_n)}$ and $\psi=\I_{X_0>u_n}$ we easily get
\begin{align}
\label{eq:computation}
\p\left(Q_{p,0}^0(u_n)\cap f^{-j}(X_0>u_n)\right) &\leq \p(Q_{p,0}^0(u_n))\p(X_0>u_n)+\left\|\I_{Q_{p,0}^0(u_n)}\right\|_{\mathcal C_1}\p(X_0>u_n) \rho_j\nonumber\\
&\leq \p(X_0>u_n)\left(\p(Q_{p,0}^0(u_n))+C'\rho_j\right)
\end{align}
Recalling that $n\p(X_0>u_n)\to\tau\geq0$ and $p$ is such that \eqref{eq:q-def EVL} holds, then
\begin{align*}
n\sum_{j=p+1}^{\lfloor n/k_n\rfloor-1}&\p\left(Q_{p,0}^0(u_n)\cap f^{-j}(X_0>u_n)\right)=n\sum_{j=R(Q_{p,0}^0(u_n))}^{\lfloor n/k_n\rfloor-1}\p\left(Q_{p,0}^0(u_n)\cap f^{-j}(X_0>u_n)\right)\\
&\leq \frac{n^2}{k_n}\p(X_0>u_n)\p(Q_{p,0}^0(u_n))+n\p(X_0>u_n)C'\sum_{j=R(Q_{p,0}^0(u_n))}^\infty \rho_j\to 0 \quad \text{as $n\to\infty$.}
\end{align*}
\end{proof}

\begin{remark}
\label{rem:L1}  Note that in the proof of Theorem~\ref{thm:check-D-D'}, the fact that the decay of correlations holds against all $L^1$ observables was only used in the proof of $\D'_p(u_n)^*$. In fact, as mentioned earlier, by adapting the proof \cite[Proposition~3.1]{FFT13}, one can easily show that $\D_p(u_n)^*$ follows from decay of correlations of H\"older observables against $L^\infty$ ones.

\end{remark}

In the proof of Theorem~\ref{thm:check-D-D'} we use the fact that we can find $p$ such that \eqref{eq:q-def EVL} holds and consequently $R(Q_{p,0}^0(u_n))\to\infty$, as $n\to\infty$. If we take $q=\max\{n\in\N: f^n(\zeta_i)=\zeta_j,\;i,j=1,\ldots,k\}$ then under mild assumptions on the system we have that $R(Q_{q,0}^0(u_n))\to\infty$, as $n\to\infty$. For example, if the systems is continuous along the orbits of $\zeta_i$, $i=1,\ldots k$, then using a continuity argument and the Hartman-Grobman theorem (when a $\zeta_i$ is periodic) one can show the previous statement. (See \cite[Lemma~4.1 and  Lemma~5.1]{AFFR16}. We remark that one can prove the statement even in situations when the orbit of some $\zeta_i$ hits a discontinuity point of $f$ as it was studied in \cite[Section~3.3]{AFV15}. 

\subsection{Illustrative scenarios of possible application}
\label{subsec:single-maximum}
As  in \cite{AFFR16}, having multiple maximal point creates a large range of possibilities since the local behaviour of $\varphi$ and of the measure $\p$ at each point raises an enormous number of case studies. Our goal here is to illustrate our convergence theorem and compute the limit marked point process for some illustrative examples. Since it would be extremely difficult to cover in a systematic way all the possibilities we make some assumptions from this point to the end of this section intended to simplify the presentation but maintain, as much as possible, the key aspects of potential application.

\textbf{Assumption 1 -- Single global maximum}: There exists a single point $\zeta\in\X$ where $\varphi$ achieves its global maximum value. We allow $\varphi(\zeta)=+\infty$. 

\textbf{Assumption 2 -- Shape of the observable}: The observable $\varphi:\X\to\R\cup\{+\infty\}$ is of the form
\begin{equation}\label{eq:observable-form}
\varphi(x)=g(\dist(x,\zeta)),
\end{equation}
where $g:V\to W$ is a strictly decreasing homeomorphism in a neighbourhood $V$ of $0$ and has one of the following three types of behaviour:

\begin{enumerate}
\item [Type 1:] there exists some strictly positive function $q:W\to\R$ such that for all $y\in\R$
\begin{equation}
\label{eq:def-g1}
\lim_{s\to g(0)}\frac{g^{-1}(s+yq(s))}{g^{-1}(s)}=\e^{-y};
\end{equation}
\item [Type 2:] $g(0)=+\infty$ and there exists $\beta>0$ such that for all $y>0$
\begin{equation}
\label{eq:def-g2}
\lim_{s\to+\infty}\frac{g^{-1}(sy)}{g^{-1}(s)}=y^{-\beta};\end{equation}
\item [Type 3:] $g(0)=D<+\infty$ and there exists $\gamma>0$ such that for all $y>0$
\begin{equation}
\label{eq:def-g3}
\lim_{s\to 0}\frac{g^{-1}(D-sy)}{g^{-1}(D-s)}=y^\gamma.
\end{equation}
\end{enumerate}

Examples of each one of the three types are, respectively, as follows:
$g(x)=-\log x$ (in this case \eqref{eq:def-g1} is easily verified with $q\equiv1$), $g(x)=x^{-1/\alpha}$ for some $\alpha>0$ (condition \eqref{eq:def-g2} is verified with $\beta=\alpha$) and $g(x)=D-x^{1/\alpha}$ for some $D\in\R$ and $\alpha>0$ (condition \eqref{eq:def-g3} is verified with $\gamma=\alpha$).

\textbf{Assumption 3 -- Regularity of $\p$}: Let us now assume we are in the particular case where $\p$ is absolutely continuous with respect to the Lebesgue measure and  its Radon-Nikodym density is sufficiently regular so that for all $x\in\X$ we have
\begin{equation}
\label{eq:density}
\lim_{\eps\to 0}\frac{\p(B_\eps(x))}{\l(B_\eps(x))}=\frac{d\p}{d\l}(x).
\end{equation}

\begin{remark}
Note that if $f$ is one dimensional smooth map modelled by the full shift as in \cite[Section~7.1]{FFT15}
and the derivative is sufficiently regular then, as seen in \cite[Section~7.3]{FFT15}, the invariant density is fairly smooth and formula \eqref{eq:density} holds for all $x\in\X$.
\end{remark}

\begin{remark}
The different types of $g$ imply that the distribution of $X_0$ falls in the domain of attraction for maxima of the Gumbel, Fr\'echet and Weibull distributions, respectively.
\end{remark}

We recall that as shown in \cite{AFV15}, under decay of correlations against $L^1$ and the previous assumptions,  either we have clustering when $\zeta$ is a repelling periodic point or at every other non-periodic point $\zeta$ we have no clustering of exceedances and an EI equal to 1. Moreover, note that under the previous assumptions condition $(R1)$ is always satisfied and, in case $\zeta$ is a repelling periodic point of prime period $p$, then $(R2)$ is also satisfied with 
\begin{equation}
\label{eq:EI-formula}
\theta=1-\frac1{\det Df^p(\zeta)}.
\end{equation}
In particular the limit of $\theta_n$, defined in \eqref{def:thetan}, exists and equals such $\theta$.
\begin{remark}
If $\p$ is not absolutely continuous with respect to Lebesgue, we can use instead observables of the form $\varphi(x)=g\left(\mu_\phi\left(B_{\dist(x,\zeta)}(\zeta)\right)\right)$, as introduced in \cite{FFT11}, and the analysis we will carry out could be easily adjusted in order to obtain essentially the same results. In particular, when $\p$ is the more general equilibrium state associated to a potential $\psi$ then condition (R2) can be verified as in \cite[Lemma~3.1]{FFT12} and the EI is given by the formula $\theta=1-\e^{\psi(\zeta)+\ldots+\psi(f^{p-1}(\zeta))}$.
\end{remark}

As mentioned above, if $\zeta$ is not periodic
the condition $U\!LC_0(u_n)$ is trivially satisfied. If $\zeta$ is a periodic point of prime period $p$, since the above assumptions guarantee that (R2) is satisfied then  condition $U\!LC_p(u_n)$ can also be easily verified, as follows.

\begin{proof}[Proof of Proposition~\ref{prop:ULC}]
Since by \eqref{eq:delta-definition}
\begin{align*}
\delta_{p,s,u}(x)&=p\left(\sum_{\kappa=0}^{\lfloor s/p\rfloor}(\kappa+1)\p(U^{\kappa}_{p,0}(u,x))+\sum_{\kappa=\lfloor s/p\rfloor+1}^{\infty}(s/p+1)\p(U^{\kappa}_{p,0}(u,x))\right)\\
&\leq p\sum_{\kappa=0}^\infty(\kappa+1)\p(U^{\kappa}_{p,0}(u,x))\leq p\sum_{\kappa=0}^\infty(\kappa+1)\p(Q^{\kappa}_{p,0}(u))
\end{align*}
for all $x\in\R_0^+$ and $y\in\R^+$, we have
\begin{align*}
\int_0^{\infty}y\e^{-yx}\delta_{p,\lfloor n/k_n\rfloor,u_n}(x)dx &\leq p\sum_{\kappa=0}^\infty(\kappa+1)\p(Q^{\kappa}_{p,0}(u_n))\int_0^\infty y\e^{-yx}dx=p\sum_{\kappa=0}^\infty(\kappa+1)\p(Q^{\kappa}_{p,0}(u_n)).
\end{align*}
So, it is sufficient to check if
\[\limsup_{n\to\infty}{n}\sum_{\kappa=0}^\infty(\kappa+1)\p(Q^{\kappa}_{p,0}(u_n))<\infty\]

By (R2), there exists $0<\theta<1$ so that $\bigcap_{j=0}^i f^{-jp}(X_0>u)$ is a ball around $\zeta$ with
\[\p\left(\bigcap_{j=0}^\kappa f^{-jp}(\{X_0>u\})\right)\sim(1-\theta)^\kappa\p(X_0>u)\]
for all $u$ sufficiently large. So, we have
\[\p(U_p^{(\kappa)}(u_n))\sim(1-\theta)^\kappa\p(U(u_n))\]
\[\p(Q_{p,0}^\kappa(u_n))=\p(U_p^{(\kappa)}(u_n))-\p(U_p^{(\kappa+1)}(u_n))\sim\theta(1-\theta)^\kappa\p(U(u_n))\]
\[\sum_{\kappa=0}^\infty(\kappa+1)\p(Q^{\kappa}_{p,0}(u_n))\sim \sum_{\kappa=0}^\infty(\kappa+1)\theta(1-\theta)^\kappa\p(U(u_n))=\frac{1}{\theta}\p(U(u_n))\]

Since by \eqref{eq:un} we have $\lim_{n\to\infty} n\p(U(u_n))=\tau$, then we conclude that condition $U\!LC_p(u_n)$ is always verified when $\zeta\in\X$ is a repelling periodic point of prime period $p\in\N$ satisfying (R2).
\end{proof}

\subsection{Convergence of the REPP}

When the mark function $m_u$ defined in \eqref{eq:mark-type} counts the number of exceedances then our atomic  random measure $A_n$ is actually a REPP as the one considered in \cite{FFT13}, namely, $
A_n(J)=\sum_{j\in v_nJ\cap \N_0} \I_{X_j>u}.$ We realise here that if we have a system that admits a first return induced map on a base $B$ with decay of correlations against $L^1$ and $\zeta\in B$  is the only global maximum of $\varphi$, which is a periodic point satisfying (R2), which is the case if Assumptions 1--3 hold, then we recover the main result in \cite{FFT13}, which states that $A_n$ converges in distribution to a compound Poisson process of intensity $\theta$ and geometric multiplicity distribution.

To see this, we note the following:
\begin{align*}
&U_p^{(\kappa)}(u)=U(u)\cap\bigcap_{i=1}^{\kappa}\left\{w_{U(u)}^i=p\right\}=\{X_0>u,X_p>u,\ldots,X_{\kappa p}>u\}\\
&Q_{p,i}^{\kappa}=\{X_i>u,X_{i+p}>u,\ldots,X_{i+\kappa p}>u,X_{i+(\kappa+1)p}\leq u\}\\
&m_u\left(\{X_j\}_{0\leq j\leq r_{U(u)}^\kappa}\right)>x%\Leftrightarrow \kappa+1>x
\Leftrightarrow \kappa\geq \lfloor x\rfloor \qquad U_{p,0}^\kappa(u,x)=\begin{cases}
Q_{p,0}^\kappa(u) & \text{if $\kappa\geq \lfloor x\rfloor$}\\
\emptyset & \text{if $\kappa<\lfloor x\rfloor$}\\
\end{cases}\\
&U_{p,0}(u,x)=\bigcup_{\kappa=\lfloor x\rfloor}^{\infty}Q_{p,0}^\kappa(u)\cup\{\zeta\}=U_p^{(\lfloor x\rfloor)}(u)\\ &R_{p,0}(u,x)=U_p^{(\lfloor x\rfloor)}(u)\cap\{r_{U_p^{(\lfloor x\rfloor)}(u)}>p\}=Q_{p,0}^{\lfloor x\rfloor}(u)
\end{align*}
Moreover, we have $\p(U_p^{(\kappa)}(u_n))\sim(1-\theta)^\kappa\p(U(u_n))$ and $\p(Q_{p,0}^\kappa(u_n))\sim\theta(1-\theta)^\kappa\p(U(u))$. The result now follows from observing that
\begin{align*}
\lim_{n\to\infty}\frac{\p(R_{p,0}(u_n,x))}{\p(U(u_n))}&=\lim_{n\to\infty}\frac{\p(Q_{p,0}^{\lfloor x\rfloor}(u_n))}{\p(U(u_n))}=\lim_{n\to\infty}\frac{\theta(1-\theta)^{\lfloor x\rfloor}\p(U(u_n))}{\p(U(u_n))}\\
&=\theta(1-\theta)^{\lfloor x\rfloor}=\theta(1-\pi(x))
\end{align*}
where $\pi(x)=1-(1-\theta)^{\lfloor x\rfloor}$ is the cumulative distribution function of a geometric distribution of parameter $\theta$, that is, $\pi(x)=\sum_{\kappa\leq x,\kappa\in\N}\theta(1-\theta)^{\kappa-1}$.
\begin{remark}
If the point $\zeta$ is not periodic and a dichotomy holds, as in \cite{AFV15}, for the first-return induced system (which we are assuming to have decay of correlations against $L^1$), then condition $\D'_0(u_n)^*$ holds and the REPP is easily seen to converge to a standard Poisson process (with intensity 1).
\end{remark}

\subsection{Computation of the limit of EOT and POT MREPP}
\label{subsec:POT}
When the mark function $m_u$ defined in \eqref{eq:mark-type} weighs the maximum excess within a cluster then our atomic  random measure $A_n$ is a POT MREPP. When there is no clustering then $A_n$ is an EOT MREPP and, as observed above, the POT and AOT MREPP coincide and provide information about the sum of all observed excesses. 

The result below gives that for uniformly expanding and certain non-uniformly expanding dynamical systems the POT MREPP, in the case of presence of clustering, and the EOT MREPP, in the case of its absence, both converge to a compound Poisson process with intensity given by the EI and whose multiplicity distribution is a Generalised Pareto Distribution (GPD), whose type depends on the type of $g$ chosen in Assumption 2.  

\begin{theorem}
\label{thm:EOT-POT}
Let $f:\X\to\X$ be a system admitting a first return induced map $F_B:B\to B$, with $B\subset \X$ and such that $F_B$ has summable decay of correlations against $L^1$ observables \ie for all $\phi\in\mathcal C_1$ and $\psi\in L^1$, then $\cv(\phi,\psi,n)\leq \rho_n$, with $\sum_{n\geq 1}\rho_n<\infty$. Assume that for every $\zeta$, for all balls $B_\eps(\zeta)$ and annuli $B_{\eps_1}(\zeta)\setminus B_{\eps_2}(\zeta)$, with $0<\eps,\eps_1<\eps_2$, then $\I_{ B_\eps(\zeta)}\in\mathcal C_1$, $\I_{B_{\eps_1}(\zeta)\setminus B_{\eps_2}(\zeta)}\in\mathcal C_1$ and their norms are uniformly bounded above. 

Let $X_0,X_1,\ldots$ be given by \eqref{eq:def-stat-stoch-proc-DS} and $(u_n)_{n\in\N}$ be a sequence satisfying \eqref{eq:un}. Assume that $\varphi$ and $\p$ are such that Assumptions 1--3 hold, where $\zeta\in B$. Then,
\begin{itemize}

\item if $\zeta$ is a periodic repelling point of prime period $p$, the POT MREPP $a_nA_n$ converges in distribution to a compound Poisson process with intensity $\theta$ given by formula \eqref{eq:EI-formula} and multiplicity distribution 
\begin{equation}
\label{eq:POT-multiplicity}
\pi(x)=\begin{cases}
1-\e^{-x}, \text{when $g$ is of type 1 and $a_n=(q(u_n))^{-1}$}\\
1-(1+x)^{-\beta}, \text{when $g$ is of type 2 and $a_n=u_n^{-1}$}\\
1-(1-x)^{\gamma}, \text{when $g$ is of type 3 and $a_n=(D-u_n)^{-1}$}
\end{cases}
\end{equation}

\item If $\zeta$ is not periodic and $f$ is continuous on the points of its orbit then the EOT MREPP $a_nA_n$ converges in distribution to a compound Poisson process with intensity 1 and multiplicity distribution given by \eqref{eq:POT-multiplicity}.
\end{itemize}

\end{theorem}

\begin{proof}
By Theorem~\ref{thm:mpp_ret_orig} we only need to prove the result for $F_B$ since then if follows for $f$. 
First we consider the case in which $\zeta$ is not periodic ($p=0$). By Assumptions 1, 2 then $R_{0,0}(u_n,x)=U_{0,0}^0(u_n,x)=B_{\eps}(\zeta)$, for some $\eps>0$ and consequently  $\I_{R_{0,0}(u_n,x)}\in\mathcal C_1$ and $\|\I_{R_{0,0}(u_n,x)}\|_{\mathcal C_1}\leq C$ for every $x>0$ and $n\in \N$. Recalling that in this case $Q_{0,0}^0(u_n)=U(u_n)$, as in \cite[Lemma~3.1]{AFV15}, it follows using a continuity argument that $\lim_{n\to\infty}R(U(u_n))=\infty$. Then all hypothesis of Theorem~\ref{thm:check-D-D'} are satisfied and in conclusion conditions $\D_0(u_n)^*$ and $\D'_0(u_n)^*$ hold. Moreover, as observed earlier, condition $U\!LC_0(u_n)$ is trivially satisfied.

Now we consider the case where $\zeta$ is a repelling periodic point of prime period $p$. By Assumptions 1, 2 then $R_{p,0}(u_n,x)=B_{\eps_1}(\zeta)\setminus B_{\eps_2}(\zeta)$ for some $\eps_1,\eps_2>0$ and consequently  $\I_{R_{p,0}(u_n,x)}\in\mathcal C_1$ and $\|\I_{R_{p,0}(u_n,x)}\|_{\mathcal C_1}\leq C$ for every $x>0$ and $n\in \N$. Moreover, as in the proof of Theorem~2 of \cite{FFT13}, using the Hartman-Grobman theorem, one can easily check  that $\lim_{n\to\infty}R(Q_{p,0}^0(u_n))=\infty$. Then all hypothesis of Theorem~\ref{thm:check-D-D'} are satisfied and in conclusion conditions $\D_p(u_n)^*$ and $\D'_p(u_n)^*$ hold.
By Assumptions 1--3 and the fact that $\zeta$ is a repelling periodic point then (R2) holds and by Proposition~\ref{prop:ULC} so does $U\!LC_p(u_n)$. Hence, the statements of the theorem follow as soon as we show that \eqref{eq:multiplicity} holds with $\pi(x)$ as in \eqref{eq:POT-multiplicity}. For $u$ sufficiently close to $g(0)$, we have
\[U_{p,0}(u,x)=\{X_0>u+x\}=B_{\frac{1}{2}g^{-1}(u+x)}(\zeta)\]
\[R_{p,0}(u,x)=U_{p,0}(u,x)\setminus U_{p,p}(u,x)=\{X_0>u+x\}\setminus\bigcap_{j=0}^1 f^{-jp}(\{X_0>u+x\})\]
By (R2), $\{X_0>u+x\}$ and $\bigcap_{j=0}^1 f^{-jp}(\{X_0>u+x\})$ are both intervals, and
\[\p(R_{p,0}(u,x))=\p(X_0>u+x)-(1-\theta)\p(X_0>u+x)=\theta\p(X_0>u+x)\]

Let $(u_n)_n$ be a normalizing sequence of levels satisfying \eqref{eq:un} such that $\lim_{n\to\infty}u_n=g(0)$. Given the assumptions \eqref{eq:density} and (R1), of regularity of $\p$ and $U(u_n)=\{X_0>u_n\}$ being a ball centred at $\zeta$, respectively, we have
\[\p(X_0>u_n)\sim\l(X_0>u_n)\frac{d\p}{d\l}(\zeta)=g^{-1}(u_n)\frac{d\p}{d\l}(\zeta)\]
\[\p(R_{p,0}(u_n,x))=\theta\p(X_0>u_n+x)\sim\theta\l(X_0>u_n+x)\frac{d\p}{d\l}(\zeta)=\theta g^{-1}(u_n+x)\frac{d\p}{d\l}(\zeta)\]

If there exists some strictly positive function $q:W\to\R$ and some strictly monotone homeomorphism $h$ such that
\[\lim_{u\to g(0)}\frac{g(g^{-1}(u)h(x))-u}{q(u)}=x\]
then, for $a_n=1/q(u_n)$,
\begin{align*}
\lim_{n\to\infty}\frac{\p(R_{p,0}(u_n,x/a_n))}{\p(X_0>u_n)}&=\theta\lim_{n\to\infty}\frac{g^{-1}(u_n+q(u_n)x)}{g^{-1}(u_n)}=\theta\lim_{n\to\infty}\lim_{u\to g(0)}\frac{g^{-1}\left(u_n+q(u_n)\frac{g(g^{-1}(u)h(x))-u}{q(u)}\right)}{g^{-1}(u_n)}
\end{align*}

In particular, for $u=u_n$ we get
\begin{align*}
\lim_{n\to\infty}\frac{\p(R_{p,0}(u_n,x/a_n))}{\p(X_0>u_n)}&=\theta\lim_{n\to\infty}\frac{g^{-1}\left(u_n+q(u_n)\frac{g(g^{-1}(u_n)h(x))-u_n}{q(u_n)}\right)}{g^{-1}(u_n)}=\theta h(x)
\end{align*}

and the probability distribution is $\pi(x)=1-h(x)$. We will analyse each type of behaviour separately.

Type 1: there exists some strictly positive function $p:W\to\R$ such that for all $y\in\R$
\[\lim_{s\to g(0)}\frac{g^{-1}(s+y q(s))}{g^{-1}(s)}=\e^{-y}\]

Then, we have
\begin{align*}
\lim_{u\to g(0)}\frac{g^{-1}(u-\log(x)q(u))}{g^{-1}(u)}&=x\\
\lim_{u\to g(0)}\frac{g(g^{-1}(u)x)-u}{q(u)}&=\lim_{u\to g(0)}\frac{g\left(g^{-1}(u)\frac{g^{-1}(u-\log(x)q(u))}{g^{-1}(u)}\right)-u}{q(u)}=-\log(x)
\end{align*}

Let $h(x)=\e^{-x}$, so that $h^{-1}(x)=-\log(x)$. Then, $a_n A_n$ converges in distribution to a compound Poisson process $A$ with intensity $\theta$ and multiplicity d.f. $\pi(x)=1-\e^{-x}$.

Type 2: $g(0)=+\infty$ and there exists $\beta>0$ such that for all $y>0$
\[\lim_{s\to+\infty}\frac{g^{-1}(sy)}{g^{-1}(s)}=y^{-\beta}\]

Then, for $q(u)=u$ we have
\begin{align*}
\lim_{u\to\infty}\frac{g^{-1}(u x^{-1/\beta})}{g^{-1}(u)}&=x\\
\lim_{u\to\infty}\frac{g(g^{-1}(u)x)-u}{q(u)}&=\lim_{u\to\infty}\frac{g\left(g^{-1}(u)\frac{g^{-1}(u x^{-1/\beta})}{g^{-1}(u)}\right)}{u}-1=x^{-1/\beta}-1
\end{align*}

Let $h(x)=(1+x)^{-\beta}$, so that $h^{-1}(x)=x^{-1/\beta}-1$. Then, $a_n A_n$ converges in distribution to a compound Poisson process $A$ with intensity $\theta$ and multiplicity d.f. $\pi=1-(1+x)^{-\beta}$.

Type 3: $g(0)=D<+\infty$ and there exists $\gamma>0$ such that for all $y>0$
\[\lim_{s\to 0}\frac{g^{-1}(D-sy)}{g^{-1}(D-s)}=y^\gamma\]

Then, for $q(u)=D-u$ we have
\begin{align*}
\lim_{u\to D}\frac{g^{-1}(D-(D-u)x^{1/\gamma})}{g^{-1}(u)}&=x\\
\lim_{u\to D}\frac{g(g^{-1}(u)x)-u}{q(u)}&=\lim_{u\to D}\frac{g\left(g^{-1}(u)\frac{g^{-1}(D-(D-u)x^{1/\gamma})}{g^{-1}(u)}\right)-u}{D-u}=1-x^{1/\gamma}
\end{align*}

Let $h(x)=(1-x)^\gamma$, so that $h^{-1}(x)=1-x^{1/\gamma}$. Then, $a_n A_n$ converges in distribution to a compound Poisson process $A$ with intensity $\theta$ and multiplicity d.f. $\pi=1-(1-x)^\gamma$. \end{proof}

\subsection{Computation of the limit of AOT MREPP for specific systems}
\label{subsec:AOT}In the case of AOT MREPP is technically much harder to compute the multiplicity distribution of the liming compound Poisson process. In order to write an explicit formula for it we need to assume a specific backward contraction in a neighbourhood of the repelling periodic point, rather than an approximate rate like in the previous cases.  

\begin{theorem}
\label{thm:AOT}
Let $f:\X\to\X$ be a system admitting a first return induced map $F_B:B\to B$, with $B\subset \X$ and such that $F_B$ has summable decay of correlations against $L^1$ observables \ie for all $\phi\in\mathcal C_1$ and $\psi\in L^1$, then $\cv(\phi,\psi,n)\leq \rho_n$, with $\sum_{n\geq 1}\rho_n<\infty$. Assume that for every $\zeta$, for all balls $B_\eps(\zeta)$ and annuli $B_{\eps_1}(\zeta)\setminus B_{\eps_2}(\zeta)$, with $0<\eps,\eps_1<\eps_2$, then $\I_{ B_\eps(\zeta)}\in\mathcal C_1$, $\I_{B_{\eps_1}(\zeta)\setminus B_{\eps_2}(\zeta)}\in\mathcal C_1$ and their norms are uniformly bounded above. 

Let $X_0,X_1,\ldots$ be given by \eqref{eq:def-stat-stoch-proc-DS} and $(u_n)_{n\in\N}$ be a sequence satisfying \eqref{eq:un}. Assume that $\varphi$ and $\p$ are such that Assumptions 1--3 hold, where $\zeta\in B$. Additionally, suppose that
\begin{itemize}
\item $\zeta$ is a periodic repelling point of period $p$
\item for some $M>1$,
\[\dist(f^p(x),\zeta)=M\dist(x,\zeta)\]
for all $x$ in a neighbourhood of $\zeta$. One example of such a dynamical system is, for instance, $f:t\mapsto mt \mod 1$, with $m\in \{2, 3, \ldots\}$ (in this case $M=m^p$)
\item There exists some strictly positive function $q:W\to\R$ and some strictly monotone homeomorphism $h_k$ such that
\begin{equation}
\label{eq:def-h}
\lim_{u\to g(0)}\frac{g_{\kappa,u}(g^{-1}(u)h_\kappa(x))}{q(u)}=x, \, \forall\kappa\in\N_0
\end{equation}
where $g_{\kappa,u}(x):=\sum_{i=0}^\kappa(g(M^ix)-u)$. As we will see, this holds when $g$ has the form given in Assumption 2.
\end{itemize}

Then, for $a_n=q(u_n)^{-1}$ the AOT MREPP $a_n A_n$ converges in distribution to a compound Poisson process with intensity $\theta=1-1/M$ and multiplicity d.f. $\pi$ given by
\[\pi(x)=1-\lim_{n\to\infty}h_{\kappa(u_n,q(u_n)x)}(x)\]
where $\kappa=\kappa(u,x)$ is the only integer such that $x\in\left[g_{\kappa,u}\left(\frac{g^{-1}(u)}{M^\kappa}\right),g_{\kappa,u}\left(\frac{g^{-1}(u)}{M^{\kappa+1}}\right)\right)$.

\end{theorem}

\begin{remark}

In particular, when $g$ is one of the three examples given for each type, the multiplicity d.f. can be computed as shown in the following table:

\begin{tabular}{c|p{11cm}}
 Examples of $g(x)$ & Respective distribution $\pi(x)$ \\
\hline
\rule{0pt}{5ex}
 $-\log(x)$ & $1-(\sqrt{M})^{-\lfloor\frac{\sqrt{1+8x/\log M}-1}{2}\rfloor}\e^{-\frac{x}{\lfloor\frac{\sqrt{1+8x/\log M}-1}{2}\rfloor+1}}$ \\
\rule{0pt}{6ex}
 $x^{-1/\alpha}$, $\alpha>0$ & $1-\left(\frac{1-M^{-1/\alpha}}{1-M^{-(\kappa(x)+1)/\alpha}}\right)^{-\alpha}(\kappa(x)+1+x)^{-\alpha}$
where $\kappa=\kappa(x)$ is the only integer such that $\frac{M^{\kappa/\alpha}-M^{-1/\alpha}}{1-M^{-1/\alpha}}\leq \kappa+1+x<\frac{M^{(\kappa+1)/\alpha}-1}{1-M^{-1/\alpha}}$ \\
\rule{0pt}{6ex}
 $D-x^{1/\alpha}$, $D\in\R$, $\alpha>0$ & $1-\left(\frac{1-M^{1/\alpha}}{1-M^{(\kappa(x)+1)/\alpha}}\right)^\alpha(\kappa(x)+1-x)^\alpha$ where $\kappa=\kappa(x)$ is the only integer such that $\frac{1-M^{-(\kappa+1)/\alpha}}{M^{1/\alpha}-1}<\kappa+1-x\leq\frac{M^{1/\alpha}-M^{-\kappa/\alpha}}{M^{1/\alpha}-1}$\\
\end{tabular}
\end{remark}

\begin{proof}
By Theorem~\ref{thm:mpp_ret_orig} we only need to prove the result for $F_B$ since then if follows for $f$. If $R_{p,0}(u_n,x)=B_{\eps_1}(\zeta)\setminus B_{\eps_2}(\zeta)$ for some $\eps_1,\eps_2>0$ then $\I_{R_{p,0}(u_n,x)}\in\mathcal C_1$ and $\|\I_{R_{p,0}(u_n,x)}\|_{\mathcal C_1}\leq C$ for every $x>0$ and $n\in \N$. Moreover, as in the proof of Theorem~2 of \cite{FFT13}, using the Hartman-Grobman theorem, one can easily check  that $\lim_{n\to\infty}R(Q_{p,0}^0(u_n))=\infty$. Then all hypothesis of Theorem~\ref{thm:check-D-D'} are satisfied and in conclusion conditions $\D_p(u_n)^*$ and $\D'_p(u_n)^*$ hold.
By Assumptions 1--3 and the fact that $\zeta$ is a repelling periodic point then (R2) holds and by Proposition~\ref{prop:ULC} so does $U\!LC_p(u_n)$. 

Hence, the statements of the theorem follow as soon as we show that $R_{p,0}(u_n,x)=B_{\eps_1}(\zeta)\setminus B_{\eps_2}(\zeta)$ for some $\eps_1,\eps_2>0$ and that \eqref{eq:multiplicity} holds with $\pi(x)$ as in \eqref{eq:POT-multiplicity}.

For $u\in(0,g(0))$, let
\[g_{\kappa,u}(x):=\sum_{i=0}^\kappa(g(M^ix)-u) \qquad b_\kappa(u):=g_{\kappa,u}\left(\frac{g^{-1}(u)}{M^\kappa}\right)\]

For $j,\kappa\in\N_0$, $t$ sufficiently close to $\zeta$ and $u$ sufficiently close to $g(0)$, we have
\[X_{jp}(t)>u\Leftrightarrow g(2\dist(f^{jp}(t),\zeta))>u\Leftrightarrow g(2 M^j\dist(t,\zeta))>u\Leftrightarrow t\in B_{\frac{g^{-1}(u)}{2 M^j}}(\zeta)\]
\[m_u(\{X_0,X_p,\ldots,X_{\kappa p}\})(t)>x\Leftrightarrow g_{\kappa,u}(2\dist(t,\zeta))>x\Leftrightarrow t\in B_{\frac{1}{2}g_{\kappa,u}^{-1}(x)}(\zeta)\]

Notice that $(b_\kappa(u))_{\kappa\in\N_0}$ is an increasing sequence for any $u\in[0,g(0))$ since $g\left(\frac{g^{-1}(u)}{M^i}\right)>u$ for $i>0$. Moreover, $b_0(u)=0$ and $b_{\kappa+1}(u)=g_{\kappa+1,u}\left(\frac{g^{-1}(u)}{M^{\kappa+1}}\right)=g_{\kappa,u}\left(\frac{g^{-1}(u)}{M^{\kappa+1}}\right)$.

Then,
\[x\geq b_{\kappa+1}(u)\Leftrightarrow g_{\kappa,u}^{-1}(x)\leq\frac{g^{-1}(u)}{M^{\kappa+1}}\Leftrightarrow U^\kappa_{p,0}(u,x)=\emptyset\]

\[x\leq b_\kappa(u)\Leftrightarrow g_{\kappa,u}^{-1}(x)\geq\frac{g^{-1}(u)}{M^\kappa}\Leftrightarrow U^\kappa_{p,0}(u,x)=B_{\frac{g^{-1}(u)}{2M^\kappa}}(\zeta)\setminus B_{\frac{g^{-1}(u)}{2M^{\kappa+1}}}(\zeta)\]

\[b_\kappa(u)\leq x\leq b_{\kappa+1}(u)\Leftrightarrow\frac{g^{-1}(u)}{M^{\kappa+1}}\leq g_{\kappa,u}^{-1}(x)\leq\frac{g^{-1}(u)}{M^\kappa}\Leftrightarrow U^\kappa_{p,0}(u,x)=B_{\frac{1}{2}g_{\kappa,u}^{-1}(x)}(\zeta)\setminus B_{\frac{g^{-1}(u)}{2M^{\kappa+1}}}(\zeta)\]

Since $(b_\kappa(u))_\kappa$ is an increasing sequence, there is at most one $\kappa=\kappa(u,x)$ such that $x\in[b_\kappa(u),b_{\kappa+1}(u))$. Notice that
\[b_\kappa(u)\leq x<b_{\kappa+1}(u)\Leftrightarrow\frac{g^{-1}(u)}{M^{\kappa+1}}<g_{\kappa,u}^{-1}(x)\leq\frac{g^{-1}(u)}{M^\kappa}\Leftrightarrow M^{-\kappa+1}<\frac{g_{\kappa,u}^{-1}(x)}{g^{-1}(u)}\leq M^{-\kappa}\]
\[\Leftrightarrow \kappa\leq -\log_{M}{\frac{g_{\kappa,u}^{-1}(x)}{g^{-1}(u)}}<\kappa+1 \Leftrightarrow \kappa=\lfloor\frac{\log(g^{-1}(u))-\log(g_{\kappa,u}^{-1}(x))}{\log M}\rfloor\]

Hence,
\begin{align*}
U_{p,0}(u,x)&:=\bigcup_{\kappa=0}^{\infty} U^\kappa_{p,0}(u,x)\cup\{\zeta\}=\bigcup_{\kappa<\kappa(u,x)}U^\kappa_{p,0}(u,x)\cup U^{\kappa(u,x)}_{p,0}(u,x)\cup\bigcup_{\kappa>\kappa(u,x)}U^\kappa_{p,0}(u,x)\cup\{\zeta\}\\
&=B_{\frac{1}{2}G_u(x)}(\zeta)\setminus B_{\frac{g^{-1}(u)}{2M^{\kappa(u,x)+1}}}(\zeta)\cup\bigcup_{\kappa=\kappa(u,x)+1}^{\infty}\left(B_{\frac{g^{-1}(u)}{2M^\kappa}}(\zeta)\setminus B_{\frac{g^{-1}(u)}{2M^{\kappa+1}}}(\zeta)\right)\cup\{\zeta\}\\
&=B_{\frac{1}{2}G_u(x)}(\zeta)
\end{align*}

where $G_u(x)=g_{\kappa(u,x),u}^{-1}(x)$. Now, we note that
\[U_{p,0}(u,x)\cap U_{p,p}(u,x)=\bigcup_{\kappa=0}^{\infty} U^\kappa_{p,0}(u,x)\cap \bigcup_{\kappa=0}^{\infty} U^\kappa_{p,p}(u,x)\cup\{\zeta\}\]

and, for $u$ sufficiently close to $g(0)$, $U^i_{p,0}(u,x)\cap U^j_{p,p}(u,x)\neq\emptyset$ only when $i=j+1$, so
\begin{align*}
&U_{p,0}(u,x)\cap U_{p,p}(u,x)=\bigcup_{\kappa=0}^{\infty}\left(U^{\kappa+1}_{p,0}(u,x)\cap U^\kappa_{p,p}(u,x)\right)\cup\{\zeta\}\\
&=\bigcup_{\kappa=0}^{\infty}\{X_0>u, X_p>u, \ldots, X_{(\kappa+1)p}>u, X_{(\kappa+2)p}\leq u,m_{u}(\{X_p,\ldots,X_{(\kappa+1)p}\})>x\}\cup\{\zeta\}
\end{align*}

Then, for $\kappa\in\N_0$ and $t$ sufficiently close to $\zeta$, we have
\[m_{u}(\{X_p,\ldots,X_{(\kappa+1)p}\})(t)>x\Leftrightarrow g_{\kappa,u}(2M \dist(t,\zeta))>x\Leftrightarrow t\in B_{\frac{g_{\kappa,u}^{-1}(x)}{2M}}(\zeta)\]

\[B_{\frac{g^{-1}(u)}{2M^{\kappa+1}}}(\zeta)\setminus B_{\frac{g^{-1}(u)}{2M^{\kappa+2}}}(\zeta)\cap B_{\frac{g_{\kappa,u}^{-1}(x)}{2M}}(\zeta)=\begin{cases}
\emptyset & \text{if}\quad \kappa<\kappa(u,x)\\
B_{\frac{G_u(x)}{2M}}(\zeta)\setminus B_{\frac{g^{-1}(u)}{2M^{\kappa(u,x)+2}}}(\zeta) & \text{if}\quad \kappa=\kappa(u,x)\\
B_{\frac{g^{-1}(u)}{2M^{\kappa+1}}}(\zeta)\setminus B_{\frac{g^{-1}(u)}{2M^{\kappa+2}}}(\zeta) & \text{if}\quad \kappa>\kappa(u,x)
\end{cases}\]

Hence, for $u$ sufficiently close to $g(0)$,
\begin{align*}
U_{p,0}(u,x)&\cap U_{p,p}(u,x)=\bigcup_{\kappa=0}^{\infty}\left(B_{\frac{g^{-1}(u)}{2M^{\kappa+1}}}(\zeta)\setminus B_{\frac{g^{-1}(u)}{2M^{\kappa+2}}}(\zeta)\cap B_{\frac{g_{\kappa,u}^{-1}(x)}{2M}}(\zeta)\right)\cup\{\zeta\}\\
&=B_{\frac{G_u(x)}{2M}}(\zeta)\setminus B_{\frac{g^{-1}(u)}{2M^{\kappa(u,x)+2}}}(\zeta)\cup\bigcup_{\kappa=\kappa(u,x)+1}^{\infty}\left(B_{\frac{g^{-1}(u)}{2M^{\kappa+1}}}(\zeta)\setminus B_{\frac{g^{-1}(u)}{2M^{\kappa+2}}}(\zeta)\right)\cup\{\zeta\}=B_{\frac{G_u(x)}{2M}}(\zeta)
\end{align*}
and
\begin{align*}
R_{p,0}(u,x)&=U_{p,0}(u,x)\setminus (U_{p,0}(u,x)\cap U_{p,p}(u,x))=B_{\frac{1}{2}G_u(x)}(\zeta) \setminus B_{\frac{G_u(x)}{2M}}(\zeta)
\end{align*}

Let $(u_n)_n$ be a normalizing sequence of levels satisfying \eqref{eq:un} such that $\lim_{n\to\infty}u_n=g(0)$. Given the assumptions \eqref{eq:density} and (R1), of regularity of $\p$ and $U(u_n)=\{X_0>u_n\}$ being a ball centred at $\zeta$, respectively, we have
\[\p(X_0>u_n)\sim\l(X_0>u_n)\frac{d\p}{d\l}(\zeta)=g^{-1}(u_n)\frac{d\p}{d\l}(\zeta)\]
\begin{align*}
\p(R_{p,0}(u_n,x))&=\p\left(B_{\frac{1}{2}G_{u_n}(x)}(\zeta)\right)-\p\left(B_{\frac{G_{u_n}(x)}{2M}}(\zeta)\right)\\
&\sim\l\left(B_{\frac{1}{2}G_{u_n}(x)}(\zeta)\right)\frac{d\p}{d\l}(\zeta)-\l\left(B_{\frac{G_{u_n}(x)}{2M}}(\zeta)\right)\frac{d\p}{d\l}(\zeta)\\
&=\left(G_{u_n}(x)-\frac{G_{u_n}(x)}{M}\right)\frac{d\p}{d\l}(\zeta)=\theta G_{u_n}(x)\frac{d\p}{d\l}(\zeta)
\end{align*}
where $\theta=1-1/M$.

If there exists some strictly positive function $q:W\to\R$ and some strictly monotone homeomorphism $h_k$ such that
\[\lim_{u\to g(0)}\frac{g_{\kappa,u}(g^{-1}(u)h_\kappa(x))}{q(u)}=x, \, \forall\kappa\in\N_0\]
then, for $a_n=1/q(u_n)$,
\begin{align*}
\lim_{n\to\infty}\frac{\p(R_{p,0}(u_n,x/a_n))}{\p(X_0>u_n)}&=\lim_{n\to\infty}\frac{\theta G_{u_n}(q(u_n)x)}{g^{-1}(u_n)}=\theta\lim_{n\to\infty}\frac{g_{\kappa(u_n,q(u_n)x),u_n}^{-1}(q(u_n)x)}{g^{-1}(u_n)}\\
&=\theta\lim_{n\to\infty}\lim_{u\to g(0)}\frac{g_{\kappa(u_n,q(u_n)x),u_n}^{-1}\left(q(u_n)\frac{g_{\kappa,u}(g^{-1}(u)h_\kappa(x))}{q(u)}\right)}{g^{-1}(u_n)}, \, \forall\kappa\in\N_0
\end{align*}

In particular, for $u=u_n$ and $\kappa=\kappa(u_n,q(u_n)x)$, we get
\begin{align*}
\lim_{n\to\infty}\frac{\p(R_{p,0}(u_n,x/a_n))}{\p(X_0>u_n)}&=\theta\lim_{n\to\infty}\frac{g_{\kappa(u_n,q(u_n)x),u_n}^{-1}\left(q(u_n)\frac{g_{\kappa(u_n,q(u_n)x),u_n}(g^{-1}(u_n)h_{\kappa(u_n,q(u_n)x)}(x))}{q(u_n)}\right)}{g^{-1}(u_n)}\\
&=\theta\lim_{n\to\infty}h_{\kappa(u_n,q(u_n)x)}(x)
\end{align*}

and the probability distribution is given by
\[\pi(x)=1-\lim_{n\to\infty}h_{\kappa(u_n,q(u_n)x)}(x)\]

Now, we will analyse each type of behaviour separately.

Type 1: there exists some strictly positive function $q:W\to\R$ such that for all $y\in\R$
\[\lim_{s\to g(0)}\frac{g^{-1}(s+y q(s))}{g^{-1}(s)}=\e^{-y}\]

Then, we have
\[\lim_{u\to g(0)}\frac{g^{-1}(u-\log(x)q(u))}{g^{-1}(u)}=x \Longrightarrow\lim_{u\to g(0)}\frac{g(g^{-1}(u)x)-u}{q(u)}=-\log(x)\]
\begin{align*}
\lim_{u\to g(0)}\frac{g_{\kappa,u}(g^{-1}(u)x)}{q(u)}&=\lim_{u\to g(0)}\frac{\sum_{i=0}^\kappa(g(M^i g^{-1}(u)x)-u)}{q(u)}=-\sum_{i=0}^\kappa\log(x M^i)\\
&=-\sum_{i=0}^\kappa(\log(x)+i\log M)=-(\kappa+1)\log(x)-\frac{\kappa(\kappa+1)}{2}\log M\\
&=-(\kappa+1)(\log(x)+\kappa\log\sqrt{M})
\end{align*}

Let $h_\kappa(x)=\e^{-\frac{x}{\kappa+1}-\kappa\log\sqrt{M}}$, so that $h_\kappa^{-1}(x)=-(\kappa+1)(\log(x)+\kappa\log\sqrt{M})$. Then, $a_n A_n$ converges in distribution to a compound Poisson process $A$ with intensity $\theta=1-1/M$ and multiplicity d.f. $\pi$ given by
\[\pi(x)=1-\lim_{n\to\infty}h_{\kappa(u_n,q(u_n)x)}(x)=1-\lim_{n\to\infty}\e^{-\frac{x}{\kappa(u_n,q(u_n)x)+1}-\kappa(u_n,q(u_n)x)\log\sqrt{M}}\]

If $g(x)=-\log(x)$, then
\begin{align*}
g_{\kappa,u}(x)&=\sum_{i=0}^\kappa(-\log(M^ix)-u)=-(\kappa+1)(\log(x)+u+\kappa\log\sqrt{M})\\
b_\kappa(u)&=g_{\kappa,u}\left(\frac{\e^{-u}}{M^\kappa}\right)=\frac{\kappa(\kappa+1)}{2}\log M\\
b_{\kappa+1}(u)&=g_{\kappa,u}\left(\frac{\e^{-u}}{M^{\kappa+1}}\right)=\frac{(\kappa+1)(\kappa+2)}{2}\log M
\end{align*}
Let $\kappa=\kappa(u,x)$ be the only integer such that $x\in[b_\kappa(u),b_{\kappa+1}(u))$, or, equivalently,
\[\frac{\kappa(\kappa+1)}{2}\log M\leq x<\frac{(\kappa+1)(\kappa+2)}{2}\log M\]
Then, $\kappa(u,x)=\lfloor\frac{\sqrt{1+8x/\log M}-1}{2}\rfloor$ and
\[\pi(x)=1-\e^{-\frac{x}{\kappa(u,x)+1}-\kappa(u,x)\log\sqrt{M}}=1-(\sqrt{M})^{-\lfloor\frac{\sqrt{1+8x/\log M}-1}{2}\rfloor}\e^{-\frac{x}{\lfloor\frac{\sqrt{1+8x/\log M}-1}{2}\rfloor+1}}\]

Type 2: $g(0)=+\infty$ and there exists $\alpha>1$ such that for all $y>0$
\[\lim_{s\to+\infty}\frac{g^{-1}(sy)}{g^{-1}(s)}=y^{-\alpha}\]

Then, for $q(u)=u$ we have
\[\lim_{u\to\infty}\frac{g^{-1}(u x^{-1/\alpha})}{g^{-1}(u)}=x\Longrightarrow\lim_{u\to\infty}\frac{g(g^{-1}(u)x)-u}{q(u)}=x^{-1/\alpha}-1\]
\begin{align*}
\lim_{u\to\infty}\frac{g_{\kappa,u}(g^{-1}(u)x)}{q(u)}&=\lim_{u\to\infty}\frac{\sum_{i=0}^\kappa(g(M^i g^{-1}(u)x)-u)}{q(u)}=\sum_{i=0}^\kappa((M^{-1/\alpha})^i x^{-1/\alpha}-1)\\
&=\frac{1-M^{-(\kappa+1)/\alpha}}{1-M^{-1/\alpha}}x^{-1/\alpha}-(\kappa+1)
\end{align*}

Let $h_\kappa(x)=\left(\frac{1-M^{-1/\alpha}}{1-M^{-(\kappa+1)/\alpha}}\right)^{-\alpha}(\kappa+1+x)^{-\alpha}$, so that $h_\kappa^{-1}(x)=\frac{1-M^{-(\kappa+1)/\alpha}}{1-M^{-1/\alpha}}x^{-1/\alpha}-(\kappa+1)$. Then, $a_n A_n$ converges in distribution to a compound Poisson process $A$ with intensity $\theta=1-1/M$ and multiplicity d.f. $\pi$ given by
\[\pi(x)=1-\lim_{n\to\infty}h_{\kappa(u_n,q(u_n)x)}(x)=1-\lim_{n\to\infty}\left(\frac{1-M^{-1/\alpha}}{1-M^{-(\kappa(u_n,u_n x)+1)/\alpha}}\right)^{-\alpha}(\kappa(u_n,u_n x)+1+x)^{-\alpha}\]

If $g(x)=x^{-1/\alpha}$ for some $\alpha>0$, then
{
\fontsize{10}{10}\selectfont
\begin{align*}
g_{\kappa,u}(x)&=\sum_{i=0}^\kappa((M^i x)^{-1/\alpha}-u)=\frac{1-M^{-(\kappa+1)/\alpha}}{1-M^{-1/\alpha}}x^{-1/\alpha}-(\kappa+1)u\\
b_\kappa(u)&=g_{\kappa,u}\left(\frac{u^{-\alpha}}{M^\kappa}\right)=\left(\frac{M^{\kappa/\alpha}-M^{-1/\alpha}}{1-M^{-1/\alpha}}-(\kappa+1)\right)u\\
b_{\kappa+1}(u)&=g_{\kappa,u}\left(\frac{u^{-\alpha}}{M^{\kappa+1}}\right)=\left(\frac{M^{(\kappa+1)/\alpha}-1}{1-M^{-1/\alpha}}-(\kappa+1)\right)u
\end{align*}
Let $\kappa=\kappa(u,u x)$ be the only integer such that $u x\in[b_\kappa(u),b_{\kappa+1}(u))$, or, equivalently,
\[\frac{M^{\kappa/\alpha}-M^{-1/\alpha}}{1-M^{-1/\alpha}}\leq \kappa+1+x<\frac{M^{(\kappa+1)/\alpha}-1}{1-M^{-1/\alpha}}\]
Notice that $\kappa(u,u x)$ does not depend on $u$; hence,
\[\pi(x)=1-\left(\frac{1-M^{-1/\alpha}}{1-M^{-(\kappa(x)+1)/\alpha}}\right)^{-\alpha}(\kappa(x)+1+x)^{-\alpha}\]
where $\kappa(x)=\kappa(u,u x)$.
}

Type 3: $g(0)=D<+\infty$ and there exists $\alpha>0$ such that for all $y>0$
{
\fontsize{10}{10}\selectfont
\[\lim_{s\to 0}\frac{g^{-1}(D-sy)}{g^{-1}(D-s)}=y^\alpha\]
}
Then, for $q(u)=D-u$ we have
{
\fontsize{10}{10}\selectfont
\[\lim_{u\to D}\frac{g^{-1}(D-(D-u)x^{1/\alpha})}{g^{-1}(u)}=x\Longrightarrow\lim_{u\to D}\frac{g(g^{-1}(u)x)-u}{q(u)}=1-x^{1/\alpha}\]
\begin{align*}
\lim_{u\to D}\frac{g_{\kappa,u}(g^{-1}(u)x)}{q(u)}&=\lim_{u\to D}\frac{\sum_{i=0}^\kappa(g(M^i g^{-1}(u)x)-u)}{q(u)}=\sum_{i=0}^\kappa(1-(M^{1/\alpha})^i x^{1/\alpha})\\
&=\kappa+1-\frac{1-M^{(\kappa+1)/\alpha}}{1-M^{1/\alpha}}x^{1/\alpha}
\end{align*}
}
Let $h_\kappa(x)=\left(\frac{1-M^{1/\alpha}}{1-M^{(\kappa+1)/\alpha}}\right)^\alpha(\kappa+1-x)^\alpha$, so that $h_\kappa^{-1}(x)=\kappa+1-\frac{1-M^{(\kappa+1)/\alpha}}{1-M^{1/\alpha}}x^{1/\alpha}$. Then, $a_n A_n$ converges in distribution to a compound Poisson process $A$ with intensity $\theta=1-1/M$ and multiplicity d.f. $\pi$ given by
\[\pi(x)=1-\lim_{n\to\infty}h_{\kappa(u_n,q(u_n)x)}(x)=1-\lim_{n\to\infty}\left(\frac{1-M^{1/\alpha}}{1-M^{(\kappa(u_n,(D-u_n)x)+1)/\alpha}}\right)^\alpha(\kappa(u_n,(D-u_n)x)+1-x)^\alpha\]

If $g(x)=D-x^{1/\alpha}$ for some $D\in\R$ and $\alpha>0$, then
{
\fontsize{10}{10}\selectfont
\begin{align*}
g_{\kappa,u}(x)&=\sum_{i=0}^\kappa(D-(M^i x)^{1/\alpha}-u)=(\kappa+1)(D-u)-\frac{1-M^{(\kappa+1)/\alpha}}{1-M^{1/\alpha}}x^{1/\alpha}\\
b_\kappa(u)&=g_{\kappa,u}\left(\frac{(D-u)^{\alpha}}{M^\kappa}\right)=\left(\kappa+1-\frac{M^{1/\alpha}-M^{-\kappa/\alpha}}{M^{1/\alpha}-1}\right)(D-u)\\
b_{\kappa+1}(u)&=g_{\kappa,u}\left(\frac{(D-u)^{\alpha}}{M^{\kappa+1}}\right)=\left(\kappa+1-\frac{1-M^{-(\kappa+1)/\alpha}}{M^{1/\alpha}-1}\right)(D-u)
\end{align*}
}
Let $\kappa=\kappa(u,(D-u)x)$ be the only integer such that $(D-u)x\in[b_\kappa(u),b_{\kappa+1}(u))$, or, equivalently,
\[\frac{1-M^{-(\kappa+1)/\alpha}}{M^{1/\alpha}-1}<\kappa+1-x\leq\frac{M^{1/\alpha}-M^{-\kappa/\alpha}}{M^{1/\alpha}-1}\]
Notice that $\kappa(u,(D-u)x)$ does not depend on $u$; hence,
\[\pi(x)=1-\left(\frac{1-M^{1/\alpha}}{1-M^{(\kappa(x)+1)/\alpha}}\right)^\alpha(\kappa(x)+1-x)^\alpha\]
where $\kappa(x)=\kappa(u,(D-u)x)$.
\end{proof}

\section{Convergence of marked rare events point processes}
\label{sec:proof-convergence}

This section is dedicated to the proof of Theroem~\ref{thm:convergence}. The argument follows the same thread as the one in the proof of \cite[Theorem~1]{FFT13} but it is much more evolved due to the sophistication associated to the MREPP and the degree of generalisation and cases addressed (like allowing multiple maxima and including the absence of clustering). One of the highlights of the proof below is the way we handled the gap created by considering general distributions for the marking of clusters, when compared to the distributions defined on the integer numbers as in the setting of \cite{FFT13}, which significantly simplified the proof of \cite[Theorem~1]{FFT13}. The major step to overcome this difficulty  is made with Proposition~\ref{prop:Laplace}, which is of independent interest since it provides a formula to compute the Laplace transform of multiple random variables with general distributions, possibly diffuse with respect to Lebesgue. 

We start with a lemma which says that the probability of not entering $U_{p,0}(u,x)$ can be approximated by the probability of not entering $R_{p,0}(u,x)$ during the same period of time.
\begin{lemma}
\label{Lem:disc-ring}
For any $p\in\N_0$, $s\in\N$, $x\geq 0$ and $u>0$ we have
\[\left|\p\big(\II_{p,0,s}(u,x)\big)-\p\big(\RR_{p,0,s}(u,x)\big)\right|\leq p\p(U_{p,0}(u,x))\]
\end{lemma}
\begin{proof}
For $p=0$ this is trivial since $U_{0,i}(u,x)=R_{0,i}(u,x)$.
For $p>0$, first observe that since $R_{p,i}(u,x)\subset U_{p,i}(u,x)$ we have $\II_{p,0,s}(u,x)\subset\RR_{p,0,s}(u,x)$. Next, observe that if $\RR_{p,0,s}(u,x)\setminus\II_{p,0,s}(u,x)$ occurs, then we may choose $j\in\{0,1,\ldots s-1\}$ such that $X_j\in U_{p,0}(u,x)$. But since $\RR_{p,0,s}(u,x)$ does occur, we must have $X_{j+j_1}\in U_{p,0}(u,x)$ for some $1\leq j_1\leq p$, otherwise $R_{p,j}(u,x)$ would occur. Similarly, if $j+j_1<s$ then $X_{j+j_1+j_2}\in U_{p,0}(u,x)$ for some $1\leq j_2\leq p$ and so on. We conclude that $X_i\in U_{p,0}(u,x)$ for some $i\in\{s-p,\ldots,s-1\}$ and this means that
\[\RR_{p,0,s}(u,x)\setminus\II_{p,0,s}(u,x)\subset \bigcup_{i=s-p}^{s-1} U_{p,i}(u,x)\]
Hence, by stationarity, it follows that
\[\left|\p\big(\II_{p,0,s}(u,x)\big)-\p\big(\RR_{p,0,s}(u,x)\big)\right|=\p\left(\RR_{p,0,s}(u,x)\setminus \II_{p,0,s}(u,x)\right)\leq p\p(U_{p,0}(u,x))\]
\end{proof}
Next we give an approximation for the probability of not entering $R_{p,0}(u,x)$ during a certain period of time.
\begin{lemma}
\label{lem:no-entrances-ring}
For any $p\in\N_0$, $s\in\N$, $x\geq 0$ and $u>0$ we have
\[\left|\p\big(\RR_{p,0,s}(u,x)\big)-\big(1-s\p(R_{p,0}(u,x))\big)\right|\leq s\sum_{j=p+1}^{s-1}\p(Q_{p,0}^0(u)\cap U_{p,j}(u,x))\]
\end{lemma}
\begin{proof}
Since $(\RR_{p,0,s}(u,x))^c=\cup_{i=0}^{s-1} R_{p,i}(u,x)$ it is clear that
\[\left|1-\p(\RR_{p,0,s}(u,x))-s\p(R_{p,0}(u,x))\right|\leq \sum_{i=0}^{s-1}\sum_{j=i+p+1}^{s-1}\p(R_{p,i}(u,x)\cap R_{p,j}(u,x)).\]
If $p>0$, the result now follows by stationarity plus the two following facts: $R_{p,j}(u,x)\subset U_{p,j}(u,x)$ and the fact that between two entrances to $R_{p,0}(u,x)$, at times $i$ and $j$, there must have existed an escape, \ie an entrance in $Q_{p,0}^0(u)$ (otherwise, an entrance to $R_{p,0}(u,x)$ and therefore to $U_{p,0}(u,x)$ at time $j$ would imply an entrance to $U_{p,0}(u,x)$ at some earlier time $i_1$ for $i+1\leq i_1\leq i+p$ which would contradict the entrance to $R_{p,0}(u,x)$ at time $i$).

If $p=0$, the result follows by stationarity plus the two following facts: $R_{0,j}(u,x)=U_{0,j}(u,x)$ and $R_{0,i}(u,x)\subset\{X_i>u\}=Q_{0,i}^0(u)$.
\end{proof}
The next lemma gives an error bound for the approximation of the probability of the process $\aa_u([0,s))$ not exceeding $x$ by the probability of not entering in $R_{p,0}(u,x)$ during the period $[0,s)$. In what follows, we use the notation 
\begin{equation}
\label{eq;A_ab-notation}
\aa_{u,a}^b:=\aa_u([a,b)),\quad \text{$\aa_u$ is as in \eqref{eq:A_u-definition}}
\end{equation}  
\begin{lemma}
\label{lem:entrances-ball-depth}
For any $s\in\N$, $x\geq 0$ and $u>0$ we have
\begin{align*}
\left|\p(\aa_{u,0}^s\leq x)-\p(\II_{p,0,s}(u,x))\right|&\leq (s-p)\sum_{j=p+1}^{s-1}\p\left(Q_{p,0}^0(u)\cap\{X_j>u\}\right)\\
&+\sum_{\kappa=1}^{\lfloor s/p\rfloor}\kappa p\p(U^{\kappa}_{p,0}(u,x))+s\sum_{\kappa=\lfloor s/p\rfloor+1}^{\infty}\p(U^{\kappa}_{p,0}(u,x))
\end{align*}
if $p>0$, and in case $p=0$ we have
\[\left|\p(\aa_{u,0}^s\leq x)-\p(\II_{0,0,s}(u,x))\right|\leq s\sum_{j=1}^{s-1}\p\left(X_0>u,X_j>u\}\right)=s\sum_{j=1}^{s-1}\p\left(Q_{0,0}^0(u)\cap\{X_j>u\}\right).\]
\end{lemma}

\begin{proof}
If $p>0$, we start by observing that
\begin{multline*}
A_{0,s}(u,x):=\left\{\aa_{u,0}^s\leq x\right\}\cap \left(\II_{p,0,s}(u,x)\right)^c \subset\\  \bigcup_{i=s-p}^{s-1} U^1_{p,i}(u,x) \cup \bigcup_{i=s-2p}^{s-1} U^2_{p,i}(u,x) \cup %\bigcup_{i=s-3p}^{s-1} U^3_{p,i}(u,x) \cup 
\ldots 
 \cup \bigcup_{i=s-\lfloor s/p\rfloor p}^{s-1} U^{\lfloor s/p\rfloor}_{p,i}(u,x) \cup \bigcup_{i=0}^{s-1}\bigcup_{\kappa>\lfloor s/p\rfloor} U^{\kappa}_{p,i}(u,x)
\end{multline*}
since $\bigcup_{i=0}^{s-\kappa p-1} U^{\kappa}_{p,i}(u,x)\subset \left\{\aa_{u,0}^s>x\right\}$ for any $\kappa\leq\lfloor s/p\rfloor$. So, by stationarity,
\[\p\left(A_{0,s}(u,x)\right)\leq \sum_{\kappa=1}^{\lfloor s/p\rfloor}\kappa p\p(U^{\kappa}_{p,0}(u,x))+s\sum_{\kappa=\lfloor s/p\rfloor+1}^{\infty}\p(U^{\kappa}_{p,0}(u,x))\]

Now, we note that
\[B_{0,s}(u,x):=\left\{\aa_{u,0}^s>x\right\}\cap \II_{p,0,s}(u,x)\subset \bigcup_{i=0}^{s-p-1}\bigcup_{j>i+p}^{s-1} Q_{p,i}^0(u)\cap \{X_j>u\}.\]
This is because no entrance in $U_{p,0}(u,x)$ during the time period $0,\ldots,s-1$ implies that there must be at least two distinct clusters during the time period $0,\ldots,s-1$. Since each cluster ends with an escape, \ie an entrance in $Q_{p,0}^0(u)$, then this must have happened at some moment $i\in\{0,\ldots,s-p-1\}$ which was then followed by another exceedance at some subsequent instant $j>i$ where a new cluster is begun.

Consequently, by stationarity, we have
\[\p\left(B_{0,s}(u,x)\right)\leq (s-p)\sum_{j=p+1}^{s-1}\p\left(Q_{p,0}^0(u)\cap \{X_j>u\}\right)\]

If $p=0$, we start by observing that $\{\aa_{u,0}^s\leq x\}\subset\II_{0,0,s}(u,x)$. Then, we note that
\[\{\aa_{u,0}^s>x\}\cap\II_{0,0,s}(u,x)\subset \bigcup_{i=0}^{s-1}\bigcup_{j=i+1}^{s-1} \{X_i>u\}\cap\{X_j>u\}\]
This is because no entrance in $U_{0,0}(u,x)$ during the time period $0,\ldots,s-1$ implies that there must be at least two exceedances during the time period $0,\ldots,s-1$.

Consequently, by stationarity, we have
\[\left|\p\left(\aa_{u,0}^s\leq x\big)-\p(\II_{0,0,s}(u,x)\right)\right|=\p\left(\{\aa_{u,0}^s>x\}\cap\II_{0,0,s}(u,x)\right)\leq s\sum_{j=1}^{s-1}\p(X_0>u,X_j>u)\]
\end{proof}
As a consequence we obtain an approximation for the Laplace transform of $\aa_{u,o}^s$. 
\begin{corollaryP}
\label{cor:exponential}
%Assuming that $\varphi$ achieves a global maximum at the point $\zeta$, 
For any $p\in\N_0$, $s\in\N$, $y\geq 0$, $a>0$ and $u>0$ sufficiently close to $u_F=\sup \varphi$ we have
\begin{align*}
\Bigg|\E\left(\e^{-y a\aa_{u,0}^s}\right)-&\left(1-s\int_0^{\infty}y\e^{-yx}\p(R_{p,0}(u,x/a))dx\right)\Bigg|\\
&\quad \leq 2s\sum_{j=p+1}^{s-1}\p\left(Q_{p,0}^0(u)\cap\{X_j>u\}\right)+\int_0^{\infty}y\e^{-yx}\delta_{p,s,u}(x/a)dx,
\end{align*}
where $\delta_{p,s,u}(x/a)$ is as in \eqref{eq:delta-definition}.
\end{corollaryP}

\begin{proof}
Using Lemmas~\ref{Lem:disc-ring}-\ref{lem:entrances-ball-depth}, for every $x>0$ and $p>0$ we have
\begin{align*}
\big|\p&\big(\aa_{u,0}^s\leq x\big)-(1-s\p(R_{p,0}(u,x)))\big|\leq \left|\p\big(\aa_{u,0}^s\leq x\big)-\p(\II_{p,0,s}(u,x)\big)\right|\\
&\quad+\left|\p\big(\II_{p,0,s}(u,x)\big)-\p\big(\RR_{p,0,s}(u,x)\big)\right|+\left|\p\big(\RR_{p,0,s}(u,x)\big)-\big(1-s\p(R_{p,0}(u,x))\big)\right|\\
&\leq (s-p)\sum_{j=p+1}^{s-1}\p\left(Q_{p,0}^0(u)\cap \{X_j>u\}\right)+\sum_{\kappa=1}^{\lfloor s/p\rfloor}\kappa p\p(U^{\kappa}_{p,0}(u,x))+s\sum_{\kappa=\lfloor s/p\rfloor+1}^{\infty}\p(U^{\kappa}_{p,0}(u,x))\\
&\quad+p\p(U_{p,0}(u,x))+s\sum_{j=p+1}^{s-1}\p\left(Q_{p,0}^0(u)\cap U_{p,j}(u,x)\right)\\
&\leq 2s\sum_{j=p+1}^{s-1}\p\left(Q_{p,0}^0(u)\cap \{X_j>u\}\right)+\delta_{p,s,u}(x)
\end{align*}
When $p=0$, we have
\begin{align*}
\big|\p\big(\aa_{u,0}^s\leq x\big)-&(1-s\p(R_{0,0}(u,x)))\big|\leq \left|\p\big(\aa_{u,0}^s\leq x\big)-\p(\II_{0,0,s}(u,x)\big)\right|\\
&+\left|\p\big(\II_{0,0,s}(u,x)\big)-\p\big(\RR_{0,0,s}(u,x)\big)\right|+\left|\p\big(\RR_{0,0,s}(u,x)\big)-\big(1-s\p(R_{0,0}(u,x))\big)\right|\\
&\leq s\sum_{j=1}^{s-1}\p\left(Q_{0,0}^0(u)\cap \{X_j>u\}\right)+s\sum_{j=1}^{s-1}\p\left(Q_{0,0}^0(u)\cap U_{0,j}(u,x)\right)\\
&\leq 2s\sum_{j=1}^{s-1}\p\left(Q_{0,0}^0(u)\cap \{X_j>u\}\right)+\delta_{0,s,u}(x).
\end{align*}
Since $\p(\aa_{u,0}^s<0)=0$, using integration by parts we have
\begin{align*}
\E\left(\e^{-y a\aa_{u,0}^s}\right)&=\e^{-y.0}\p(\aa_{u,0}^s=0)+\int_0^{\infty} \e^{-yx} d\p(\aa_{u,0}^s\leq x/a)\\
&=\p(\aa_{u,0}^s=0)+\lim_{x\rightarrow\infty}\left[\e^{-yx}\p(\aa_{u,0}^s\leq x/a)-\e^{-y.0}\p(\aa_{u,0}^s\leq 0)\right]-\int_0^{\infty}\p(\aa_{u,0}^s\leq x/a) d\e^{-yx}\\
&=\p(\aa_{u,0}^s=0)-\p(\aa_{u,0}^s\leq 0)-\int_0^{\infty}(-y\e^{-yx})\p(\aa_{u,0}^s\leq x/a)dx=\int_0^{\infty}y\e^{-yx}\p(\aa_{u,0}^s\leq x/a)dx
\end{align*}

\begin{align*}
\text{Then,}\;\Bigg|\E\left(\e^{-y a\aa_{u,0}^s}\right)&-\left(1-s\int_0^{\infty}y\e^{-yx}\p(R_{p,0}(u,x/a))dx\right)\Bigg|\\
&=\left|\int_0^{\infty}y\e^{-yx}\p(\aa_{u,0}^s\leq x/a)dx-\int_0^{\infty}y\e^{-yx}(1-s\p(R_{p,0}(u,x/a)))dx\right|\\
&\leq \int_0^{\infty} y\e^{-yx}\left[
2s\sum_{j=p+1}^{s-1}\p\left(Q_{p,0}^0(u)\cap \{X_j>u\}\right)+\delta_{p,s,u}(x/a)\right]dx\\
&=2s\sum_{j=p+1}^{s-1}\p\left(Q_{p,0}^0(u)\cap \{X_j>u\}\right)+\int_0^{\infty}y\e^{-yx}\delta_{p,s,u}(x/a)dx
\end{align*}
\end{proof}

Next result gives the main induction tool to build the proof of Theorem~\ref{thm:convergence}.
\begin{lemma}
\label{prop:main-step}
Let $p\in\N_0$, $s,t,\varsigma\in\N$ and consider $x_1\in\R^+_0$, $\underline{x}=(x_2,\ldots,x_\varsigma)\in (\R^+_0)^{\varsigma-1}$, $s+t-1<a_2<b_2<a_3<\ldots<b_\varsigma\in\N_0$. For $u>0$ sufficiently close to $u_F=\varphi(\zeta)$ we have
\begin{align*}
\big|\p(\aa_{u,0}^s\leq x_1, \aa_{u,a_2}^{b_2}&\leq x_2, \ldots, \aa_{u,a_\varsigma}^{b_\varsigma}\leq x_\varsigma)-\p(\aa_{u,0}^s\leq x_1)\p(\aa_{u,a_2}^{b_2}\leq x_2, \ldots, \aa_{u,a_\varsigma}^{b_\varsigma}\leq x_\varsigma)\big|\\&\hspace{0.4cm} \leq s\iota(u,t)+4s\sum_{j=p+1}^{s-1}\p\left(Q_{p,0}^0(u)\cap \{X_j>u\}\right)+2\delta_{p,s,u}(x_1)
\end{align*}
where $\delta_{p,s,u}$ is as in \eqref{eq:delta-definition} and \begin{equation}
\label{eq:def-iota}
\iota(u,t)=\sup_{s\in \N}\max_{i=0,\ldots,s-1}\left\{\left|\p(R_{p,i}(u,x_1))\p\big(\cap_{j=2}^\varsigma \{\aa_{u,a_j}^{b_j}\leq x_j\}\big)-\p\big( \cap_{j=2}^\varsigma \{\aa_{u,a_j}^{b_j}\leq x_j\}\cap R_{p,i}(u,x_1)\big)\right|\right\}.
\end{equation}

\end{lemma}
\begin{proof}
Let 
\begin{align*}
A_{x_1,\underline{x}}&:=\{\aa_{u,0}^s\leq x_1, \aa_{u,a_2}^{b_2}\leq x_2, \ldots, \aa_{u,a_\varsigma}^{b_\varsigma}\leq x_\varsigma\},\;\;\qquad B_{x_1}:=\{\aa_{u,0}^s\leq x_1\}\\
\tilde A_{x_1,\underline{x}} &:=\RR_{p,0,s}(u,x_1)\cap\{ \aa_{u,a_2}^{b_2}\leq x_2, \ldots, \aa_{u,a_\varsigma}^{b_\varsigma}\leq x_\varsigma\},\quad
\tilde B_{x_1}:=\RR_{p,0,s}(u,x_1),\\
D^{\underline{x}}&:=\{\aa_{u,a_2}^{b_2}\leq x_2, \ldots, \aa_{u,a_\varsigma}^{b_\varsigma}\leq x_\varsigma\}.
\end{align*}

If $x_1>0$, by Lemmas~\ref{Lem:disc-ring} and \ref{lem:entrances-ball-depth} we have
\begin{align}
\label{eq:approx1}
\left|\p(B_{x_1})-\p(\tilde B_{x_1})\right|\nonumber &\leq \left|\p(\aa_{u,0}^s\leq x_1) -\p(\II_{p,0,s}(u,x_1))\right|+\left|\p(\II_{p,0,s}(u,x_1))-\p(\RR_{p,0,s}(u,x_1))\right|\nonumber\\
&\leq \left|\p(\{\aa_{u,0}^s\leq x_1\}\triangle \II_{p,0,s}(u,x_1))\right|+\left|\p(\RR_{p,0,s}(u,x_1)\setminus\II_{p,0,s}(u,x_1))\right|\nonumber\\
&\leq (s-p)\sum_{j=p+1}^{s-1}\p\left(Q_{p,0}^0(u)\cap \{X_j>u\}\right)+\sum_{\kappa=1}^{\lfloor s/p\rfloor}\kappa p\p(U^{\kappa}_{p,0}(u,x_1))\nonumber\\
&+s\sum_{\kappa=\lfloor s/p\rfloor+1}^{\infty}\p(U^{\kappa}_{p,0}(u,x_1))+p\p(U_{p,0}(u,x_1))\nonumber\\
&\leq s\sum_{j=p+1}^{s-1}\p\left(Q_{p,0}^0(u)\cap \{X_j>u\}\right)+\delta_{p,s,u}(x_1)
\end{align}
and also
{
\fontsize{10}{10}\selectfont
\begin{align}
\label{eq:approx2}
\left|\p(A_{x_1})-\p(\tilde A_{x_1})\right|\nonumber &\leq \left|\p(\{\aa_{u,0}^s\leq x_1\}\cap D^{\underline{x}})-\p(\II_{p,0,s}(u,x_1)\cap D^{\underline{x}})\right|+\left|\p\left((\RR_{p,0,s}(u,x_1)\setminus\II_{p,0,s}(u,x_1))\cap D^{\underline{x}})\right)\right|\nonumber\\
&\leq \left|\p\left((\{\aa_{u,0}^s\leq x_1\}\triangle \II_{p,0,s}(u,x_1))\cap D^{\underline{x}}\right)\right|+\left|\p\left((\RR_{p,0,s}(u,x_1)\setminus\II_{p,0,s}(u,x_1))\cap D^{\underline{x}})\right)\right|\nonumber\\
&\leq \left|\p(\{\aa_{u,0}^s\leq x_1\}\triangle \II_{p,0,s}(u,x_1))\right|+\left|\p(\RR_{p,0,s}(u,x_1)\setminus\II_{p,0,s}(u,x_1))\right|\nonumber\\
&\leq s\sum_{j=p+1}^{s-1}\p\left(Q_{p,0}^0(u)\cap \{X_j>u\}\right)+\delta_{p,s,u}(x_1)
\end{align}
}

If $x_1=0$, we notice that $\{\aa_{u,0}^s\leq x_1\}=\{\aa_{u,0}^s=0\}=\{X_0\leq u,\ldots,X_{s-1}\leq u\}=\II_{p,0,s}(u,0)$, so estimates \eqref{eq:approx1} and \eqref{eq:approx2} are still valid by Lemma~\ref{Lem:disc-ring}.

Using stationarity and adapting the proof of Lemma~\ref{lem:no-entrances-ring}, it follows that

$\left|\p(\tilde A_{x_1,\underline{x}})-(1-s\p(R_{p,0}(u,x_1)))\p(D^{\underline x_1})\right|\leq Err,$ where

\[Err= \left|s\p(R_{p,0}(u,x_1))\p(D^{\underline{x}})-\sum_{i=0}^{s-1}\p(R_{p,i}(u,x_1)\cap D^{\underline{x}})\right|+s\sum_{j=p+1}^{s-1}\p(Q_{p,0}^0(u)\cap U_{p,j}(u,x_1)).\]

Now, since, by definition of $\iota(u,t)$, 
\begin{multline*}\left|s\p(R_{p,0}(u,x_1))\p(D^{\underline x})-\sum_{i=0}^{s-1}\p(R_{p,i}(u,x_1)\cap D^{\underline x})\right|\\\leq\sum_{i=0}^{s-1}\left|\p(R_{p,i}(u,x_1))\p(D^{\underline x})-\p(R_{p,i}(u,x_1)\cap D^{\underline x})\right|\leq s\iota(u,t),\end{multline*}

we conclude that 
\begin{equation}
\label{eq:approx3}
\left|\p(\tilde A_{x_1,\underline{x}})-(1-s\p(R_{p,0}(u,x_1)))\p(D^{\underline x})\right|\leq s\iota(u,t)+s\sum_{j=p+1}^{s-1}\p(Q_{p,0}^0(u)\cap U_{p,j}(u,x_1)).
\end{equation}

Also, by Lemma~\ref{lem:no-entrances-ring} we have
\begin{equation}
\label{eq:approx4}
\left|\p(\tilde B_{x_1})\p(D^{\underline x})-(1-s\p(R_{p,0}(u,x_1)))\p(D^{\underline x})\right|\leq s\sum_{j=p+1}^{s-1}\p(Q_{p,0}^0(u)\cap U_{p,j}(u,x_1)).
\end{equation}

Putting together the estimates \eqref{eq:approx1},\eqref{eq:approx2}, \eqref{eq:approx3} and \eqref{eq:approx4} we get
\begin{align*}
|\p( & A_{x_1,\underline{x}})-\p(B_{x_1})\p(D^{\underline x})|\leq \left|\p( A_{x_1,\underline{x}})-\p( \tilde A_{x_1,\underline{x}})\right|+\left|\p(\tilde A_{x_1,\underline{x}})-(1-s\p(R_{p,0}(u,x_1)))\p(D^{\underline x})\right|\\
&\hspace{2cm}+\left|\p(\tilde B_{x_1})\p(D^{\underline x})-(1-s\p(R_{p,0}(u,x_1)))\p(D^{\underline x})\right|+\left|\p(B_{x_1})-\p(\tilde B_{x_1})\right|\p(D^{\underline x})\\
&\hspace{2cm}\leq s\iota(u,t)+4s\sum_{j=p+1}^{s-1}\p\left(Q_{p,0}^0(u)\cap \{X_j>u\}\right)+2\delta_{p,s,u}(x_1)
\end{align*}
\end{proof}

Let us consider a function $F:(\R_0^+)^n \to \R$ which is continuous on the right in each variable separately and such that for each $R=(a_1,b_1]\times\ldots\times(a_n,b_n]\subset(\R_0^+)^n$ we have
\[\mu_F(R):=\sum_{c_i\in\{a_i,b_i\}}(-1)^{card\{i\in\{1,\ldots,n\}:c_i=a_i\}}F(c_1,\ldots,c_n)\geq 0\]
Such $F$ is called an \emph{$n$-dimensional Stieltjes measure function} and such $\mu_F$ has a unique extension to the Borel $\sigma$-algebra in $(\R_0^+)^n$, which is called the \emph{Lebesgue-Stieltjes measure} associated to $F$.

For each $I\subset\{1,\ldots,n\}$, let $F_I(\underline{x}):=F(\delta_1 x_1,\ldots,\delta_n x_n)$, where $\delta_i=
\begin{cases}
1 & \text{if $i\in I$}\\
0 & \text{if $i\notin I$}
\end{cases}$

If $F$ is an $n$-dimensional Stieltjes measure function, it is easy to see that $F_I$ is also an $n$-dimensional Stieltjes measure function, which has an associated Lebesgue-Stieltjes measure $\mu_{F_I}$. We have the following proposition:

\begin{propositionP}
\label{prop:Laplace}
Given $n\in\N$, $I\subset\{1,\ldots,n\}$ and two functions $F,G:(\R_0^+)^n \to \R$ such that $F$ is a bounded $n$-dimensional Stieltjes measure function, let
\[\int G(\underline{x})dF_I(\underline{x}):=
\begin{cases}
G(0,\ldots,0)F(0,\ldots,0) & \text{for $I=\emptyset$}\\
\int G(\underline{x})d\mu_{F_I} & \text{for $I\neq\emptyset$}
\end{cases}\]
where $\mu_{F_I}$ is the Lebesgue-Stieltjes measure associated to $F_I$. Then,
{
\fontsize{10}{10}\selectfont
\[\int_0^\infty\!\!\!\!\!\!\ldots\!\!\int_0^\infty \e^{-y_1 x_1-\ldots-y_n x_n}F(\underline{x})dx_1\ldots dx_n=\frac{1}{y_1\ldots y_n}\sum_{I\subset\{1,\ldots,n\}}\int \e^{-\sum_{i\in I} y_i x_i}dF_I(\underline{x})\]
}
\end{propositionP}

\begin{proof}
We use induction over $n$. For $n=1$, using integration by parts,
\[\int_0^\infty \e^{-y_1 x_1}F(x_1)dx_1=\lim_{a\to\infty}\left[-\frac{\e^{-y_1.a}}{y_1}F(a)+\frac{\e^{-y_1.0}}{y_1}F(0)+\frac{1}{y_1}\int_0^a \e^{-y_1 x_1}dF(x_1)\right]\]
\[=\frac{1}{y_1}\left(F(0)+\int_0^\infty \e^{-y_1 x_1}dF(x_1)\right)=\frac{1}{y_1}\sum_{I\subset\{1\}}\int\e^{-\sum_{i\in I} y_i x_i}dF_I(x_1)\]

For $n>1$, using integration by parts again,
{
\fontsize{10}{10}\selectfont
\begin{align*}
\int_0^\infty\!\!\!\!\!\!\ldots\!\!\int_0^\infty & \e^{-y_1 x_1-\ldots-y_n x_n}F(\underline{x})dx_1\ldots dx_n=\\
&=\lim_{a\to\infty}\int_0^a \e^{-y_n x_n}\left(\int_0^\infty\!\!\!\!\!\!\ldots\!\!\int_0^\infty \e^{-y_1 x_1-\ldots-y_{n-1} x_{n-1}}F(\underline{x})dx_1\ldots dx_{n-1}\right)dx_n\\
&=\lim_{a\to\infty}-\frac{\e^{-y_n a}}{y_n}\int_0^\infty\!\!\!\!\!\!\ldots\!\!\int_0^\infty \e^{-y_1 x_1-\ldots-y_{n-1} x_{n-1}}F(x_1,\ldots,x_{n-1},a)dx_1\ldots dx_{n-1}\\
&\quad+\frac{1}{y_n}\int_0^\infty\!\!\!\!\!\!\ldots\!\!\int_0^\infty \e^{-y_1 x_1-\ldots-y_{n-1} x_{n-1}}F(x_1,\ldots,x_{n-1},0)dx_1\ldots dx_{n-1}\\
&\quad+\frac{1}{y_n}\int_0^\infty \e^{-y_n x_n}d\left(\int_0^\infty\!\!\!\!\!\!\ldots\!\!\int_0^\infty \e^{-y_1 x_1-\ldots-y_{n-1} x_{n-1}}F(\underline{x})dx_1\ldots dx_{n-1}\right)
\end{align*}
}
Since $F$ is bounded, we have
\[\lim_{a\to\infty}-\frac{\e^{-y_n a}}{y_n}\int_0^\infty\!\!\!\!\!\!\ldots\!\!\int_0^\infty \e^{-y_1 x_1-\ldots-y_{n-1} x_{n-1}}F(x_1,\ldots,x_{n-1},a)dx_1\ldots dx_{n-1}=0\]

Assuming that the result is valid for the $n-1$ dimensional functions $f_{x_n}(x_1,\ldots,x_{n-1})=\allowbreak  F(x_1,\ldots,x_{n-1},x_n)$ for every $x_n\geq 0$, we have
{
\fontsize{10}{10}\selectfont
\begin{align*}
\int_0^\infty\!\!\!\!\!\!\ldots\!\!\int_0^\infty &\e^{-y_1 x_1-\ldots-y_{n-1} x_{n-1}}F(x_1,\ldots,x_{n-1},0)dx_1\ldots dx_{n-1}\\
&=\int_0^\infty\!\!\!\!\!\!\ldots\!\!\int_0^\infty \e^{-y_1 x_1-\ldots-y_{n-1} x_{n-1}}f_0(x_1,\ldots,x_{n-1})dx_1\ldots dx_{n-1}\\
&=\frac{1}{y_1\ldots y_{n-1}}\sum_{I\subset\{1,\ldots,n-1\}}\int \e^{-\sum_{i\in I} y_i x_i}d(f_0)_I(x_1,\ldots,x_{n-1})\\
&=\frac{1}{y_1\ldots y_{n-1}}\sum_{I\subset\{1,\ldots,n-1\}}\int \e^{-\sum_{i\in I} y_i x_i}dF_I(\underline{x})
\end{align*}
}
and
{
\fontsize{10}{10}\selectfont
\begin{align*}
\frac{1}{y_n}\int_0^\infty& \e^{-y_n x_n}d\left(\int_0^\infty\!\!\!\!\!\!\ldots\!\!\int_0^\infty \e^{-y_1 x_1-\ldots-y_{n-1} x_{n-1}}F(\underline{x})dx_1\ldots dx_{n-1}\right)\\
&=\frac{1}{y_n}\int_0^\infty \e^{-y_n x_n}d\left(\int_0^\infty\!\!\!\!\!\!\ldots\!\!\int_0^\infty \e^{-y_1 x_1-\ldots-y_{n-1} x_{n-1}}f_{x_n}(x_1,\ldots,x_{n-1})dx_1\ldots dx_{n-1}\right)\\
&=\frac{1}{y_n}\int_0^\infty \e^{-y_n x_n}d\left(\frac{1}{y_1\ldots y_{n-1}}\sum_{J\subset\{1,\ldots,n-1\}}\int \e^{-\sum_{i\in J} y_i x_i}d(f_{x_n})_J(x_1,\ldots,x_{n-1})\right)\\
&=\frac{1}{y_1\ldots y_n}\sum_{J\subset\{1,\ldots,n-1\}}\int_0^\infty \e^{-y_n x_n}d\left(\int\e^{-\sum_{i\in J} y_i x_i}dF_{J\cup\{n\}}(\underline{x})\right)\\
&=\frac{1}{y_1\ldots y_n}\sum_{I\subset\{1,\ldots,n\},n\in I}\int\e^{-\sum_{i\in I} y_i x_i}dF_I(\underline{x})
\end{align*}
}

So, 
{
\fontsize{10}{10}\selectfont
$\displaystyle\int_0^\infty\!\!\!\!\!\!\ldots\!\!\int_0^\infty \e^{-y_1 x_1-\ldots-y_n x_n}F(\underline{x})dx_1\ldots dx_n=\frac{1}{y_1\ldots y_n}\sum_{I\subset\{1,\ldots,n\}}\int\e^{-\sum_{i\in I} y_i x_i}dF_I(\underline{x}).$
}
\end{proof}

\begin{corollaryP}
\label{cor:fgm-main-estimate}
Let $p\in\N_0$, $s,t,\varsigma\in\N$ and consider $y_1,y_2,\ldots,y_\varsigma\in\R_0^+$, $a>0$, $s+t-1<a_2<b_2<a_3<\ldots<b_\varsigma\in\N_0$. For $u$ sufficiently close to $u_F=\varphi(\zeta)$ we have
\[\E\left(\e^{-y_1 a\aa_{u,0}^s-y_2 a\aa_{u,a_2}^{b_2}-\ldots-y_\varsigma a\aa_{u,a_\varsigma}^{b_\varsigma}}\right)=\E\left(\e^{-y_1 a\aa_{u,0}^s}\right)\E\left(\e^{- y_2 a\aa_{u,a_2}^{b_2}-\ldots-y_\varsigma a\aa_{u,a_\varsigma}^{b_\varsigma}}\right)+ Err\]
where
$|Err|\leq s\iota(u,t)+4s\sum_{j=p+1}^{s-1}\p\left(Q_{p,0}^0(u)\cap \{X_j>u\}\right)+2\int_0^{\infty}y_1\e^{-y_1 x}\delta_{p,s,u}(x/a)dx$
and $\iota(u,t)$ is given by \eqref{eq:def-iota} and $\delta_{p,s,u}$ as in \eqref{eq:delta-definition}.
\end{corollaryP}
\begin{proof}
Using the same notation as in the proof of Lemma~\ref{prop:main-step}, let $F^{(A)}(x_1,\ldots,x_\varsigma)=\p(A_{x_1,\underline{x}})$, $F^{(B)}(x_1)=\p(B_{x_1})$ and $F^{(D)}(x_2,\ldots,x_\varsigma)=\p(D^{\underline{x}})$. Then, $F^{(A)}$, $F^{(B)}$ and $F^{(D)}$ are both bounded Stieltjes measure functions, with
\begin{align*}
\mu_{F^{(A)}}(U_1)&=\p\left(( a\aa_{u,0}^s, a\aa_{u,a_2}^{b_2},\ldots, a\aa_{u,a_\varsigma}^{b_\varsigma})\in U_1\right)\\
\mu_{F^{(B)}}(U_2)&=\p( a\aa_{u,0}^s\in U_2)\qquad
\mu_{F^{(D)}}(U_3)=\p\left(( a\aa_{u,a_2}^{b_2},\ldots, a\aa_{u,a_\varsigma}^{b_\varsigma})\in U_3\right)
\end{align*}
where $U_1$, $U_2$ and $U_3$ are Borel sets in $(\R_0^+)^\varsigma$, $\R_0^+$ and $(\R_0^+)^{\varsigma-1}$, respectively.

Therefore, we can apply the previous proposition and we obtain
{
\fontsize{10}{10}\selectfont
\begin{align*}
&\E\left(\e^{-y_1 a\aa_{u,0}^s-y_2 a\aa_{u,a_2}^{b_2}-\ldots-y_\varsigma a\aa_{u,a_\varsigma}^{b_\varsigma}}\right)-\E\left(\e^{-y_1 a\aa_{u,0}^s}\right)\E\left(\e^{- y_2 a\aa_{u,a_2}^{b_2}-\ldots-y_\varsigma a\aa_{u,a_\varsigma}^{b_\varsigma}}\right)\\
&=\sum_{I\subset\{1,\ldots,\varsigma\}}\int\e^{-\sum_{i\in I}y_i a x_i}d(F^{(A)})_I(x_1,\ldots,x_\varsigma)\\
&\quad-\sum_{I\subset\{1\}}\int\e^{-\sum_{i\in I}y_i a x_i}d(F^{(B)})_I(x_1)\sum_{I\subset\{2,\ldots,\varsigma\}}\int\e^{-\sum_{i\in I}y_i a x_i}d(F^{(D)})_I(x_2,\ldots,x_\varsigma)\\
&=y_1\ldots y_{\varsigma}a^{\varsigma}\int_0^\infty\!\!\!\!\!\!\ldots\!\!\int_0^\infty\e^{-y_1 a x_1-\ldots-y_{\varsigma} a  x_{\varsigma}}F^{(A)}(x_1,\ldots,x_\varsigma)dx_1\ldots dx_{\varsigma}\\
&\quad-\left(y_1 a\int_0^\infty\e^{-y_1 a x_1}F^{(B)}(x_1)dx_1\right)\left(y_2\ldots y_{\varsigma}a^{\varsigma-1}\int_0^\infty\!\!\!\!\!\!\ldots\!\!\int_0^\infty\e^{-y_2 a x_2-\ldots-y_{\varsigma} a x_{\varsigma}}F^{(D)}(x_2,\ldots,x_\varsigma)dx_2\ldots dx_{\varsigma}\right)\\
&=y_1\ldots y_{\varsigma}a^{\varsigma}\int_0^\infty\!\!\!\!\!\!\ldots\!\!\int_0^\infty\e^{-y_1 a x_1-\ldots-y_{\varsigma} a  x_{\varsigma}}(F^{(A)}-F^{(B)}F^{(D)})(x_1,\ldots,x_\varsigma)dx_1\ldots dx_{\varsigma}
\end{align*}
}
Hence, using Lemma~\ref{prop:main-step} and the change of variables $x=a x_1$,
\begin{multline*}
\left|\E\left(\e^{-y_1 a\aa_{u,0}^s-y_2 a\aa_{u,a_2}^{b_2}-\ldots-y_\varsigma a\aa_{u,a_\varsigma}^{b_\varsigma})}\right)-\E\left(\e^{-y_1 a\aa_{u,0}^s}\right)\E\left(\e^{- y_2 a\aa_{u,a_2}^{b_2}-\ldots-y_\varsigma a\aa_{u,a_\varsigma}^{b_\varsigma}}\right)\right|\\
\leq y_1\ldots y_{\varsigma}a^{\varsigma}\int_0^\infty\!\!\!\!\!\!\ldots\!\!\int_0^\infty\e^{-y_1 a x_1-\ldots-y_{\varsigma} a  x_{\varsigma}}\left|\p(A_{x_1,\underline{x}})-\p(B_{x_1})\p(D^{\underline{x}})\right|dx_1\ldots dx_{\varsigma}\\
\leq s\iota(u,t)+4s\sum_{j=p+1}^{s-1}\p\left(Q_{p,0}^0(u)\cap\{X_j>u\}\right)+2\int_0^\infty y_1\e^{-y_1 x}\delta_{p,s,u}(x/a)dx
\end{multline*}
\end{proof}

\begin{propositionP}
\label{prop:Laplace-estimates}
Let $X_0, X_1, \ldots$ be given by \eqref{eq:def-stat-stoch-proc-DS}, where $\varphi$ achieves a global maximum at the point $\zeta$. Let $(u_n)_{n\in\N}$ be a sequence satisfying \eqref{eq:un} and $(a_n)_{n\in\N}$ a normalising sequence. Assume that conditions $\D_p(u_n)^*$, $\D'_p(u_n)^*$ and $U\!LC_p(u_n)$ hold for some $p\in\N_0$. Let $J\in\RR$ be such that $J=\bigcup_{\ell=1}^\varsigma I_\ell$ where $I_j=[a_j,b_j)\in\S$,
$j=1,\ldots,\varsigma$ and $a_1<b_1<a_2<\cdots<b_{\varsigma-1}<a_\varsigma<b_\varsigma$.
Then, for all $y_1,y_2,\ldots,y_\varsigma\in\R_0^+$, we have
\[\E\left(\e^{-\sum_{\ell=1}^\varsigma y_\ell a_n\aa_{u_n}(nI_\ell)}\right)-\prod_{\ell=1}^\varsigma \E^{k_n|I_\ell|}\left(\e^{-y_\ell a_n\aa_{u_n,0}^{\lfloor n/k_n\rfloor}}\right)\xrightarrow[n\to\infty]{}0\]
\end{propositionP}
\begin{proof}
Without loss of generality, we can assume that $y_1,y_2,\ldots,y_\varsigma\in\R^+$, because if we had $y_j=0$ for some $j=1,\ldots,\varsigma$ then we could consider $J=\bigcup_{\ell=1}^{j-1}I_\ell\cup\bigcup_{\ell=j+1}^\varsigma I_\ell$ instead.
Let $h:=\inf_{j\in \{1,\ldots,\varsigma\}}\{b_j-a_j\}$, $H:=\lceil\sup\{x:x\in J\}\rceil=\lceil b_\varsigma\rceil$, $\hat y:=\inf\{y_j:j=1,\ldots,\varsigma\}>0$ and $\hat Y:=\sup\{y_j:j=1,\ldots,\varsigma\}$. Let $n$ be sufficiently large so that, in particular, $k_n>2/h$ and set $\varrho_n:=\lfloor n/k_n\rfloor$. We consider the following partition of $n[0,H]\cap{\mathbb Z}$ into blocks of length $\varrho_n$, $J_1=[0,\varrho_n)$, $J_2=[\varrho_n,2\varrho_n)$, ..., $J_{Hk_n}=[(Hk_n-1)\varrho_n,Hk_n\varrho_n)$, $J_{Hk_n+1}=[Hk_n\varrho_n,Hn)$. We further cut each $J_i$ into two blocks:
\[J_i^*:=[(i-1)\varrho_n,i\varrho_n-t_n)\; \mbox{ and } J_i':=J_i-J_i^*\]
Note that $|J_i^*|=\varrho_n-t_n$ and $|J_i'|=t_n$.

Let $\mathscr S_\ell=\mathscr S_\ell(k)$ be the number of blocks $J_j$ contained in $n I_\ell$, that is,
\[\mathscr S_\ell:=\#\{j\in \{1,\ldots,Hk_n\}:J_j\subset n I_\ell\}\]
By assumption on the relation between $k_n$ and $h$, we have $\mathscr S_\ell>1$ for every $\ell\in \{1,\ldots,\varsigma\}$.
For each such $\ell$, we also define $i_\ell:=\min\{j\in \{1,\ldots,k\}:J_j\subset n I_{\ell}\}$.
Hence, it follows that $J_{i_\ell},J_{i_\ell+1},\ldots,J_{i_\ell+\mathscr S_\ell-1}\subset n I_\ell$. Moreover, by choice of the size of each block we have that
\begin{equation}
\label{eq:Sl-estimate}
\mathscr S_\ell\sim k_n|I_\ell|
\end{equation}

First of all, recall that for every $0\leq x_i, z_i\leq 1$, we have
\begin{equation}
\label{eq:inequality}
\left|\prod x_i-\prod z_i\right|\leq \sum |x_i-z_i|.
\end{equation}

We start by making the following approximation, in which we use \eqref{eq:inequality} and stationarity,
{
\fontsize{10}{10}\selectfont
\begin{align*}
\Bigg|\E\left(\e^{-\sum_{\ell=1}^\varsigma y_\ell a_n\aa_{u_n}(nI_\ell)}\right)&-\E\left(\e^{-\sum_{\ell=1}^\varsigma y_\ell\sum_{j=i_\ell}^{i_\ell+\mathscr S_\ell-1}a_n\aa_{u_n}(J_j)}\right)\Bigg|\leq\E\left(1-\e^{-\sum_{\ell=1}^\varsigma y_\ell a_n\aa_{u_n}\left(nI_\ell\setminus\cup_{j=i_\ell}^{i_\ell+\mathscr S_\ell-1}J_j\right)}\right)\\
&\leq \E\left(1-\e^{-2\sum_{\ell=1}^\varsigma y_\ell a_n\aa_{u_n}(J_1)}\right)
\leq 2\varsigma K\E\left(1-\e^{-a_n\aa_{u_n}(J_1)}\right),
\end{align*}
}

where $\max\{y_1,\ldots,y_\varsigma\}\leq K\in\N$.
In order to show that we are allowed to use the above approximation we just need to check that $\E\left(1-\e^{-a_n\aa_{u_n}(J_1)}\right)\to0$ as $n\to\infty$.
By Corollary~\ref{cor:exponential} we have 
\begin{equation}
\label{eq:error1}
\E\left(\e^{-a_n\aa_{u_n}(J_1)}\right)=1-\varrho_n\int_0^{\infty}\e^{-x}\p(R_{p,0}(u_n,x/a_n))dx+Err,
\end{equation}
where 
\[\left|Err\right|\leq 2\varrho_n\sum_{j=p+1}^{\varrho_n-1}\p\left(Q_{p,0}^0(u_n)\cap \{X_j>u_n\}\right)+\int_0^{\infty}\e^{-x}\delta_{p,\varrho_n,u_n}(x/a_n)dx \to 0\]
as $n\to\infty$ by $\D'_p(u_n)^*$ and $U\!LC_p(u_n)$.
Since $\int_0^{\infty}\e^{-x}\p(R_{p,0}(u_n,x/a_n))dx \leq \int_0^{\infty}\e^{-x}\p(U(u_n))dx=\p(U(u_n))$ we get $\E\left(\e^{-a_n\aa_{u_n}(J_1)}\right)\xrightarrow[n\to\infty]{}1$ by \eqref{eq:un}.

Now, we proceed with another approximation which consists of replacing $J_j$ by $J_j^*$. Using \eqref{eq:inequality}, stationarity and \eqref{eq:Sl-estimate}, we have
\begin{align*}
\Bigg|\E\left(\e^{-\sum_{\ell=1}^\varsigma y_\ell\sum_{j=i_\ell}^{i_\ell+\mathscr S_\ell-1}a_n\aa_{u_n}(J_j)}\right)&-\E\left(\e^{-\sum_{\ell=1}^\varsigma y_\ell\sum_{j=i_\ell}^{i_\ell+\mathscr S_\ell-1}a_n\aa_{u_n}(J_j^*)}\right)\Bigg|\leq 
\E\left(1-\e^{-\sum_{\ell=1}^\varsigma y_\ell\mathscr S_\ell a_n\aa_{u_n}(J'_1)}\right)\\
&\leq K\sum_{\ell=1}^\varsigma \mathscr S_\ell \E\left(1-\e^{-a_n\aa_{u_n}(J'_1)}\right)
%&
\lesssim KHk_n\E\left(1-\e^{-a_n\aa_{u_n}(J'_1)}\right),
\end{align*}
and we must show that $k_n\E\left(1-\e^{-a_n\aa_{u_n}(J'_1)}\right)\to0,$ as $n\to\infty$, in order for the approximation to make sense. By Corollary~\ref{cor:exponential} we have 
\begin{equation}
\label{eq:error1}
\E\left(\e^{-a_n\aa_{u_n}(J'_1)}\right)=1-t_n\int_0^{\infty}\e^{-x}\p(R_{p,0}(u_n,x/a_n))dx+Err,
\end{equation}
where 
\[k_n\left|Err\right|\leq 2k_n t_n\sum_{j=p+1}^{t_n-1}\p\left(Q_{p,0}^0(u_n)\cap \{X_j>u_n\}\right)+k_n\int_0^{\infty}\e^{-x}\delta_{p,t_n,u_n}(x/a_n)dx \to 0\]

as $n\to\infty$ by $\D'_p(u_n)^*$ and $U\!LC_p(u_n)$. We get, by \eqref{eq:un} as well,
\begin{equation}
\label{eq:error}
k_n\E\left(1-\e^{-a_n\aa_{u_n}(J'_1)}\right)\sim k_n t_n\int_0^{\infty}\e^{-x}\p(R_{p,0}(u_n,x/a_n))dx\xrightarrow[n\to\infty]{}0
\end{equation}

Let us fix now some $\hat\ell\in\{1,\ldots,\varsigma\}$ and $i\in\{i_{\hat\ell},\ldots,i_{\hat\ell}+\mathscr S_{\hat\ell}-1\}$. Let $M_i=y_{\hat\ell}\sum_{j=i}^{i_{\hat\ell}+\mathscr S_{\hat\ell}-1}a_n\aa_{u_n}(J_j^*)$ and $L_{\hat\ell}=\sum_{\ell=\hat\ell+1}^\varsigma y_\ell\sum_{j=i_{\ell}}^{i_\ell+\mathscr S_\ell-1}a_n\aa_{u_n}(J_j^*).$
Using stationarity and Corollary~\ref{cor:fgm-main-estimate} along with the facts that $\iota(u_n,t)\le\gamma(n,t)$ and $y_{\hat\ell}\e^{-y_{\hat\ell}x}\le \hat Y\e^{-\hat y x}$, we obtain
\begin{equation*}
\left|\E\left(\e^{-y_{\hat\ell}a_n\aa_{u_n}(J_{i_{\hat\ell}}^*)-M_{i_{\hat\ell}+1}-L_{\hat\ell}}\right)-\E\left(\e^{-y_{\hat\ell}a_n\aa_{u_n}(J_1^*)}\right)\E\left(\e^{-M_{i_{\hat\ell}+1}-L_{\hat\ell}}\right) \right|\leq \Upsilon_n,
\end{equation*}
where 
\[\Upsilon_n=\varrho_n\gamma(n,t_n)+4\varrho_n\sum_{j=p+1}^{\varrho_n-1}\p\left(Q_{p,0}^0(u_n)\cap \{X_j>u_n\}\right)+2\hat Y\int_0^{\infty}\e^{-\hat y x}\delta_{p,\varrho_n,u_n}(x/a_n)dx\]

Since $\E\left(\e^{-y_{\hat\ell}a_n\aa_{u_n}(J_1^*)}\right)\leq 1$, it follows by the same argument that
\begin{align*}
\Big|\E\Big(\e^{-M_{i_{\hat\ell}}-L_{\hat\ell}}\Big)-&\E^2\Big(\e^{-y_{\hat\ell}a_n\aa_{u_n}(J_1^*)}\Big)\E\Big(\e^{-M_{i_{\hat\ell}+2}-L_{\hat\ell}}\Big)\Big|\\ 
&\leq\Big|\E\Big(\e^{-M_{i_{\hat\ell}}-L_{\hat\ell}}\Big)-\E\Big(\e^{-y_{\hat\ell}a_n\aa_{u_n}(J_1^*)}\Big)\E\Big(\e^{-M_{i_{\hat\ell}+1}-L_{\hat\ell}}\Big) \Big|\\
&+\E\Big(\e^{-y_{\hat\ell}a_n\aa_{u_n}(J_1^*)}\Big)\Big|\E\Big(\e^{-M_{i_{\hat\ell}+1}-L_{\hat\ell}}\Big)-\E\Big(\e^{-y_{\hat\ell}a_n\aa_{u_n}(J_1^*)}\Big)\E\Big(\e^{-M_{i_{\hat\ell}+2}-L_{\hat\ell}}\Big)\Big|\\
&\leq 2\Upsilon_n,
\end{align*}
Hence, proceeding inductively with respect to $i\in\{i_{\hat\ell},\ldots,i_{\hat\ell}+\mathscr S_{\hat\ell}-1\}$, we obtain
\[
\Big|\E\Big(\e^{-M_{i_{\hat\ell}}-L_{\hat\ell}}\Big)-\E^{\mathscr S_{\hat\ell}}\Big(\e^{-y_{\hat\ell}a_n\aa_{u_n}(J_1^*)}\Big)\E\Big(\e^{-L_{\hat\ell}}\Big)\Big|\leq \mathscr S_{\hat\ell}\Upsilon_n
\]
In the same way, if we proceed inductively with respect to $\hat\ell\in\{1,\ldots,\varsigma\}$, we get
\[
\left|\E\left(\e^{-\sum_{\ell=1}^\varsigma y_\ell\sum_{j=i_\ell}^{i_\ell+\mathscr S_\ell-1}a_n\aa_{u_n}(J_j^*)}\right)-\prod_{\ell=1}^\varsigma \E^{\mathscr S_{\ell}}\Big(\e^{-y_{\ell}a_n\aa_{u_n}(J_1^*)}\Big)\right|\leq \sum_{\ell=1}^\varsigma \mathscr S_{\ell}\,\Upsilon_n.
\]
By \eqref{eq:Sl-estimate}, we have $\sum_{\ell=1}^\varsigma \mathscr S_{\ell}\,\Upsilon_n\lesssim H k_n\Upsilon_n$ and
\begin{align*}
k_n\Upsilon_n &=k_n\varrho_n\gamma(n,t_n)+4k_n\varrho_n\sum_{j=p+1}^{\varrho_n-1}\p\left(Q_{p,0}^0(u_n)\cap\{X_j>u_n\}\right)+2k_n\hat Y\int_0^{\infty}\e^{-\hat y x}\delta_{p,\varrho_n,u_n}(x/a_n)dx\\
&\sim n\gamma(n,t_n)+4n\sum_{j=p+1}^{\varrho_n-1}\p\left(Q_{p,0}^0(u_n)\cap\{X_j>u_n\}\right)+\frac{2\hat Y}{\hat y}k_n\int_0^{\infty}\hat y\e^{-\hat y x}\delta_{p,\varrho_n,u_n}(x/a_n)dx\\
&\to 0,\quad \text{as $n\to\infty$, by $\D_p(u_n)^*$, $\D'_p(u_n)^*$ and $U\!LC_p(u_n)$.}
\end{align*}
Using \eqref{eq:inequality} and stationarity, again, we have the final approximation
\begin{align*}
\left|\prod_{\ell=1}^\varsigma \E^{\mathscr S_{\ell}}\Big(\e^{-y_{\ell}a_n\aa_{u_n}(J_1)}\Big)-\prod_{\ell=1}^\varsigma \E^{\mathscr S_{\ell}}\Big(\e^{-y_{\ell}a_n\aa_{u_n}(J_1^*)}\Big)\right| %&\leq K\sum_{\ell=1}^\varsigma \mathscr S_\ell \E\left(1-\e^{-a_n\aa_{u_n}(J'_1)}\right)\\
&\lesssim KHk_n\E\left(1-\e^{-a_n\aa_{u_n}(J'_1)}\right).
\end{align*}
Since in \eqref{eq:error} we have already proved that $k_n\E\left(1-\e^{-a_n\aa_{u_n}(J'_1)}\right)\to 0$ as $n\to\infty$, we only need to gather all the approximations and recall \eqref{eq:Sl-estimate} to finally obtain the stated result.
\end{proof}

\begin{proof}[Proof of Theorem~\ref{thm:convergence}]
In order to prove convergence of $a_n A_n$ to a process $A$, it is sufficient to show that for any $\varsigma$ disjoint intervals $I_1, I_2,\ldots, I_\varsigma\in\S$, the joint distribution of $a_n A_n$ over these intervals converges to the joint distribution of $A$ over the same intervals, \ie 
\[
(a_n A_n(I_1),a_n A_n(I_2),\ldots,a_n A_n(I_\varsigma))\xrightarrow[n\to\infty]{}(A(I_1), A(I_2), \ldots, A(I_\varsigma)),
\]
which will be the case if the corresponding joint Laplace transforms converge. Hence, we only need to show that 
\[
\psi_{a_n A_n}(y_1, y_2,\ldots,y_\varsigma)\to \psi_A(y_1, y_2,\ldots,y_\varsigma)=\E\left(\e^{-\sum_{\ell=1}^\varsigma y_\ell A(I_\ell)}\right), \quad \text{as $n\to\infty$,}
\] 
for every $\varsigma$ non-negative values $y_1, y_2,\ldots,y_\varsigma$, each choice of $\varsigma$ disjoint intervals $I_1, I_2,\ldots, I_\varsigma\in\S$ and each $\varsigma\in\N$. Note that $\psi_{a_n A_n}(y_1, y_2,\ldots,y_\varsigma)=\E\left(\e^{-\sum_{\ell=1}^\varsigma y_\ell a_n A_n(I_\ell)}\right)=\E\left(\e^{-\sum_{\ell=1}^\varsigma y_\ell a_n\aa_{u_n}(v_nI_\ell)}\right)$ and
{
\fontsize{10}{10}\selectfont
\begin{align*}
\left|\E\left(\e^{-\sum_{\ell=1}^\varsigma y_\ell a_n \aa_{u_n}(v_n I_\ell)}\right)-\E\left(\e^{-\sum_{\ell=1}^\varsigma y_\ell A(I_\ell)}\right)\right|&\leq \left|\E\left(\e^{-\sum_{\ell=1}^\varsigma y_\ell a_n\aa_{u_n}(v_nI_\ell)}\right)- \prod_{\ell=1}^\varsigma \E^{k_n\frac{v_n}{n}|I_\ell|}\left(\e^{-y_\ell a_n\aa_{u_n,0}^{\lfloor n/k_n\rfloor}}\right)\right|\\
&+\left|\prod_{\ell=1}^\varsigma \E^{k_n\frac{v_n}{n}|I_\ell|}\left(\e^{-y_\ell a_n\aa_{u_n,0}^{\lfloor n/k_n\rfloor}}\right)-\E\left(\e^{-\sum_{\ell=1}^\varsigma y_\ell A(I_\ell)}\right)\right|
\end{align*}
}
By Proposition~\ref{prop:Laplace-estimates}, the first term on the right goes to $0$ as $n\to\infty$. Also, by Corollary~\ref{cor:exponential}, we have
\begin{align*}
\E\left(\e^{-y a_n\aa_{u_n,0}^{\lfloor n/k_n\rfloor}}\right)&=1-\lfloor n/k_n\rfloor\int_0^{\infty}y\e^{-yx}\p(R_{p,0}(u_n,x/a_n))dx+Err,
\end{align*}
where 
\[|Err|\leq \frac{2n}{k_n}\sum_{j=p+1}^{\lfloor n/k_n\rfloor-1}\p\left(Q_{p,0}^0(u_n)\cap \{X_j>u_n\}\right)+
\int_0^{\infty}y\e^{-yx}\delta_{p,\lfloor n/k_n\rfloor,u_n}(x/a_n)dx\]
Since, by $\D'_p(u_n)^*$ and $U\!LC_p(u_n)$, we have that $k_n|Err|\to 0$ as $n\to\infty$, it follows that
\[\E^{k_n}\left(\e^{-y a_n\aa_{u_n,0}^{\lfloor n/k_n\rfloor}}\right)\sim\left(1-\frac{n}{k_n}\int_0^{\infty}y\e^{-yx}\p(R_{p,0}(u_n,x/a_n))dx\right)^{k_n}\sim\e^{-n\int_0^{\infty}y\e^{-yx}\p(R_{p,0}(u_n,x/a_n))dx},\]
as $n\to\infty$.

Hence,
\begin{align*}
\prod_{\ell=1}^\varsigma& \E^{k_n\frac{v_n}{n}|I_\ell|}\left(\e^{-y_\ell a_n\aa_{u_n,0}^{\lfloor n/k_n\rfloor}}\right)\sim \prod_{\ell=1}^\varsigma\left(\e^{-n\int_0^{\infty}y_\ell\e^{-y_\ell x}\p(R_{p,0}(u_n,x/a_n))dx}\right)^{\frac{v_n}{n}|I_\ell|}\\
&=\e^{-v_n\sum_{\ell=1}^\varsigma|I_\ell|\int_0^{\infty}y_\ell\e^{-y_\ell x}\p(R_{p,0}(u_n,x/a_n))dx}=\e^{-\sum_{\ell=1}^\varsigma|I_\ell|\int_0^{\infty}y_\ell\e^{-y_\ell x}\frac{\p(R_{p,0}(u_n,x/a_n))}{P(U(u_n))}dx}
\end{align*}
{
\fontsize{10}{10}\selectfont
\begin{align*}
\lim_{n\to\infty}\int_0^{\infty}y\e^{-y x}\frac{\p(R_{p,0}(u_n,x/a_n))}{P(U(u_n))}dx&=\int_0^{\infty}y\e^{-y x}\theta(1-\pi(x))dx%\\
%&=\theta\left([-\e^{-y x}(1-\pi(x))]_0^{\infty}-\int_0^{\infty}(-\e^{-y x})d(1-\pi(x))\right)\\
%&
=\theta\left(1-\pi(0)-\int_0^{\infty}\e^{-y x}d\pi(x)\right)\\
&
=\theta(1-\phi(y))
\end{align*}
}
where $\phi$ is the Laplace transform of $\pi$, and
\[\lim_{n\to\infty}\e^{-\sum_{\ell=1}^\varsigma|I_\ell|\int_0^{\infty}y_\ell\e^{-y_\ell x}\frac{\p(R_{p,0}(u_n,x/a_n))}{P(U(u_n))}dx}=\e^{-\theta\sum_{\ell=1}^\varsigma|I_\ell|(1-\phi(y_l))}=\E\left(\e^{-\sum_{\ell=1}^\varsigma y_\ell A(I_\ell)}\right)\]
where $A$ is a compound Poisson process of intensity $\theta$ and multiplicity d.f. $\pi$.
\end{proof}

\section{Convergence of random measures for induced and original systems}
\label{sec:induced-original}

In this section we prove Theorem~\ref{thm:mpp_ret_orig}. We start by settling notation. For all $A,B\in\mathcal B$ and $j\in\N_0$ we define $r_{A,B}^j$ as $r_A^j$ simply by replacing iterations by $f$ by iterations by $F_B$. To ease the notation we let $U_n:=U(u_n)$. We will assume throughout that $n$ is sufficiently large so that 
%\begin{equation}
%\label{eq:UnsubsetB}
$U_n\subset B$.
%\end{equation}

We start with the following simple observation.
\begin{lemma}
\label{lem:hits-relation}
If $x\in \{r_B>j\}$ then for all $i\in\N$ we have $r_{U_n}^i(f^j(x))=r_{U_n}^i(x)-j$.
\end{lemma}

\begin{proof}
We will use induction. Note that since $U_n\subset B$ then $r_B(x)>j$ implies that $r_{U_n}(x)>j$ and then it is clear that $r_{U_n}(f^j(x))=r_{U_n}(x)-j$. Moreover, $F_{U_n}(f^j(x))=F_{U_n}(x)$ since $F_{U_n}(f^j(x))=f^{r_{U_n}(f^j(x))}(f^j(x))=f^{r_{U_n}(x)-j}(f^j(x))=f^{r_{U_n}(x)}(x)=F_{U_n}(x)$.

Assume now by hypothesis that the statement of the lemma holds for $i$ and that $F_{U_n}^i(f^j(x))=F_{U_n}^i(x)$. Then $r_{U_n}^{i+1}(f^j(x))=r_{U_n}(F_{U_n}^i(f^j(x)))+r_{U_n}^i(f^j(x))=r_{U_n}(F_{U_n}^i(x))+r_{U_n}^i(x)-j=r_{U_n}^{i+1}(x)-j$. Moreover, $F_{U_n}^{i+1}(f^j(x))=f^{r_{U_n}(F_{U_n}^i(f^j(x)))}(F_{U_n}^i(f^j(x)))=f^{r_{U_n}(F_{U_n}^i(x))}(F_{U_n}^i(x))=F_{U_n}^{i+1}(x)$.
\end{proof}

The next two lemmata have as purpose to see that we can replace $\p$ by $\p_B$ to study the distribution of $A_n$. Let $J\in\RR$ so that $J=\cup_{l=1}^k I_j$, where $I_j=[a_j,b_j)\in\S$ are disjoint intervals. Let $\mathbf{x}=(x_1, x_2,\ldots, x_k)\in\R^k$ and define the event \begin{equation}
\label{eq:A-event}
\mathbb A(J,\mathbf x,n)=\{A_n(I_1)>x_1, A_n(I_2)>x_2,\ldots, A_n(I_k)>x_k\}.
\end{equation} 

We begin proving that  $\p(\mathbb A(J,\mathbf x,n))$ can be approximated by $\int_B r_B.\I_{\mathbb A(J,\mathbf x,n)} d\p$. But before we recall two useful formulas that are standard for induced maps:
{
\fontsize{10}{10}\selectfont
\begin{align}
\int_B r_B d\p_B&= \sum_{j=0}^\infty \p_B(r_B>j)
\label{eq:mean-formula}\\
\p(A)&=\sum_{j=0}^\infty \p(B\cap\{r_B>j\}\cap f^{-j}(A)) \label{eq:induced-measure}
\end{align}
}
\begin{lemma}
\label{lem:int-overB}
For any small $\eps_0,\eps_1>0$ and $n$ sufficiently large we have 
$$
\int_B r_B.\I_{\mathbb A(J^{\eps_1-},\mathbf x,n)}d\p-\eps_0\leq \p(\mathbb A(J,\mathbf x,n))\leq \int_B r_B.\I_{\mathbb A(J^{\eps_1+},\mathbf x,n)}d\p+\eps_0 
$$
\end{lemma}
\begin{proof}

By Lemma \ref{lem:hits-relation}, we have that
\[
r_{U_n} \circ f^j=r_{U_n}-j \;\;\; \text{in}\;\;\;  \{r_B >j\} \subset \{r_{U_n} >j\}.
\]

Let $\eps_0,\eps_1>0$. We start by choosing $N^*$ such that
\[
\sum_{j>N^*} \p(B\cap\{r_B>j\})<\eps_0,
\]
which is possible since $\int_B r_B d\p<\infty$.

Let $N_1$ be sufficiently large such that $N^*\p(U_n)=N^*v_n^{-1}<\eps_1$, $\forall n>N_1$. 

We start by proving the second inequality of Lemma \ref{lem:int-overB}.
We have that
{
\fontsize{10}{10}\selectfont
\begin{align}
\p(\mathbb A(J,\mathbf x,n))&< \sum_{j=0}^{N^*}\p(B\cap\{r_B>j\}\cap f^{-j}(\mathbb A(J,\mathbf x,n)))+\eps_0
\label{eq:des1}\\
&< \sum_{j=0}^{N^*}\p(B\cap\{r_B>j\}\cap \{\mathbb A(J^{\eps_1+},\mathbf x,n)\})+\eps_0
\label{eq:des2}
%&< \sum_{j=0}^{\infty}\p(B\cap\{r_B>j\}\cap \{A_n(J^{\eps_1+})>x\})+\eps_0
%\nonumber
%\\
< \int_B r_B \I_{\mathbb A(J^{\eps_1+},\mathbf x,n)}d\p+\eps_0.
%\label{eq:eq}
\end{align}
}

Inequality (\ref{eq:des1}) follows from (\ref{eq:induced-measure}).
The first inequality in (\ref{eq:des2}) holds because if $x\in B\cap \{r_B>j\}$, then $x\in B\cap \{r_{U_n}>j\}$ which implies that
$r^i_{U_n}\circ f^j(x)=r^i_{U_n}(x)-j$. Thus, if $r^i_{U_n}\circ f^j(x)\in v_n J$ then $r^i_{U_n}(x)\in v_n J^{\eps_1+}$, because $v_n \eps_1>N^*\geq j$ and so $r^i_{U_n}(x)=r^i_{U_n}\circ f^j(x)+j$ belongs to $v_n J^{\eps_1+}$.
The second inequality in (\ref{eq:des2}) follows from (\ref{eq:mean-formula}).
Thus, the second inequality of Lemma \ref{lem:int-overB} holds.

Now we turn to the first inequality.
We have that
{
\fontsize{10}{10}\selectfont
\begin{align}
\p(\mathbb A(J,\mathbf x,n))&> \sum_{j=0}^{N^*}\p(B\cap\{r_B>j\}\cap f^{-j}(\mathbb A(J,\mathbf x,n)))
\geq \sum_{j=0}^{N^*}\p(B\cap\{r_B>j\}\cap \{\mathbb A(J^{\eps_1-},\mathbf x,n)\})
\label{eq:des2-}\\
&\geq \sum_{j=0}^{\infty}\p(B\cap\{r_B>j\}\cap \{\mathbb A(J^{\eps_1-},\mathbf x,n)\})-\eps_0
\label{eq:des3-} =\int_B r_B \I_{\mathbb A(J^{\eps_1-},\mathbf x,n)}d\p-\eps_0 .
\end{align}
}

The second inequality in (\ref{eq:des2-}) holds because if $x\in \{r_B>j\}\subset\{r_{U_n}>j\}$, then, by Lemma \ref{lem:hits-relation},  
$r^i_{U_n}(f^j(x))=r^i_{U_n}(x)-j$. Thus, if $r^i_{U_n}(x)\in v_n J^{\eps_1-}$ then $r^i_{U_n}(f^j(x))\in v_n J$, because $v_n \eps_1>N^*\geq j$.
The inequality in (\ref{eq:des3-}) follows from (\ref{eq:induced-measure}).
\end{proof}

The next lemma shows that $\int_B r_B.\I_{\mathbb A(J,\mathbf x,n)} d\p$ can be approximated by $\p_B(\mathbb A(J,\mathbf x,n))$.
\begin{lemma}
\label{lem:muB}
For any small $\eps_0,\eps_1>0$ and $n$ sufficiently large we have 
$$
\p_B(\mathbb A(J^{\eps_1-},\mathbf x,n))-\eps_0\leq\int_B r_B.\I_{\mathbb A(J,\mathbf x,n)}d\p \leq \p_B(\mathbb A(J^{\eps_1+},\mathbf x,n))+\eps_0. 
$$
\end{lemma}
\begin{proof}

We start by noting that since $F_B$ is $\p$-invariant in $B$,
$$
\int_B r_B.\I_{\mathbb A(J,\mathbf x,n)}d\p = \frac{1}{M} \sum_{j=0}^{M-1} \int_B r_B\circ F_B^j\I_{\mathbb A(J,\mathbf x,n)}\circ F_B^jd\p.
$$

Let $\eps_0,\eps_1>0$. We will see that for $n$ sufficiently large
$$
\int_B r_B\circ F_B^j\I_{\mathbb A(J^{\eps_1-},\mathbf x,n)}d\p-\eps_0/2\leq\int_B r_B\circ F_B^j\I_{\mathbb A(J,\mathbf x,n)}\circ F_B^jd\p\leq \int_B r_B\circ F_B^j\I_{\mathbb A(J^{\eps_1+},\mathbf x,n)}d\p+\eps_0/2.
$$

As in Lemma \ref{lem:hits-relation}, we have that 
$$r_{U_n}^i\circ F_B^j=r_{U_n}^i-r_{B}^j \;\;\; \text{in} \;\;\; B\cap\{r_{U_n}>r_B^j\}=B\cap\{r_{U_n,B}>j\}.$$ 

Now, let $\eps_2$ be such that, if $\p(D)<\eps_2$ for some $D\subset B$ then 

\begin{equation}
\label{eq:eps2}
\int_D r_B\circ F_B^jd\p<\eps_0/2\;\; \;\;\forall j\in\N_0.
\end{equation}

Let $M^*$ be sufficiently large such that $\p(\{r_B^M>M^*\}\cap B)<\eps_2/2$.
Let $N_2$ be such that $\forall n>N_2$, $\p(B\cap\{r_{U_n,B}\leq M\})<\eps_2/2$.
%Observe that $\p(B\cap\{r_{U_n,B}^i\leq M\})\leq \p(B\cap\{r_{U_n,B}\leq M\})$ for all $i\in \N$ because $r_{U_n,B}^i\geq r_{U_n,B}$ for all $i\in \N$. 
We also assume that $\exists N_3\in\N$ $\forall n>N_3$ $M^*\p(U_n)=M^*v_n^{-1}<\eps_1$.
Let $G_n=B\cap \{r_B^M\leq M^*\}\cap \{r_{U_n,B}> M\}$. 

By construction, we have that $\p(B\setminus G_n)\leq  \p(B\cap\{r_{U_n,B}\leq M\})+\p(B\cap\{r_B^M>M^*\})<\eps_2/2+\eps_2/2=\eps_2$ for $n>N_2$.

Since $r_B$ is integrable in $B$ and $\p|_B$ is $F_B$-invariant, then the sequence of functions $\{r_B\circ F_B^j\}_{j\in \N}$ is uniformly integrable in $B$, \ie  $\int_B r_B\circ F_B^j d\p= \int_B r_B\circ F_B^{j'} d\p$, $\forall j,j'\in\N.$

Observe now that in $G_n$ and for $n>\max \{N_2,N_3\}$ we have that
\[
r_{U_n}^i\circ F_B^j=r_{U_n}^i- r_{B}^i
\]
because, if $x\in G_n$ then $x\in B\cap\{r_{U_n,B}> M\}$ which implies that $x\in B\cap\{r_{U_n,B}> j\}$. If $x\in G_n$ then $r_B^j(x)\leq r_B^M(x)\leq M^*$ and since $n>N_3$, then $r_B^j(x)v_n^{-1}<\eps_1$, and so
\begin{align*}
r_{U_n}^i\circ F_B^j(x)\in v_n J\Rightarrow r_{U_n}^i(x)\in v_n J^{\eps_1+}\quad\text{and}\quad
%\end{align*}
%%and
%\begin{align*}
 r_{U_n}^i(x)\in v_n J^{\eps_1-}\Rightarrow r_{U_n}^i\circ F_B^j(x)\in v_n J  
\end{align*}

In this way, we obtain that, for $n>N_3$,
\begin{align}
\label{eq:cont1}
\mathbb A(J^{\eps_1-},\mathbf x,n)\cap G_n\subset F_B^{-j}(\mathbb A(J,\mathbf x,n))\cap G_n
\end{align}
%and
\begin{align}
\label{eq:cont2}
F_B^{-j}(\mathbb A(J,\mathbf x,n))\cap G_n \subset \mathbb A(J^{\eps_1+},\mathbf x,n)\cap G_n.
\end{align}

We then may write
\[
\int_B r_B\circ F_B^j\I_{\mathbb A(J,\mathbf x,n)}\circ F_B^jd\p=\int_{G_n} r_B\circ F_B^j\I_{\mathbb A(J,\mathbf x,n)}\circ F_B^jd\p+\int_{B\setminus G_n} r_B\circ F_B^j\I_{\mathbb A(J,\mathbf x,n)}\circ F_B^jd\p.
\]
 By the choice of $N_2$ and $M^*$,
\[
0\leq \int_{B\setminus G_n} r_B\circ F_B^j\I_{\mathbb A(J,\mathbf x,n)}\circ F_B^jd\p\leq \int_{B\setminus G_n} r_B\circ F_B^j   d\p<\eps_0/2.
\]
The last inequality follows from (\ref{eq:eps2}) since $\p(B\setminus G_n)<\eps_2$. Thus,
\begin{align*}
\int_{B} r_B\circ F_B^j\I_{\mathbb A(J,\mathbf x,n)}\circ F_B^jd\p&\leq \int_{G_n} r_B\circ F_B^j\I_{\mathbb A(J,\mathbf x,n)}\circ F_B^jd\p+\frac{\eps_0}2
%&\leq \int_{G_n} r_B\circ F_B^j\I_{A_n(J^{\eps_1+})>x}d\p+\eps_0/2\\
\leq \int_{B} r_B\circ F_B^j\I_{\mathbb A(J^{\eps_1+},\mathbf x,n)}d\p+\frac{\eps_0}2.
\end{align*}

The first inequality follows from (\ref{eq:cont1}).

On the other hand, we have that
\begin{align*}
\int_{B} &r_B\circ F_B^j\I_{\mathbb A(J,\mathbf x,n)}\circ F_B^jd\p\geq \int_{G_n} r_B\circ F_B^j\I_{\mathbb A(J,\mathbf x,n)}\circ F_B^jd\p\\
%&\geq \int_{G_n} r_B\circ F_B^j\I_{A_n(J^{\eps_1-})>x}d\p\\
&\geq \int_{B} r_B\circ F_B^j\I_{\mathbb A(J^{\eps_1-},\mathbf x,n)}d\p-\int_{B\setminus G_n}r_B\circ F_B^j\I_{\mathbb A(J^{\eps_1-},\mathbf x,n)}d\p
\geq \int_{B} r_B\circ F_B^j\I_{\mathbb A(J^{\eps_1-},\mathbf x,n)}d\p-\eps_0/2.
\end{align*}

The second inequality above follows from (\ref{eq:cont2}) and the last inequality follows from (\ref{eq:eps2}).
Hence,
\[
\int_{B} r_B\circ F_B^j\I_{\mathbb A(J^{\eps_1-},\mathbf x,n)}d\p-\eps_0/2\leq \int_{B} r_B\circ F_B^j\I_{\mathbb A(J,\mathbf x,n)}\circ F_B^jd\p\leq \int_{B} r_B\circ F_B^j\I_{\mathbb A(J^{\eps_1+},\mathbf x,n)}d\p+\eps_0/2.
\]

By ergodicity of $F_B$, by the Ergodic Theorem and Kac's Theorem we obtain that, if $M^*$ is sufficiently large, then 
\[
\int_B\left|\frac{1}{M}\sum_{j=0}^{M-1}r_B\circ F_B^j-\p(B)^{-1}\right|d\p<\eps_0/2.
\]
Consequently,
\[
\int_{B} \p(B)^{-1}\I_{\mathbb A(J^{\eps_1-},\mathbf x,n)}d\p-\eps_0 \leq \frac{1}{M} \sum_{j=0}^{M-1} \int_B r_B\circ F_B^j\I_{\mathbb A(J,\mathbf x,n)}\circ F_B^jd\p \leq \int_{B} \p(B)^{-1}\I_{\mathbb A(J^{\eps_1+},\mathbf x,n)}d\p+\eps_0,
\]
and the result follows.
\end{proof}

Finally, the last lemma allows to approximate $\p_B(\mathbb A(J,\mathbf x,n))$ by $\p_B(\mathbb A^B(J,\mathbf x,n))$, where $\mathbb A^B(J,\mathbf x,n)$ is defined as $\mathbb A(J,\mathbf x,n)$ by replacing the role of $A_n$ by that of $A^B_n$.
\begin{lemma}
\label{lem:muB}
For any small $\eps_0,\eps_1, \eps'_1>0$ and $n$ sufficiently large we have 
$$
\p_B(\mathbb A^B(J^{\eps_1-},\mathbf x,n))-\eps_0\leq\p_B(\mathbb A(J,\mathbf x,n)) \leq \p_B(\mathbb A^B(J^{\eps'_1+},\mathbf x,n))+\eps_0 
$$
\end{lemma}
\begin{proof}

We recall that
$$A_n(I_l)(x)=\sum_{j=0}^{N(I_l)(x,u)} m_u(\{X_i\}_{i\in v_n(I_l)_j(x,u)\cap\N_0})$$
and that $$A_n^B(I_l)(x)=\sum_{j=0}^{N(I_l)(x,u)} m_u(\{X_i^B\}_{i\in v_n^B(I_l)_j(x,u)\cap\N_0})$$ where $X_j^B=\varphi\circ F_B^j$ and $v_n^B=\frac{1}{\p_B(U_n)}=\frac{\p(B)}{\p(U_n)}$. %, for every $J\in\RR$, with $J=J_1\cup\ldots\cup J_k$ and $J_i\in\S,\forall l\in\{1,\ldots,k\}$.
Note that $r_B^j(x)=\sum_{i=0}^{j-1}r_B\circ F_B^i(x)$.

By the Ergodic Theorem and Kac's Theorem, we have that $|\frac{1}{j}\sum_{i=0}^{j-1}r_B\circ F_B^i(x)-\frac{1}{\p(B)}|\rightarrow 0$ $\p_B$-a.e. because $F_B$ is ergodic with respect to $\p_B$ and $\int_B r_B d\p_B=\frac{1}{\p(B)}$.

Observe that
\[
\left|\frac{1}{j}\sum_{i=0}^{j-1}r_B\circ F_B^i(x)-\frac{1}{\p(B)}\right|<\delta\Leftrightarrow \left(\frac{1}{\p(B)}-\delta\right)j<r_B^j(x)< \left(\frac{1}{\p(B)}+\delta\right)j.
\]

Define now the following subsets of $B$:
\[
E_M^{\eps_3}:=\left\{x\in B: \left(\frac{1}{\p(B)}-\eps_3\right)j\leq \sum_{i=0}^{j-1}r_B\circ F_B^i(x)\leq \left(\frac{1}{\p(B)}+\eps_3\right)j, \;\;\;\forall j\geq M\right\}.
\]

Note that $\p_B(B\setminus E_M^{\eps_3})\xrightarrow[M\to\infty]{}0$.
Let $F_M=\left\{r_{U_n,B}\geq M\right\}$. We have that $B\setminus F_M=B\cap \left(F_B^{-1}U_n\cup\ldots\cup F_B^{-(M-1)}U_n \right)$ and so $\p_B(B\setminus F_M)\leq M \p_B(U_n)\xrightarrow[n\to\infty]{}0$.

Let $M$ be sufficiently large such that  $\p_B(B\setminus E_M^{\eps_3})<\eps_0/2$ and $N_4$ sufficiently large such that $\forall n>N_4$, $\p_B(B\setminus F_M)<\eps_0/2$.

We have that $F_M\subset \left\{r^i_{U_n,B}\geq M\right\}, \forall i\in \N$, since $r^i_{U_n,B}\geq r_{U_n,B}$. Moreover, if $x\in E_M^{\eps_3}\cap F_M$, then 
\[
\left(\frac{1}{\p(B)}-\eps_3\right) r^i_{U_n,B}(x)\leq r^i_{U_n}(x)= \sum_{i=0}^{r^i_{U_n,B}(x)-1}r_B\circ F_B^i(x)=r_B^{r^i_{U_n,B}(x)}(x)\leq \left(\frac{1}{\p(B)}+\eps_3\right) r^i_{U_n,B}(x).
\]

So, we may write
\begin{equation}
\label{eq:r^i} r^i_{U_n}(x)=(1+\alpha)\p(B)^{-1}r^i_{U_n,B}(x),
\end{equation}
where $|\alpha|<\eps_3\p(B)$.

Consequently,
\begin{align*}
r^i_{U_n}(x)\in v_n J&\Leftrightarrow v_n^{-1}r^i_{U_n}(x)\in J
%&\Leftrightarrow (1+\alpha)\p(B)^{-1} r^i_{U_n,B}(x)v_n^{-1}\in J\\
\Leftrightarrow (1+\alpha)(v_n^B)^{-1} r^i_{U_n,B}(x)\in J \\
&\Leftrightarrow (v_n^B)^{-1} r^i_{U_n,B}(x)\in (1+\alpha)^{-1}J 
\Rightarrow (v_n^B)^{-1} r^i_{U_n,B}(x)\in J^{\frac{\eps_3}{1-\eps_3}J_{\sup}+},
\end{align*}
where $J_{\sup}=\sup J.$

On the other hand, using again (\ref{eq:r^i}),
\begin{align*}
r^i_{U_n,B}(x)\in  v_n^B J^{\eps_3 J_{\sup}-}
%&\Leftrightarrow(v_n^B)^{-1} r^i_{U_n,B}(x)\in J^{\eps_3 J_{\sup}-}\\
&\Leftrightarrow\p(B)^{-1} r^i_{U_n,B}(x)v_n^{-1}\in J^{\eps_3 J_{\sup}-}
%&\Leftrightarrow  (1+\alpha)\p(B)^{-1} r^i_{U_n,B}(x)v_n^{-1} (1+\alpha)^{-1}\in J^{\eps_3 J_{\sup}-}\\
\Leftrightarrow  r^i_{U_n}(x)v_n^{-1} (1+\alpha)^{-1}\in J^{\eps_3 J_{\sup}-}\\
&\Leftrightarrow  r^i_{U_n}(x)\in v_n (1+\alpha)J^{\eps_3 J_{\sup}-}
\Rightarrow r^i_{U_n}(x)\in v_n J.
\end{align*}

We then have
\begin{align}
\label{eq:eps3q}
\mathbb A(J,\mathbf x,n)\cap F_M\cap E_M^{\eps_3} \subset \mathbb A^B(J^{\frac{\eps_3}{1-\eps_3}J_{\sup}+},\mathbf x,n)\cap F_M\cap E_M^{\eps_3}
\end{align}
and 
\begin{align}
\label{eq:eps3}
\mathbb A^B(J^{\eps_3 J_{\sup}-},\mathbf x,n)\cap F_M\cap E_M^{\eps_3}  \subset \mathbb A(J,\mathbf x,n)\cap F_M\cap E_M^{\eps_3}.
\end{align}

By (\ref{eq:eps3q}), we obtain
\begin{align*}
\p_B(\mathbb A(J,\mathbf x,n))&\leq \p_B(\mathbb A(J,\mathbf x,n)\cap F_M\cap E_M^{\eps_3})+\p_B(B\setminus F_M)+\p_B(B\setminus E_M^{\eps_3})\\
&\leq \p_B(\mathbb A(J,\mathbf x,n)\cap F_M\cap E_M^{\eps_3})+\eps_0%\\
%&\leq \p_B( \{A^B_n(J^{\frac{\eps_3}{1-\eps_3}J_{\sup}+})>x\}\cap F_M\cap E_M^{\eps_3})+\eps_0\\
%&
\leq \p_B(  \mathbb A^B(J^{\eps'_1+},\mathbf x,n) )+\eps_0,
\end{align*}
where $\eps'_1=\frac{\eps_3}{1-\eps_3}J_{\sup}.$
By (\ref{eq:eps3}), we obtain
\begin{align*}
\p_B&(\mathbb A(J,\mathbf x,n))\geq \p_B(\{\mathbb A(J,\mathbf x,n)\}\cap F_M\cap E_M^{\eps_3})\\
%&\geq \p_B( \{A^B_n(J^{\eps_3J_{\sup}-})>x\}\cap F_M\cap E_M^{\eps_3})\\
&\geq \p_B( \mathbb A^B(J^{\eps_3J_{\sup}-},\mathbf x,n) )-\p_B(B\setminus F_M)-\p_B(B\setminus E_M^{\eps_3})
\geq \p_B( \mathbb A^B(J^{\eps_1-},\mathbf x,n))-\eps_0,
\end{align*}
where $\eps_1=\eps_3J_{\sup}$, concluding in this way the proof.
\end{proof}

\bibliographystyle{abbrv}
\bibliography{AOT-POT}

{\fontsize{10}{10}\selectfont}

\end{document}